\documentclass[12pt,fleqn]{article}
\usepackage{latexsym}
\usepackage{amsmath}
\usepackage{amssymb}
\usepackage{epsfig}
\usepackage{amsfonts}
\usepackage{hyperref}

\begin{document}
\newtheorem{definition}{Definition}[section]
\newtheorem{theorem}[definition]{Theorem}
\newtheorem{lemma}[definition]{Lemma}
\newtheorem{proposition}[definition]{Proposition}
\newtheorem{examples}[definition]{Examples}
\newtheorem{corollary}[definition]{Corollary}
\def\square{\Box}
\newtheorem{remark}[definition]{Remark}
\newtheorem{remarks}[definition]{Remarks}
\newtheorem{exercise}[definition]{Exercise}
\newtheorem{example}[definition]{Example}
\newtheorem{observation}[definition]{Observation}
\newtheorem{observations}[definition]{Observations}
\newtheorem{algorithm}[definition]{Algorithm}
\newtheorem{criterion}[definition]{Criterion}
\newtheorem{algcrit}[definition]{Algorithm and criterion}
\newtheorem{lists}[definition]{Lists}
\newtheorem{Table}[definition]{Table}

\newenvironment{prf}[1]{\trivlist
\item[\hskip \labelsep{\it
#1.\hspace*{.3em}}]}{~\hspace{\fill}~$\square$\endtrivlist}
\newenvironment{proof}{\begin{prf}{Proof}}{\end{prf}}

\title{ Mumford curves and Mumford groups in positive characteristic }

\author{Harm H.~Voskuil and Marius van der Put\\
{\small  Bernoulli Institute, University of Groningen, P.O.~Box 407,} \\ 
{\small 9700 AG Groningen, the Netherlands}\\
{\small m.van.der.put@rug.nl}\;\;\;\;{\small harm.voskuil@planet.nl}}


\date{\today}




\maketitle

\begin{abstract} A Mumford group is a discontinuous subgroup $\Gamma$
 of ${\rm PGL}_2(K)$, where $K$ denotes a non archimedean 
valued field, such that the quotient  by $\Gamma$ is a curve of genus 0.  As abstract group $\Gamma$ is an amalgam of a  finite tree of finite groups. For $K$ of
positive characteristic the large collection of amalgams having two or three branch points
is classified. Using these data Mumford curves with a large group of
automorphisms are discovered. A long combinatorial proof, involving
the classification of the finite simple groups, is needed for
establishing an upper bound for the order of the group of
automorphisms of a Mumford curve.  Orbifolds in the category of rigid spaces are
introduced. For the projective line the relations with Mumford groups
and singular stratified bundles are studied. This paper is a sequel to
\cite{P-V}. Part of it clarifies, corrects  and extends  work  of
G.~Cornelissen, F.~Kato and K.~Kontogeorgis. \begin{footnote}{MSC2010,
  14E09, 20E08, 30G06. Keywords: Rigid geometry, Discontinuous groups, Mumford curves,
  Mumford groups, Amalgams, Orbifolds, Stratified bundles} \end{footnote}
 \end{abstract}


\section*{ \rm Introduction}
Let $K$ be a complete non archimedean valued field. For convenience we
will suppose that $K$ is algebraically closed.  A {\it Schottky group}
$\Delta$ is a finitely generated, discontinuous subgroup of ${\rm PGL}_2(K)$ such
that $\Delta$ contains no elements ($\neq 1$) of finite order and
$\Delta\not \cong \{1\}, \mathbb{Z}$.
It turns out that $\Delta$ is a free non-abelian group on $g>1$
generators. Let $\Omega\subset \mathbb{P}^1_K$ denote the rigid open subspace of
ordinary points for $\Delta$. Then $X:=\Omega /\Delta$ is an algebraic curve over $K$ of genus $g$. 
The curves obtained in this way are called {\it Mumford curves}. Let
$\Gamma \subset {\rm PGL}_2(K)$ denote the normalizer of $\Delta$. Then
$\Gamma /\Delta$ acts on $X$ and is in fact the group of the automorphisms
of $X$.   For $K\supset \mathbb{Q}_p$, the theme of automorphisms of Mumford
 curves is of interest for $p$-adic orbifolds and for $p$-adic hypergeometric differential equations. 
 According to F.~Herrlich \cite{He} one has for Mumford curves $X$ of genus $g>1$  the bound
$|{\rm Aut}(X)|\leq 12(g-1)$ if $p>5$. For $p=2,3,5$ there are $p$-adic ``triangle groups'' and the
bounds are $n_p(g-1)$ with $n_2=48,\ n_3=24,\ n_5=30$.\\

 {\it In this paper we investigate the case that $K$ has characteristic $p>0$}. \\
 The order of the automorphism group can be
much larger than $12(g-1)$. Using the Riemann--Hurwitz--Zeuthen
formula one easily shows (see also the proof of Corollary \ref{6.2}):\\

\noindent {\it If $g>1$ and $|{\rm Aut}(X)|> 12(g-1)$, then $X/{\rm Aut}(X)\cong \mathbb{P}^1_K$ and the morphism $X\rightarrow
X/{\rm Aut}(X)$ is branched above {\rm 2} or {\rm 3} points}.\\

There exist Mumford curves $X=\Omega /\Delta$ 
 with  genus $g>1$ and such that   $|{\rm Aut}(X)|> 12(g-1)$. Hence
the normalizer $\Gamma$ of $\Delta \subset {\rm PGL}(2,K)$ satisfies
$\Omega /\Gamma\cong \mathbb{P}^1_K$. This leads to the definition of
a {\it Mumford group}:\\
 This is a finitely generated, discontinuous subgroup $\Gamma$ of
${\rm PGL}_2(K)$ such that $\Omega /\Gamma \cong \mathbb{P}_K^1$,
where $\Omega \subset \mathbb{P}^1_K$ is the rigid open subset of the
ordinary points for the group $\Gamma$. We {\it exclude} the possibilities
that $\Gamma$ is finite and that $\Gamma$ contains a subgroup of
finite index, isomorphic to $\mathbb{Z}$.  
 A point $a\in \mathbb{P}^1_K$ is called 
a {\it branch point}  if a preimage $b\in \Omega$ of $a$ has a non
trivial stabilizer in $\Gamma$. \\

On the other hand, a Mumford group $\Gamma$ contains a normal subgroup $\Delta$, which is
of finite index and has no elements $\neq 1$ of finite order. Thus
$\Delta$ is a Schottky group, $X:=\Omega /\Delta$ is a Mumford
curve. Above we have excluded the cases that the genus of $X$ is 0 or 1.
The group $A:=\Gamma/\Delta$ is a subgroup of ${\rm Aut}(X)$ such that
$X/A\cong \mathbb{P}^1_K$. \\ 

In several papers \cite{C-K-K, C-K , C-K 2, C-K 3, C-K 4, P-V, P-V2003}  the construction and the
classification of Mumford groups over a field $K$ of characteristic $p>0$ are studied. 
Here we continue this study. First we recall that a Mumford group is,
as an abstract group, a finite tree of finite groups $(T,G)$. In the work of
F.~Herrlich \cite{He} and in \cite{P-V} a criterion is proved which decides
whether the `amalgam' $\pi_1(T,G)$ of a finite tree of finite
groups $(T,G)$ is {\it realizable}, i.e., $\pi_1(T,G)$ has an embedding in
${\rm PGL}(2,K)$ as discontinuous group. If there is a realization, then, in
general, there are some
families of realizations. Thus in classifying Mumford groups we
classify in fact the realizable finite trees of finite groups $(T,G)$.
 Still, for the purpose of classification, there are too many 
 realizable $(T,G)$.\\ 

 Since we are interested in
Mumford curves $X$ with $|{\rm Aut}(X)|> 12(g-1)$ and $g>1$, we only 
consider trees of groups $(T,G)$ which produce $2$ or $3$ branch
points. The number of branch points ${\rm br}$ depends only on $(T,G)$ and not on the chosen realization.
A formula for ${\rm br}$, proved in \cite{P-V}, answers in principle
the question of classifying realizable $(T,G)$ with ${\rm br}$=2 or 3.

 However a delicate (especially for $p=2,3$)
combinatorial computation in \S\S 1-2 combined with \cite{P-V} is
needed to obtain the complete lists (\S\S 3-4). For completeness, a not
well known family of realizable amalgams is studied in \S 5. The
amalgams of \S 5 do
not produce Mumford curves with large automorphism groups. The lists
in \S\S 3-4 correct, clarify and extend the data of \cite{C-K-K}.\\
    
In order to find bounds for $|{\rm Aut}(X)|$ in terms of the genus $g$ of
Mumford curves we have to investigate the lists \S\S 3-4 of the realizable $(T,G)$ with two or
three branch points. 
For any normal Schottky group $\Delta \subset \Gamma$ of finite index
and free on $g>1$ generators one has $g-1=\mu(\Gamma) [\Gamma:\Delta]
$ for a rational number $\mu(\Gamma)$ which has a formula in terms of
the data of the tree of groups $(T,G)$. We are only interested in the
case $\mu(\Gamma)<\frac{1}{12}$. This produces lists \ref{6.3} shorter than 
 those of \S\S 3-4. \\ 

Each group $\Gamma$ with $\mu(\Gamma)<\frac{1}{12}$ contains a normal Schottky group $\Delta$
such that the Mumford curve $X=\Omega/\Delta$ has automorphism group
${\rm PGL}_2(\mathbb{F}_q)$ or ${\rm PSL}_2(\mathbb{F}_q)$.
The value of $q$ is uniquely determined by the tree of groups $(T,G)$.
We then search for the lowest genus $g$ such that there exist Mumford curves $X$ of genus $g$
with automorphism group ${\rm Aut}(X)= {\rm PGL}_2(\mathbb{F}_q)$. 
This leads to Theorem \ref{6.8}, the {\it discovery}
of two families of Mumford curves having many automorphisms namely:\\

\noindent {\it
The amalgams ${\rm PGL}_2(\mathbb{F}_q)\ast_{C_{q+1}}D_{q+1}$ and
 $D_{q-1}\ast_{C_{q-1}}B(n,q-1)$ with $q=p^n>2$.}
Here $C_*,\ D_*$ denote cyclic, dihedral groups and
$B(n,q-1)=\{ {a\ b\choose 0\ 1}|\ a\in \mathbb{F}_q^*,\ b\in
\mathbb{F}_q\}$. {\it The above amalgams $\Gamma$ contain a unique normal
Schottky subgroup such that the corresponding Mumford curve $X$ has
the properties ${\rm Aut}(X)={\rm PGL}_2(\mathbb{F}_q)$ and $X$ has genus $\frac{q(q-1)}{2}$.
In \S {\rm \ref{section 7}} equations for these curves are derived }(Corollary
\ref{7.6}, \S \ref{section 7.1.1.0} \S \ref{section 7.1.1}, \S \ref{section 7.2}). \\

Based on these families we claim the following bound (Theorem \ref{6.18}):\\

\noindent {\it For a Mumford curve $X$ with genus $g$ one has
\[ |{\rm Aut}(X)|\leq \max\{ 12(g-1),g(\sqrt{8g+1}+3)\} \]
 with three exceptions for $p=3$ and $g=6$, see Proposition  {\rm \ref{6.22}}.\\
 Equality holds precisely for the cases of  Theorem {\rm\ref{6.8}}
 and the few  cases of Mumford curves with $g\in \{3,4,5,6\}$ and the four amalgams 
$\Gamma$ with $\mu(\Gamma)=\frac{1}{12}$ studied in Proposition {\rm \ref{6.14}}
and Lemma {\rm \ref{6.19}}}.\\

The Hurwitz bound $|{\rm Aut}(X)|\leq 84(g-1)$ does not hold for curves in
positive characteristic. Curves in positive characteristic can have
many more automorphisms.
However, for ordinary curves, i.e., curves such that the $p$-rank of the Jacobian equals the genus $g$,
the number of automorphisms is somewhat restricted compared to curves of lower $p$-rank.
Several results concerning the upper bound of the number of automorphisms of ordinary curves
have been established (see \cite{N,K-M}). {\it Mumford curves are ordinary}.

A sharp upper bound that is actually attained by some infinite family of ordinary curves is as yet unknown.
In fact, as far as the authors know the two families of Mumford curves described in \S \ref{section 7}
have more automorphisms than any other currently known infinite family of ordinary curves if $p>2$.
For $p=2$ a family of smooth plane curves of genus $\frac{q^2-q}{2}$, $q=2^n\geq 4$ 
that are ordinary and have automorphism group 
${\rm PGL}_2(\mathbb{F}_q)$ is known (see \cite{Fu} Theorem 1). 
This known family of ordinary curves contains at least one of the families of Mumford curves
(for the case  $p=2$) as described in Theorem \ref{6.8} (see \S \ref{section 7.2.1}). \\

Now we fix  $\Gamma$ with $\mu(\Gamma)<\frac{1}{12}$. 
To prove the claim (Theorem \ref{6.18}), we have to
search for normal Schottky subgroups $\Delta \subset \Gamma$ of
minimal index $[\Gamma : \Delta]$. Indeed, if the inequality holds for
a normal Schottky subgroup $\Delta$ of minimal index, then it holds for all normal Schottky
subgroups of finite index, contained in $\Delta$.

 It seems impossible, in general, to compute these groups $\Delta$ and the minimal
 values for $[\Gamma :\Delta]$. However determining, what we call a
{\it suitable bound}: $[\Gamma:\Delta]\geq N_0(\Gamma)$ for all $\Delta$, is enough for
proving the claim. Obtaining these lower bounds is a hard problem in
combinatorial group theory. The long rather complicated, delicate
section \S \ref{section 8} is needed for a solution. For the most complicated case,
the amalgam $\Gamma={\rm PGL}_2(\mathbb{F}_q)\ast_{B(n,q-1)}B(2\cdot n,q-1)$,
the long computation of $N_0(\Gamma)$ in \S \ref{section 8.3}     
finally involves the classification of finite simple groups. 
Many of our results {\it differ from those of } \cite{C-K-K} (see Remarks \ref{CKK-1} and \ref{CKK-2}).\\

Linear differential equations on the complex line with prescribed
regular singularities, e.g., hypergeometric differential equations,
can be induced by complex-linear representations of discrete subgroups
of ${\rm PGL}(2,\mathbb{C})$. In \S \ref{section 9} we study an analogue of this namely, the way Mumford groups induce
stratified bundles on $\mathbb{P}^1_K$. Moreover the  relation with
orbifolds on $\mathbb{P}^1_K$ is explained in \S \ref{section 9}.\\

In this paper, $K$ is supposed to have characteristic $p>0$, to be
complete with respect to a non trivial valuation and to be
algebraically closed. Moreover we will use terminology and results
from \cite{P-V}. 

\section{ \rm The finite subgroups of ${\rm PGL}_2(K)$}
For a finite subgroup $G\neq 1$ of ${\rm PGL}_2(K)$ the morphism
$m: \mathbb{P}^1\rightarrow \mathbb{P}^1/G\cong \mathbb{P}^1$ is
ramified (or branched) above at most three points. A {\it branch
  group} is a subgroup $H\neq \{1\}$
of $G$ such that $H$ is the stabilizer $H_a$ of a point $a$. 

If G is not a cyclic group, then $H_a$ is conjugated to $H_b$ if and
only if :\\ 
(1). $m(a) = m(b)$, or\\
(2). the pair $(H_a,G)=(C_2,D_\ell)$ with $p\not=2$, $p\not|\ell$,
$\ell$ odd, or \\
(3). $(H_a,G)=(C_3,A_4)$ with $p\not= 2,3$.\\
In the latter two cases the cyclic group $H_a=H_b$ is the stabilizer of two distinct points $a$ and $b$,
such that $m(a)\not=m(b)$. These cases give rise to the amalgams treated in \S 5. \\

\noindent {\bf The classification of the finite subgroups, up to
  conjugation, is: }
\begin{itemize}
\item[(a).] {\it Borel type} $B(n,m)$ represented by $\{ {a\ b\choose 0\ 1 }|\
  a^m=1, b\in B  \}$ where $p\nmid m$, $B\subset K$ is a vector space
  over $\mathbb{F}_p$ of dimension $n$, and $aB=B$ for every $a$
with $a^m=1$. If $B\neq 0$, then it follows that $m| p^n-1$.  \begin{enumerate}
\item $B(0,m)$ is the cyclic group $C_m$. Two branch points.
\item $B(n,1)$ with $n>0$ is called a $p$-group (i.e., isomorphic to
  $C_p^n$ ). One branch point.
\item $B(n,m)$ with $n>0,m>1$. Two branch points. The branch groups are
$C_m$ and $B(n,m)$.
\end{enumerate} 

\item[(b).] {\it not Borel type, $p\mid \#G$ and two branch points}.
\begin{enumerate}
\item ${\rm PGL}_2(\mathbb{F}_q)$ with $q=p^n$. One branch group is 
$\{{a\ b \choose 0\ 1 }| a\in \mathbb{F}_q^*, b\in \mathbb{F}_q \}$
which is $B(\mathbb{F}_q)=B(n,q-1)$, where $B(\mathbb{F}_q)$ denotes the Borel group.
 The other branch group is $C_{q+1}$.
\item ${\rm PSL}_2(\mathbb{F}_q)$ with $q=p^n$ and $p\neq 2$ (for
  $p=2$ one has ${\rm PSL}_2(\mathbb{F}_q)={\rm PGL}_2(\mathbb{F}_q)$). One branch group is 
$\{{a\ b \choose 0\ a^{-1} }| a\in \mathbb{F}_q^*, b\in \mathbb{F}_q
\}$ (modulo $\pm 1$) which is $B(n,(q-1)/2)$.
 The other branch group is $C_{(q+1)/2}$.
\item $p=2$ and $D_\ell$ with odd $\ell$. The branch groups are $C_2$
  and $C_\ell$. 
\item $p=3$ and $A_5\subset A_6\cong {\rm PSL}_2(\mathbb{F}_9)$.
Branch groups $C_5$ and $B(1,2)\cong S_3$. Note that ${\rm PSL}_2(\mathbb{F}_3)\cong A_4\subset A_5$. 
\end{enumerate} 

\item[(c).]  $p\nmid \# G$,  
{\it  $G\in \{D_n,A_4,S_4,A_5\}$ and three branch
  points}. Cyclic branch groups with orders $(2,2,n), (2,3,3), (2,3,4),
  (2,3,5)$, respectively. 
\end{itemize}

\begin{remarks}\label{1.1}  {\rm  One can show the following. Suppose that the finite
 group $G$  has an embedding in ${\rm PGL}_2(K)$. Then the number of
 branch points and the branch subgroups of $G$ do not depend on the
 choice of the embedding. More in detail, the following holds.

 In general the image of an embedding of a finite group $G$ into
 the group ${\rm PGL}_2(K)$ is unique up to conjugation by an element
 $g\in {\rm PGL}_2(K)$ (See
\cite{B} prop. 4.1 and \cite{Fa} theorem 6.1).
The exceptions are certain groups of Borel type. The situation can be
explained as follows. 

 As an abstract group, a group $G$ of Borel type is the semi-direct product of a cyclic group $C_m$
(with given generator and $p\nmid m$) and an elementary $p$-group
$B$ (i.e., $\cong (\mathbb{Z}/p\mathbb{Z})^n$ for some $n$). Let $\mathbb{F}_q\subset K$ denote the smallest extension of $\mathbb{F}_p$ such that the $m$th
roots of unity belong to $\mathbb{F}_q$. We identify $C_m$ with
$\mu_m(K)$. The action (by conjugation) of $C_m$ on $B$ makes $B$ into
a finite dimensional vector space over $\mathbb{F}_q$. For any embedding
$\phi$ of $G$ into ${\rm PGL}_2(K)$ one has that $\phi(B)$ has a
unique fixed point for its action on $\mathbb{P}^1_K$. We take this fixed point to be $\infty$ and so
$\phi(B)\subset {1\ K\choose 0\ 1 }$. Further, the resulting $\phi'
:B\rightarrow K$ is $\mathbb{F}_q$-linear. The map $\phi'$ is an
arbitrary $\mathbb{F}_q$-linear injection and $\phi$ is determined by
$\phi'$. Conjugation by any ${\lambda\ \mu \choose 0\ 1}$ (with
$\lambda \in K^*,\mu \in K$) changes $\phi'$ into $\lambda \cdot \phi'$.     
One concludes the following:\\
 If $\dim _{\mathbb{F}_q}B=1$, then there is up to conjugation only
 one embedding of $G$ into ${\rm PGL}_2(K)$. However, if  $\dim
 _{\mathbb{F}_q}B>1$, then the set of conjugacy classes of embeddings
 of $G$ is very large. There are embeddings $\phi$ such that
 $\phi'(B)\subset K$ lies in the algebraic closure of $\mathbb{F}_q$
 and there are embeddings $\phi$ such that $\lambda \phi'(B)$ is not
 contained in any local subfield of $K$ for any $\lambda \in K^*$.  }\end{remarks}

 \section{ \rm Realizable $G_1\ast_{G_3}G_2$}  

A finite tree of finite groups $(T,G)$ is given by a finite tree $T$, for
every vertex $v$ a finite group $\Gamma_v$ and for every edge
$e$ a finite group $\Gamma_e$. Further, if $e$ is an edge of $v$,
then an injective homomorphism $\Gamma _e\rightarrow \Gamma_v$ is
given (often regarded as an inclusion). $(T,G)$ is called {\it realizable} if its amalgam
$\Gamma$ has an embedding as a discontinuous group in ${\rm PGL}_2(K)$. 
The problem to classify the realizable $(T,G)$ has been solved in
\cite{P-V} together with a formula for the number of branch points of the
amalgam $\Gamma$ of $(T,G)$. However, due to the complexity of the
groups and trees involved, the classification of the $(T,G)$ with two or
three branch points requires new investigations. We restrict ourselves
to indecomposable trees of groups $(T,G)$, i.e., all $\Gamma_v\neq 1$
and all $\Gamma _e\neq 1$ (because of \cite{P-V}, Prop. 3.2).  According to
\cite{P-V} 3.7, a cyclic group $C_m$ with $p\nmid  m$ will not be a vertex group
$\Gamma _v$. Proposition \ref{2.1} summarizes \cite{P-V}, 3.8--3.11.

\begin{proposition}\label{2.1} $G_1\ast_{G_3}G_2$ is realizable in precisely the
  following cases:\\
{\rm (a)}. $G_1,G_2$ both not Borel type and
\begin{enumerate}
\item $p\nmid \ \#G_3$ and $G_3$ is a branch group of $G_1$ and of
  $G_2$ (and then $G_3=C_m$). 
\item In addition for $p=2$, the groups $D_\ell \ast_{C_2}D_m$ with
  odd $\ell, m$. 
\item In addition for $p=3$, the group ${\rm
    PSL}_2(\mathbb{F}_3)\ast_{C_3}{\rm PSL}_2(\mathbb{F}_3)$.
\end{enumerate}
{\rm (b)}. $G_1=B(N,m)$ with $N>0,\ m\geq 1$.
\begin{enumerate}
\item $m>1$, $G_3=C_m$ and $G_3$ is a branch group of $G_2$. 
 \item $q=p^n, \ m=q-1, \ \ell >1$ and $B(\ell \cdot n,q-1)\ast_{B(\mathbb{F}_q)}{\rm
    PGL}_2(\mathbb{F}_q)$ and $B(\mathbb{F}_q)\cong B(n,q-1)$ is a Borel group of ${\rm PGL}_2(\mathbb{F}_q)$.
\item $p\neq 2$, $q=p^n, \ m=(q-1)/2, \ \ell >1$ and $B(\ell \cdot n,(q-1)/2)\ast_{B(\mathbb{F}_q)}{\rm
    PSL}_2(\mathbb{F}_q)$ and $B(\mathbb{F}_q)\cong B(n,(q-1)/2)$ is a Borel group of ${\rm PSL}_2(\mathbb{F}_q)$.
\item In addition for $p=2$, the groups $B(N,1)\ast_{C_2}D_\ell$ with
odd $\ell$ and $N>1$
\item In addition for $p=3$, the groups $B(N,2)\ast_{D_3}A_5$ with $N>1$.
\end{enumerate}
\end{proposition}

\begin{corollary}\label{2.2} Let $G_1,G_2$ denote groups not of Borel
  type. Suppose that both groups have a cyclic branch group of order $\ell$ 
  and that the order of both groups is divisible by the prime $p$. 
  Then one of the following holds:\\
{\rm (1)}. $p\nmid \ell$ and $G_1\cong G_2$.\\
{\rm (2)}. $p=2, \ \ell=q+1$ and $\{G_1,G_2\}=\{D_{q+1}, {\rm PGL}_2(\mathbb{F}_q)  \}$.\\
{\rm (3)}. $p=2, \ \ell =2$ and $\{G_1,G_2\}=\{ D_n ,D_m\}$ with odd
$m,n$.\\
{\rm (4)}. $p=3, \ \ell =5$ and $\{G_1,G_2\}=\{A_5 , {\rm
  PSL}_2(\mathbb{F}_9)\}$.\\
{\rm (5)}. $p=3,\ \ell =3$ and $G_1\cong G_2 \cong {\rm PSL}_2(\mathbb{F}_3)$.
\end{corollary}

The proof follows from \S 1, the parts (b) and (c). 

\begin{remarks}\label{2.3} {\rm 
(1). Every amalgam $\Gamma:=G_1\ast_{G_3}G_2$ has infinitely many embeddings into
${\rm PGL}_2(K)$ that are not conjugated by elements of ${\rm PGL}_2(K)$. 
In fact, each amalgam $\Gamma$ has an embedding into ${\rm PGL}_2(K)$ 
such that no conjugate of the embedding  by an element of ${\rm PGL}_2(K)$
can be defined over a local field inside $K$.\\
\noindent (2). In case ${\rm br}(\Gamma)=2,3$, then one can normalize (by conjugation)
any embedding $\Gamma \rightarrow {\rm PGL}_2(K)$ such that the branch points are $0,\infty$
or $0,1,\infty$. It does not follow that the set of ramification points
for $\Gamma$ and/or the set $\Omega$ of ordinary points for $\Gamma$ are defined
over a local field inside $K$.\\
\noindent (3). It can be shown that any realizable amalgam has an
embedding which is defined over a local field inside $K$. }\end{remarks}

\section{\rm Mumford groups with two branch points}

The amalgams $B(n_1,1)\ast B(n_2,1)$ with $n_1,n_2>0$ are the only decomposable
Mumford groups with two branch points. {\it Below we consider only
indecomposable Mumford groups}.\\

We recall from \cite{P-V} that one associates to an indecomposable $\Gamma$, a finite tree of groups
$T^c$. For a vertex $v$ and an edge $e$ one writes $\Gamma_v$ and
$\Gamma_e$ for the corresponding groups. The group $\Gamma$ is the
amalgam of $T^c$. We explain now the ingredients in the formulas for
the number of branch points of $\Gamma$.

The symbol ${\rm br}()$ denotes the number of branch points. ${\rm Max}(j)$ for
$j=2,3$ denotes the set of vertices $v$ with ${\rm br}(\Gamma_v)=j$ and  
${\rm max}(j)=\#{\rm Max}(j)$. The set ${\rm Maxp}$ is only defined for $p=2,3$. This
is due to the special groups and trees $T^c$ occurring for $p=2,3$. ${\rm Maxp}$
consists of the vertices $v$ such that $\Gamma_v$ is a $p$-group
$B(n_0,1)=C_p^{n_0}$ with $n_0\geq 1$ and moreover:\\
(i). there is given a $p$-cyclic subgroup $A\subset \Gamma _v $,\\  
(ii).  there are at least two edges $e=\{v',v\}$,\\
(iii). for every edge $e=\{v',v\}$ the group  $\Gamma_e\cong C_p$ is
identified with $A\subset \Gamma _v$.\\
  Let $d_v\geq 2$ denote the number of the edges of $v\in {\rm Maxp}$.
Define now  ${\rm maxp}:=\sum _{v\in{\rm Maxp}}(d_v-1)$. For $p>3$ one puts ${\rm maxp}=0$.
 According to \cite{P-V}, Thm. 5.3  one has:
\begin{theorem} \label{3.1} Let $\Gamma$ be an indecomposable Mumford group.
 Then:\\
{\rm (1)}. ${\rm br}(\Gamma)=\sum _{v\ \mbox{\tiny vertex of }T^c}{\rm br}(\Gamma_v) -
\sum_{e\ \mbox{\tiny  edge of }T^c}{\rm br }(\Gamma_e)$.\\
{\rm (2)}. ${\rm br}(\Gamma)={\rm max}(3)+{\rm maxp}+2$. If $p\neq 2,3$, then ${\rm maxp}=0$.
\end{theorem}

\noindent The following arguments lead to the list of (amalgams) trees of
groups $T^c$:\\
 {\rm (i)}. ${\rm max}(3)=0$ implies that  $p| \#\Gamma_v$ for every vertex $v$
 and \S 1 (a) and (b) produce all possibilities.\\
{\rm (ii)}. ${\rm maxp}=0$ means that for $p=2$ the amalgam $D_n\ast_{C_2}C_2^{n_0}\ast_{C_2}D_m$ with $n_0\geq 1$ 
and odd $m,n$ is not present in the tree of groups and that for $p=3$ the amalgam
${\rm PSL}_2(\mathbb{F}_3)\ast _{C_3}C_3^{n_0}\ast_{C_3}{\rm PSL}_2(\mathbb{F}_3)$ with $n_0\geq 1$ is not
present in $T^c$. See Remark \ref{4.2} for more explication. We note that 
these two amalgams have in fact three branch points. \\
{\rm (iii)}. By deleting end part(s) of a $T^c$ in the list one obtains
other trees in the list. Therefore we only write down the maximal trees of
groups in the list.\\
{\rm (iv)}. Part (a) (1) of Prop 2.1 is made explicit by using
Corollary 2.2.\\
{\rm (v)}. Prop. 5.8 and Corollary 5.6 of \cite{P-V}.\\

\begin{proposition}[The Mumford groups $\Gamma$ with two branch points]\label{3.2}
Write $q=p^n$. Then $\Gamma$ is one of the following amalgams:
\begin{small}
\begin{enumerate}
\item[\rm (i).] $B(2n\cdot n_1,q+1)\ast_{C_{q+1}}{\rm PGL}_2(\mathbb{F}_q)\ast_{B(n,q-1)}B(n\cdot n_2,q-1)$ 
          with $n_1\geq 1, n_2\geq 2$. 
\item[\rm (ii).] $B(n\cdot n_1,q-1)\ast_{B(n,q-1)}{\rm PGL}_2(\mathbb{F}_q)
           \ast_{C_{q+1}}{\rm PGL}_2(\mathbb{F}_q)\ast_{B(n,q-1)}B(n\cdot n_2,q-1)$
           with  $n_1,n_2\geq 2$.
\item[\rm (iii).] $p\neq 2$, $B(s\cdot n\cdot
  n_1,(q+1)/2)\ast_{C_{(q+1)/2}}{\rm PSL}_2(\mathbb{F}_q)
            \ast_{B(n,(q-1)/2)}B(n\cdot n_2,(q-1)/2)$
            with $n_1> 0$, $n_2\geq 2$.
            Here $s=2$ if $q>3$ and $s=n=1$ if $p=q=3$.
\item[\rm (iv).]  $p\neq 2$, $B(n\cdot
  n_1,(q-1)/2)\ast_{B(n,(q-1)/2)}{\rm PSL}_2(\mathbb{F}_q)
            \ast_{C_{(q+1)/2}}{\rm PSL}_2(\mathbb{F}_q)\ast_{B(n,(q-1)/2)}B(n\cdot n_2,(q-1)/2)$
            with $n_1,n_2\geq 2$.
\item[\rm (v).] $B(n_1,m)\ast _{C_m} B(n_2,m)$ with $n_1,n_2,m\geq 1$. For $m=1$
  these are all the decomposable groups with two branch points.
\item[\rm (vi).] $p=3$, $B(2\cdot n_1,5)\ast_{C_5}A_5\ast_{B(1,2)}B(n_2,2)$
          with $n_1\geq 1$, $n_2\geq 2$.
\item[\rm (vii).] $p=3$, $B(n_1,2)\ast_{B(1,2)}A_5\ast_{C_5}A_5\ast_{B(1,2)}B(n_2,2)$
            with $n_1,n_2\geq 2$.
\item[\rm (viii).] $p=3$, $B(n_1,2)\ast_{B(1,2)}A_5\ast_{C_5}{\rm PSL}_2(\mathbb{F}_9)\ast_{B(2,4)}B(2\cdot n_2,4)$
            with $n_1, n_2\geq 2$.
\item[\rm (ix).] $p=2$, $2\not|\ell$, $B(n_1,\ell)\ast_{C_\ell}D_\ell\ast_{C_2}B(n_2,1)$
          with $\ell| 2^{n_1}-1$; $n_1,n_2\geq 2$.
\item[\rm (x).] $p=2$, $2\not|\ell$, $B(n_1,1)\ast_{C_2}D_\ell\ast_{C_\ell}D_\ell\ast_{C_2}B(n_2,1)$
            with $n_1,n_2\geq 2$.
\item[\rm (xi).] $p=2$, $q>2$,
  $B(n_1,1)\ast_{C_2}D_{q+1}\ast_{C_{q+1}}{\rm PGL}_2(\mathbb{F}_q)
           \ast_{B(n,q-1)}B(n\cdot n_2,q-1)$
            with $n_1,n_2\geq 2$.
\end{enumerate}
\end{small}
\end{proposition}

\section{ \rm Mumford groups with three branch points}

According to Theorem 3.1, a realizable amalgam
    $\Gamma$ has three branch points if 
one of the following statements holds:
\begin{enumerate}
\item[(i)] The group $\Gamma$ is indecomposable and ${\rm max}(3)=0$ and ${\rm maxp}=1$.
\item[(ii)] The group $\Gamma$ is indecomposable and ${\rm max}(3)=1$ and ${\rm maxp}=0$.
\item[(iii)] The group $\Gamma$ is a free amalgam
  $E\ast\Gamma^\prime$ with $\Gamma^\prime$
         is   a discontinuous group or a finite  group with ${\rm br}(\Gamma^\prime)=2$ and $E$ a finite $p$-group. 
\end{enumerate} 
In the sequel we consider the cases (i) and (ii).

\subsection{\rm The case ${\rm br}(\Gamma)=3$, ${\rm maxp}=1$ and ${\rm max}(3)=0$.}
\begin{remark} \label{4.2} {\rm 
From [P-V] we recall that ${\rm maxp}$ is only defined for $p=2,3$ and
occurs in the description in Theorem 3.14 of the contracted finite, 
indecomposable  tree of groups $(T=T^c,G)$ associated to an indecomposable
discontinuous group $\Gamma$. There can be vertices $v\in T$
such that the vertex group $\Gamma_v$ is a $p$-group $C_p^{n_0}$ with $n_0\geq 1$ 
and has $d_v\geq 2$ edges $e=\{v,\tilde{v}\}$ with edge group $C_p$. For $p=2$, the vertex
group $\Gamma_{\tilde{v}}$ is $D_\ell$ with odd $\ell$. For $p=3$, the vertex
group $\Gamma_{\tilde{v}}$ is ${\rm PSL}_2(\mathbb{F}_3)$.  Now ${\rm maxp}$ is the sum of all  $(d_v-1)$. If
${\rm maxp}=1$, then there is only one such vertex $v$ and $d_v=2$. 

 For $p=2$ this means that $D_\ell \ast_{C_2}C_2^{n_0}\ast_{C_2}D_m$
 with $n_0\geq 1$ and odd $\ell , m$ occurs precisely once in the amalgam for
 $\Gamma$. We note that, for a technical reason,   
$D_\ell \ast_{C_2}D_m$ is not allowed in \cite{P-V} Theorem 3.14 and its
occurrence is replaced by  $D_\ell \ast_{C_2}C_2\ast_{C_2}D_m$. 
We will adhere to the same convention in the propositions below.
In the same way,  ${\rm maxp}=1$ for $p=3$ means that  
${\rm PSL}_2(\mathbb{F}_3)\ast_{C_3}C_3^{n_0}\ast_{C_3}{\rm PSL}_2(\mathbb{F}_3)$ with $n_0\geq 1$ occurs
once  in the amalgam for $\Gamma$. We conclude:\\

\noindent 
 {\it Let ${\rm br}(\Gamma)=3$, ${\rm maxp}=1$ and ${\rm max}(3)=0$.
Then $p| \sharp\Gamma_v$ for all vertices $v\in T^c$ and one of the following two statements holds: 
\begin{enumerate}
\item[\rm (i).] $p=2$ and $D_\ell\ast_{C_2}C_2^{n_0}\ast_{C_2}D_m$, $2\not|\ell,m$
         with $n_0\geq 1$ occurs exactly once in the description
         of $\Gamma$ as an amalgam.
\item[\rm (ii).] $p=3$ and ${\rm
    PSL}_2(\mathbb{F}_3)\ast_{C_3}C_3^{n_0}\ast_{C_3}{\rm PSL}_2(\mathbb{F}_3)$ 
        with $n_0\geq 1$ occurs exactly once in the
        description of $\Gamma$ as an amalgam.
\end{enumerate} }

We recall from the beginning of \S 3, that for $p=2$ the two edge
groups $C_2$ are mapped to the same subgroup of $C_2^{n_0}$. Similarly,
for $p=3$, the two edge groups $C_3$ are mapped to the same subgroup
of $C_3^{n_0}$.\\   

In order to find all the possible amalgams, listed in Propositions \ref{4.3}
and \ref{4.4},  it is sufficient to determine all the finite groups $G$
that are not of Borel type and whose order is divisible by $p$
and that have a branch group different from $ C_p$ and $ C_p^{n_0}$ with $n_0\geq 1$ in common with the group
 $D_\ell\ast_{C_2}C_2^{n_0}\ast_{C_2}D_m$, $2\not|\ell,m$ for $p=2$
 and   ${\rm PSL}_2(\mathbb{F}_3)\ast_{C_3}C_3^{n_0}\ast_{C_3}{\rm PSL}_2(\mathbb{F}_3)$ for $p=3$.
 This branch group is a cyclic group $C_r$ with $r$ not divisible by $p$.
One can now create a realizable amalgam by adding $\ast_{C_r}B(n_1,r)$, $\ast_{C_r}G$ 
or $\ast_{C_r}G\ast_{B_G}B(n_1,s)$ to the already obtained amalgam. 
Here $B_G$ is the branch group of $G$ distinct from $C_r$. 
} \end{remark}

\begin{proposition}\label{4.3}
Let $p=2$ and let $\Gamma$ be such that ${\rm br}(\Gamma)=3$, ${\rm maxp}=1$ and ${\rm max}(3)=0$.
Then $\Gamma$ is one of the following amalgams:
\begin{small}
\begin{enumerate}
\item[\rm (i).] $B(n_1,\ell)\ast_{C_\ell}D_\ell\ast_{C_2}C_2^{n_0}\ast_{C_2}D_m\ast_{C_m}B(n_2,m)$,
          odd $\ell,m$.
\item[\rm (ii).] $B(n_1,\ell)\ast_{C_\ell}D_\ell\ast_{C_2}C_2^{n_0}\ast_{C_2}D_m\ast_{C_m}D_m\ast_{C_2}B(n_2,1)$,
             odd $\ell,m$.
\item[\rm (iii).] $B(n_1,1)\ast_{C_2}D_\ell\ast_{C_\ell}D_\ell\ast_{C_2}C_2^{n_0}\ast_{C_2}D_m\ast_{C_m}D_m\ast_{C_2}B(n_2,1)$,
            odd $\ell,m$.
\item[\rm (iv).]  $q>2$,
  $B(n_1,\ell)\ast_{C_\ell}D_\ell\ast_{C_2}C_2^{n_0}\ast_{C_2}D_{q+1}\ast_{C_{q+1}}{\rm PGL}_2(\mathbb{F}_q)
            \ast_{B(\mathbb{F}_q)}B(n_2,q-1)$, odd $\ell$.
\item[\rm (v).]  $q_1,q_2>2$, $B(n_1,q_1-1)\ast_{B(\mathbb{F}_{q_1})}{\rm PGL}_2(\mathbb{F}_{q_1})\ast_{C_{q_1+1}}D_{q_1+1}
           \ast_{C_2}C_2^{n_0}\ast_{C_2}D_{q_2+1}\ast_{C_{q_2+1}}{\rm PGL}_2(\mathbb{F}_{q_2})
           \ast_{B(\mathbb{F}_{q_2})}B(n_2,q_2-1)$.
\end{enumerate}
\end{small}
       Here $n_0\geq 1$.
       The group $B(\mathbb{F}_q)=B(n,q-1)$ is a Borel subgroup of
       ${\rm PGL}_2(\mathbb{F}_q)$. As in Proposition {\rm  3.2} we have
       written the amalgams of maximal length. By deleting end group(s) one obtains
       the other possibilities. 
  \end{proposition}

\begin{proposition}\label{4.4}
Let $p=3$ and let $\Gamma$ be such that ${\rm br}(\Gamma)=3$, ${\rm maxp}=1$ and ${\rm max}(3)=0$.
Then $\Gamma$ is one of the following amalgams:
\begin{small}
\begin{enumerate}
\item[\rm (i).] $B(n_1,2)\ast_{C_2}{\rm PSL}_2(\mathbb{F}_3)\ast_{C_3}C_3^{n_0}\ast_{C_3}{\rm PSL}_2(\mathbb{F}_3)\ast_{C_2}B(n_2,2)$.
\item[\rm (ii).]
  $B(n_1,2)\ast_{C_2}{\rm PSL}_2(\mathbb{F}_3)\ast_{C_3}C_3^{n_0}\ast_{C_3}{\rm
    PSL}_2(\mathbb{F}_3)\ast_{C_2}{\rm PSL}_2(\mathbb{F}_3)\ast_{C_3}B(n_2,1)$.
\item[\rm (iii).] $B(n_1,1)\ast_{C_3}{\rm
    PSL}_2(\mathbb{F}_3)\ast_{C_2}{\rm
    PSL}_2(\mathbb{F}_3)\ast_{C_3}C_3^{n_0}\ast_{C_3}{\rm
    PSL}_2(\mathbb{F}_3)\ast_{C_2}{\rm PSL}_2(\mathbb{F}_3)\ast_{C_3}B(n_2,1)$.
\end{enumerate}
\end{small}
The same remarks as in \ref{3.2} and \ref{4.3} apply.
 \end{proposition}

\subsection{ \rm The case ${\rm br}(\Gamma)=3$, ${\rm maxp}=0$ and ${\rm max}(3)=1$.}
There exists exactly one vertex $v_0\in T^c$ such that $\ \ \ p\not| \#\Gamma_{v_0}$.
For all other vertices $v\in T^c$ one has $p| \#\Gamma_v$. 
The group $\Gamma_{v_0}$ equals one of the groups $D_\ell, A_4, S_4, A_5$, 
where $p$ does not divide the order of the group.

The group $\Gamma_{v_0}$ has three branch points 
corresponding to the maximal cyclic subgroups of $\Gamma_{v_0}$.
The triples consisting of the orders of these cyclic groups are
 $(2,2,\ell), (2,3,3), (2,3,4)$ and $(2,3,5)$ for the groups
$D_\ell, A_4, S_4$ and $A_5$, respectively. 
The maximal cyclic subgroups of order 2 (resp. 3) in the group $D_\ell$ 
with $\ell$ odd (resp. $A_4$) are conjugated.
Hence the stabilisers of the branch points are the same.
If $\ell$ is even, then the branch points of the group $D_\ell$ 
correspond to groups that are in different conjugacy classes.\\ 

In order to find all the possible amalgams, i.e., the lists of
propositions \ref{4.5}--\ref{4.8}, it is sufficient to determine all the finite groups $G$
that are not of Borel type and whose orders are divisible by $p$
and  have a branch group  in common with the group
 $\Gamma_{v_0}$.\\
 
In the following propositions we write (as before) the list of amalgams of maximal
length. By deleting end group(s) one obtains all possibilities. 

\begin{proposition}\label{4.5} For $\Gamma_{v_0}\cong A_5$ and $p>5$, the group
  $\Gamma$ is equal to\\
 the amalgam of $A_5$ along its maximal cyclic subgroups $C_2$, $C_3$ and $C_5$
  to groups $B(n_1,2)$, $B(n_2,3)$ and $B(n_3,5)$, respectively.
\end{proposition}

\begin{proposition}\label{4.6}
For $\Gamma_{v_0}\cong S_4$ and $p>3$, $\Gamma$ is one of the following:
\begin{enumerate}
\item[\rm (i).] The amalgam of $S_4$ along its maximal cyclic subgroups $C_2$, $C_3$ and $C_4$
          to groups $B(n_1,2)$, $B(n_2,3)$ and $B(n_3,4)$, respectively.
\item[\rm (ii).] If $p=5$, then one can replace the group  $B(n_2,3)$
  in item {\rm (i)} by the group ${\rm PSL}_2(\mathbb{F}_5)\ast_{B(\mathbb{F}_5)}B(n_4,2)$ with $n_4\geq 1$. 
\item[\rm (iii).]  If $p=7$, then one can replace the group
  $B(n_3,4)$ in item {\rm (i)}
           by the group ${\rm PSL}_2(\mathbb{F}_7)\ast_{B(\mathbb{F}_7)}B(n_5,3)$ with $n_5\geq 1$.
\end{enumerate}
\end{proposition}
\begin{proof}
The extra groups in (ii) and (iii) for $p=5,7$ are possible because $(p+1)/2=3,4$, respectively.
\end{proof}

\begin{proposition}\label{4.7}
For $\Gamma_{v_0}\cong A_4$ and $p>3$, $\Gamma$ is one of the following:
\begin{enumerate}
\item[\rm (i).] The amalgam of $A_4$ along its maximal cyclic subgroups $C_2$, $C_3$ and $C_3$
          to groups $B(n_1,2)$, $B(n_2,3)$ and $B(n_3,3)$, respectively.
\item[\rm (ii).] If $p=5$, then one can replace one or both of the
  groups $B(n_2,3)$ and $B(n_3,3)$ in item 
{\rm (i)} by the group ${\rm PSL}_2(\mathbb{F}_5)\ast_{B(\mathbb{F}_5)}B(n_4,2)$ with $n_4\geq 1$. 
\end{enumerate}
\end{proposition}

\begin{proof}
The extra groups in (ii) for $p=5$ are possible because  $\frac{p+1}{2}=3$.\end{proof}

\begin{proposition}\label{4.8}
$\Gamma_{v_0}\cong D_\ell$, $p\not|\ell$ and $p>2$. 
Then $\Gamma$ is one of the following:
\begin{enumerate}
\item[\rm (i).] The amalgam of $D_\ell$ along its maximal cyclic subgroups $C_2$, $C_2$ and $C_\ell$
  to groups $B(n_1,2)$, $B(n_2,2)$ and $B(n_3,\ell)$, respectively.
 The cyclic groups $C_2$ are  identical in the group $D_\ell$ only if $\ell$ is odd.
\item[\rm (ii).] If $\ell=q+1$, then one can replace the group
  $B(n_3,\ell)$ in item {\rm (i)}
           by the group ${\rm
             PGL}_2(\mathbb{F}_q)\ast_{B(\mathbb{F}_q)}B(n_4,q-1)$
           with $n_4\geq 1$.
  \item[\rm (iii).] If $\ell=(q+1)/2$, then one can replace the group
  $B(n_3,\ell)$ in item {\rm (i)}
           by the group ${\rm PSL}_2(\mathbb{F}_q)\ast_{B(\mathbb{F}_q)}B(n_5,(q-1)/2)$ with 
           $n_5\geq 1$. 
\item[\rm (iv).] If $p=3$ and $\ell=5$, then one can replace the group
  $B(n_3,\ell)=B(n_3,5)$ in item {\rm (i)}
           by the group $A_5\ast_{B(1,2)}B(n_6,2)$ with $n_6\geq 1$. 
\item[\rm (v).] If $p=3$, then one can replace one or both of the groups $B(n_1,2)$ and $B(n_2,2)$ 
           in items {\rm (i)-(iv)}
           by a group ${\rm PSL}_2(\mathbb{F}_3)\ast_{C_3}B(n_7,1)$ with $n_7\geq 1$. 
\end{enumerate}
\end{proposition}

\begin{remarks}\label{CKK-1} \label{4.9} {\it Comparison with the
    results of }{\rm \cite{C-K-K,K 2005}}. \\
 {\rm (1). The groups $\Gamma$ with two branch point that are missing 
from proposition 4.6 of \cite{C-K-K} are the following:
\begin{enumerate}
\item[i)] $p=3$, $B(n_1,2)\ast_{C_{2}}{\rm PSL}_2(\mathbb{F}_3)
            \ast_{C_3}B(n_2,1)$
            with $n_1,n_2\geq 1$ and $n_1$ odd.
\item[ii)] $p=3$, $B(n_1,2)\ast_{B(1,2)}A_5\ast_{C_5}{\rm PSL}_2(\mathbb{F}_9)\ast_{B(2,4)}B(2\cdot n_2,4)$
            with $n_1, n_2\geq 1$.
\end{enumerate}
(2). The groups $\Gamma$ with three branch point that are missing 
from proposition 4.7 of \cite{C-K-K} are the following:
\begin{enumerate}
\item[i)] The groups $\Gamma$  with ${\rm maxp}=1$.
\item[ii)] For $p=3$ the amalgam of $D_\ell$ along its maximal cyclic subgroups $C_2$, $C_2$ and $C_\ell$
          to groups ${\rm PSL}_2(\mathbb{F}_3)\ast_{C_3}B(n_1,1)$,
          $B(n_2,2)$ and $B(n_3,\ell)$, respectively.
          Here $\ell| p^{n_3}-1$ and $n_1,n_2, n_3\geq 0$.
          The group  $B(n_2,2)$ may be replaced by a group
          ${\rm PSL}_2(\mathbb{F}_3)\ast_{C_3}B(n_4,1)$ with $n_4\geq 1$.
\end{enumerate}
\noindent (3). Our findings also conflict with the statement in remark 8.8 in \cite{K 2005}
that the groups in proposition 4.7 (D) and (E) in \cite{C-K-K} are not realizable.\\ 


 } \end{remarks}

\section{\rm Realizable amalgams $\ast_H\{ G_i\mid i=1,\ldots,m\}$}

In \S 2 we recalled the realizable amalgams $G_1\ast_{G_3}G_2$ from
\cite{P-V}. 
{\it In this section we study, for completeness, amalgams of three or
  more finite non-cyclic groups $G_i$ along a single subgroup $H$. The
results of this section will not be used in }\S\S 6-9.

A tree $T^c$ associated to such a group is such that every edge $e\in T^c$ is stabilized by the group $H$.
If $T^c$ contains three or more  vertices, then there is at least one vertex $v\in T^c$
such that two edges $e\ni v$ are stabilized by {\it the same group}
$H$.  Thus, if $\cdots \Gamma _u\ast_H\Gamma_v\ast_H\Gamma _w\cdots $
is part of the tree of groups $T^c$, then we mean that the two
(injective) homomorphisms $H\rightarrow \Gamma_v$ are the same (or equivalently,
$H$ is identified in only one way as a subgroup of $\Gamma_v$).

It follows that the group $\Gamma_v$ (stabiliser of the vertex $v$) either has two distinct branch points
that are stabilized by the same group $H$ or the group $\Gamma_v$ is a $p$-group and $p$ equals $2$ or $3$.
The cases where $\Gamma_v$ is a $p$-group are exactly those with ${\rm maxp}>0$.
They have been described in \cite{P-V} remarks 3.15 and theorem 4.11.

The only finite non-cyclic groups that have two distinct branch points that
are stabilized by the same subgroup are the groups
$D_\ell$, $p\not=2$, $p\not| \ell$ with $\ell$ odd and the group $A_4$ with $p\not=2, 3$.
The branch groups are maximal cyclic subgroups of order 2 for the group $D_\ell$ 
with $\ell$ odd and of order $3$ for the group $A_4$.
Hence $H\cong C_2$ and $p\not=2$ or $H\cong C_3$ and $p\not=2,3$ must hold.
Moreover, if $H\cong C_2$, then every vertex $v\in T^c$ 
that is not an extremal vertex must have as its stabilizer
a group $D_\ell$, $p\not\mid \ell$ with $\ell$ odd.
If $H\cong C_3$, then every vertex $v\in T^c$ 
that is not an extremal vertex must have as its stabilizer
a group $A_4$. 

The amalgam $\Gamma$ itself does not depend on the order of the groups $G_i$,
but the realizability of $\Gamma$ as a discontinuous group does
depend on the order of the groups $G_i$  (see also \cite{P-V} remark 3.13 (3)).  

We summarise the results in the following proposition:

\begin{proposition}\label{5.1}
Let $\Gamma:=\ast_H\{ G_i\mid i=1,\ldots,m\} $ with $m\geq 3$
be the amalgam of the finite non-cyclic groups $G_i$ along a single subgroup $H$.
Then the amalgam $\Gamma$ is realizable as a discontinuous group if and only if
one of the following statements holds:
\begin{enumerate}
\item[{\rm (i)}] $p\not=2$, $H\cong C_2$, $G_i\cong D_{\ell_i}$, 
         $p\not| \ell_i$, $\ell_i$ odd for $i=2,\dots,m-1$
          and $G_1,G_m$ are groups that have a branch group $C_2$.
\item[{\rm (ii)}] $p\not=2, 3$, $H\cong C_3$, $G_i\cong A_4$ for $i=2,\dots,m-1$
          and $G_1,G_m$ are groups that have a branch group $C_3$.
\item[{\rm (iii)}] $p=2$, $H\cong C_2$, $G_i\cong D_{\ell_i}$, $\ell_i$ odd for $i=2,\dots,m$
             and $G_1\cong D_{\ell_1}$, $\ell_1$ odd or $G_1\cong C_2^{n_0}$, $n_0\geq 2$.
\item[{\rm (iv)}] $p=3$, $H\cong C_3$, $G_i\cong A_4\cong {\rm PSL}_2(\mathbb{F}_3)$ for $i=2,\dots,m$
             and $G_1\cong A_4$ or $G_1\cong C_3^{n_0}$, $n_0\geq 2$.
\end{enumerate}

In the cases {\rm (i)} and {\rm (ii)} the tree $T^c$ is a line segment consisting of $m$ vertices $v_i$
with extremal vertices $v_1$ and $v_m$ and stabilizers $\Gamma_{v_i}=G_i$.

In the cases {\rm (iii)} and {\rm (iv)} the tree $T^c$ is a star with all edges containing the central vertex $v_0$,
where $v_0=v_1$ if $G_1$ is a $p$-group.
The tree consists of $m$ vertices $v_i$ with $\Gamma_{v_i}= G_i$, $i=1,\ldots,m$ if $G_1$ is a $p$-group
and of $m+1$ vertices $v_i$ with $\Gamma_{v_i}= G_i$, $i=1,\ldots,m$ 
and $\Gamma_{v_0}\cong C_p$ if $G_1$ is not a $p$-group. 
\end{proposition}

We sketch an alternative proof of Proposition 5.1 in the following  remark.
\begin{remark}\label{5.2} {\rm
The group $D_{2\ell}$ normalizes the group $D_\ell$ if $p\not=2$
and the group $S_4$ normalizes the group $A_4$ if $p\not=3$.
These normalizers permute the two branch points stabilized by the cyclic groups
$C_2$ and $C_3$ in the groups $D_\ell$ and $A_4$, respectively.
This allows us to construct examples of groups of type (i) and (ii)
as subgroups of a realizable amalgam of two finite groups.\\

Let $\Gamma$ be a discontinuous group, isomorphic to $S_4\ast_{C_3}S_4$.
Let us fix a subgroup $H\cong C_3$ in $S_4\ast_{C_3}S_4$.
Define $\Delta :=<G\subset S_4\ast_{C_3} S_4\mid G\cong A_4, H\subset G>$.
Then $\Delta$ is a well-defined amalgam $\ast_H\{ G\subset S_4\ast_{C_3} S_4\mid G\cong A_4, H\subset G\}$.
It is an amalgam of infinitely many groups $A_4$ along a single group $C_3$.
The group $\Delta$ contains all the subgroups of $S_4\ast_{C_3}S_4$ 
that are isomorphic to the group $A_4$.
The subgroup $\Delta\subset \Gamma$ is normal.
The group $\Delta$ has infinite index in the group $S_4\ast_{C_3}S_4$ and is not finitely generated.
One has $\Gamma/\Delta \cong (S_4/A_4)\ast (S_4/A_4)\cong C_2\ast C_2$.
Examples of amalgams $\ast_{C_3}\{ G_i\mid G_i\cong A_4, i=1,\ldots, m\}$ 
for any value of $m$ are contained in
the realizable amalgam $S_4\ast_{C_3}S_4$.\\

A similar construction using the group $D_{2\ell}\ast_{C_2}D_{2\ell}$ gives examples of realizable amalgams
$\ast_{C_2}\{ G_i\mid G_i\cong D_{\ell_i}, \ell_i|\ell, i=1,\ldots, m \}$.
Let us fix a subgroup $H\cong C_2$ in $D_{2\ell}\ast_{C_2}D_{2\ell}$.
The subgroup $< G\subset D_{2\ell}\ast_{C_2}D_{2\ell}\mid
G\cong D_\ell, H\subset G>$ is an amalgam
$\ast_H\{ G\subset D_{2\ell}\ast_{C_2}D_{2\ell}\mid
G\cong D_\ell, H\subset G\}$.
The amalgam is well-defined, normal and of infinite index
in the group $D_{2\ell}\ast_{C_2}D_{2\ell}$.
It follows that the realizable amalgam $D_{2\ell}\ast_{C_2}D_{2\ell}$
contains subgroups that are amalgams of the form
$\ast_{C_2}\{ G_i\mid G_i\cong D_{\ell_i}, \ell_i|\ell, i=1,\ldots, m \}$
for any  $m$.\\

For the groups of type (iii) and (iv) this method does not work.
However, one can embed these in a realizable amalgam of two groups that are not finitely generated.
Let $A\subset K$ be the subgroup generated by the elements $\pi^{-n}$, $n\in \mathbb{Z}_{\geq 0}$ 
for fixed $\pi\in K$ with $0<|\pi|<1$. Then $A$ is an infinite dimensional $\mathbb{F}_p$ vector space.
The group $B_p:={1\ A\choose 0\ 1 }\subset {\rm PGL}_2(K)$ is
discontinuous, but is not finitely generated.
The amalgams $D_\ell\ast_{C_2}B_2$ ($p=2$) 
and ${\rm PSL}_2(\mathbb{F}_3)\ast_{C_3}B_3$ ($p=3$) are realizable. 

For $p=2$ one  considers the subgroup  generated by all the subgroups $D_\ell$ 
and for $p=3$ the subgroup generated by all the subgroups ${\rm PSL}_2(\mathbb{F}_3)$. 
Both are normal subgroups of infinite index in the amalgams.
These subgroups are amalgams along a single group $C_p$
of infinitely many groups $D_\ell$ if $p=2$ and infinitely many groups
${\rm PSL}_2(\mathbb{F}_3)$ if $p=3$.

Indeed, in case $p=2$ one fixes  a subgroup $H\cong C_2$ of a group $G\cong D_\ell$
that is contained in the amalgam $D_\ell\ast_{C_2}B_2$.
Then the infinite amalgam $\ast_H\{ G\in D_\ell\ast_{C_2}B_2\mid G\cong D_\ell, H\subset G\}$
is well-defined and equals the group $< G\subset D_\ell\ast_{C_2}B_2\mid G\cong D_\ell>$.

For $p=3$, fix $H\cong C_3$ as subgroup of a subgroup $G\cong {\rm PSL}_2(\mathbb{F}_3)$
of ${\rm PSL}_2(\mathbb{F}_3)\ast_{C_3}B_3$. Then
$\ast_H\{ G\in {\rm PSL}_2(\mathbb{F}_3)\ast_{C_3}B_3\mid G\cong {\rm PSL}_2(\mathbb{F}_3), H\subset G \}$
is well-defined and equals the group $<G\subset {\rm PSL}_2(\mathbb{F}_3)\ast_{C_3}B_3\mid
G\cong {\rm PSL}_2(\mathbb{F}_3)>$.
}\end{remark}

\section{\rm Automorphisms of Mumford curves}  \label{section 6}

\subsection{\rm The function $\mu$} \label{section 6.1}
As explained in the introduction we want to investigate Mumford
curves $X$ of genus $g>1$ such that $|{\rm Aut}(X)|>12(g-1)$. Mumford
curves of this type are produced by choosing a realization $\Gamma
\subset {\rm PGL}(2,K)$ of an amalgam with 2 or 3 branch points (thus
from the lists in \S\S 3-4) and choosing a normal Schottky subgroup of
finite index $\Delta \subset \Gamma$. 

 Let $\Omega \subset \mathbb{P}^1_K$
denote the rigid open set of ordinary points for $\Gamma$. Then
$X=\Omega /\Delta$ and $\Gamma/\Delta$ is a subgroup of ${\rm Aut}(X)$.
{\it We note that $\Gamma/\Delta$ can be a proper subgroup of ${\rm Aut}(X)$}.\\

One defines $\mu(\Gamma)$ by the formula $g-1=\mu(\Gamma) \cdot
[\Gamma :\Delta]$, where $g$ is the number of free generators of $\Delta$.
 If one replaces $\Delta$ by a normal subgroup of
index $d$, then both sides of the equality are multiplied by
$d$. Hence $\mu(\Gamma)$ does not depend on the choice of $\Delta$.\\

We introduce the notation $\mu(G)=-\frac{1}{|G|}$ for any finite group
$G$. As before, the amalgam $\Gamma$ corresponds to a canonical 
finite tree of finite groups $T^c$. By a combinatorial 
analysis, see  \cite{K-P-S}, one obtains the formula
\[ \mu(\Gamma)=\sum_{v\in T^c}\mu(\Gamma_v)-\sum_{e\in T^c}\mu(\Gamma_e)
             =\sum_{e\in T^c}1/|\Gamma_e|-\sum_{v\in T^c}1/|\Gamma_v|.\]

\noindent This formula can also be derived from the usual Riemann--Hurwitz--Zeuthen formula using
the methods of the proof of theorem 5.3 in \cite{P-V}.
Instead of only counting the branch points, one can extend the proof 
and keep track of the contribution of each branch group to the Riemann--Hurwitz--Zeuthen formula.
              
We note that this formula makes also sense for a finite group (which
is the case $g=0$) and for decomposable amalgams. In particular, suppose that the
amalgam $\Gamma$ is the free product of amalgams $\Gamma_1$ and
$\Gamma_2$. Then one has $\mu(\Gamma)=1+\mu(\Gamma_1)+\mu(\Gamma_2)$.

Consider the example $\Gamma=C_2\ast C_3=<a,b|\  a^2=b^3=1>$. Then 
$\mu(\Gamma)=1-\frac{1}{2}-\frac{1}{3}$ and 
$\Delta=<abab, baba>$ is a Schottky group of rank $2$ and $\Gamma
/\Delta =D_3$.  Thus $X=\Omega /\Delta$ has genus 2 and the formula
$g-1=\mu(\Gamma)\cdot |\Gamma/\Delta|$ holds for this example.\\

In the sequel we will exclude the case $g=0$ and the case $g=1$, which
corresponds to the amalgams $D_\ell *_{C_\ell}D_\ell$ for $(\ell,
p)=1$ and $\mu(\Gamma)=0$. Moreover we will suppose that $\Gamma$ is
indecomposable (since otherwise $\mu (\Gamma) \geq \frac{1}{6}$).

 For a choice of a Schottky group $\Delta \subset
\Gamma$ (normal and of finite index) the group ${\rm Aut}(X)$ equals
$\Gamma'/\Delta$, where $\Gamma '\supset \Gamma$ is the normaliser of
$\Delta$ in ${\rm PGL}_2(K)$. One can verify that $\Gamma'$ is again
indecomposable. By the above formulas one has $\mu(\Gamma ')\cdot
[\Gamma':\Gamma]= \mu(\Gamma)$. See \S 6.3 for examples.\\

{\it The strategy for  the sections {\rm \S\S 6-8} is as  follows}. \\
In \S \ref{section 6.2} we compute the lists \ref{6.3} of
realizable $\Gamma$ with two or three branch points and with
$\mu(\Gamma)<\frac{1}{12}$. The phenomenon of inclusions 
$\Gamma \subset \Gamma '$ in the lists \ref{6.3} is clarified 
in \S \ref{section 6.3} by introducing a determinant.

 In \S \ref{section 6.4} we produce normal Schottky subgroups $\Delta$ of realizable $\Gamma$
(from the lists in \S \ref{section 6.2}) of minimal index. Theorem \ref{6.8} describes two
extreme families found in this way.  The Mumford curves corresponding to these extreme families 
are studied in \S \ref{section 7.1} and \S \ref{section 7.2}.

 Based on
these extreme families a precise bound for the order of the group of
automorphisms is proposed. Finally, the long section \S \ref{section
  8} provides a proof of this bound.

\subsection{\rm Amalgams $\Gamma$ with $\mu(\Gamma) <\frac{1}{12}$
}\label{section 6.2}

The first step is to show that for many amalgams $\mu(\Gamma)\geq \frac{1}{12}$
holds. A useful formula for computations is:\\

Suppose that the tree $T^c$ of $\Gamma$ has an edge $e$ with
vertices $v_1,v_2$ and $T^c\setminus \{e\}$ has two connected
components $\Gamma_1$ and $\Gamma_2$. Then one has  
 \[\mu(\Gamma)=\mu(\Gamma_1)+1/|\Gamma_{v_1}|+\mu(\Gamma_{v_1}\ast_{\Gamma_{e}}\Gamma_{v_2})
+1/|\Gamma_{v_2}|+\mu(\Gamma_2)\geq \mu(\Gamma_{v_1}\ast_{\Gamma_{e}}\Gamma_{v_2}).\]
The inequality is strict if $\Gamma \neq  \Gamma_{v_1}\ast_{\Gamma_{e}}\Gamma_{v_2}$.

\begin{lemma}\label{6.1}
Let $\Gamma$ be an indecomposable realizable amalgam ($\neq
D_\ell*_{C_\ell}D_\ell$) such that its canonical tree $T^c$ contains an edge $e$ 
stabilised by a cyclic group of order $m\leq 5$. Then $\mu(\Gamma)\geq 1/12$.
\end{lemma}

\begin{proof} Using the above formula,
we may assume that the realizable $\Gamma$ equals $G_1*_{C_m}G_2$ with $G_1,G_2$ finite groups. 
 Consider first the case where $p\not| m$. This corresponds to the cases  (a) part 1,  (b)
part 1 of Proposition \ref{2.1}. One easily verifies that always $\mu(\Gamma)\geq \frac{1}{12}$. The case
$p|m$ corresponds to the cases (a) part 2, 3; (b)
part 4 of Proposition 2.1. Again one finds $\mu(\Gamma)\geq \frac{1}{12}$.    
\end{proof}

\begin{corollary}\label{6.2} Suppose that $\mu(\Gamma)<\frac{1}{12}$. Then the
  vertex groups $\Gamma_v$ belong to $\{D_\ell, {\rm
    PGL}_2(\mathbb{F}_q), {\rm PSL}_2(\mathbb{F}_q)$, $B(n,\ell)\}$.
Moreover, ${\rm maxp}=0$ and $\Gamma$ has at most three branch points.
\end{corollary}

\begin{proof}
By Lemma \ref{6.1}, the groups $A_4,S_4$ do not occur as vertex groups for
$\Gamma$. The group $A_5$ does not occur. Indeed, for $p=3$ one has for case  (b)
part 5 of Proposition \ref{2.1} that:
   \[\mu(A_5\ast_{B(1,2)}B(m,2))\geq\mu(A_5\ast_{B(1,2)}B(2,2)) 
   =\frac{17}{180} >\frac{1}{12}.\]
Further, ${\rm maxp}=1$ implies that an edge has stabilizer $C_2$ or $C_3$.

The Riemann--Hurwitz--Zeuthen formula for $X:=\Omega /\Delta\rightarrow
\mathbb{P}^1=\Omega /\Gamma$
reads $2g-2=(-2)|\Gamma/\Delta |+\sum
_{i=1}^m\frac{|\Gamma/\Delta|}{e_ip^{d_i}}((e_i+1)p^{d_i}-2)$, where
the branch points are  $a_1,\dots ,a_m\in \mathbb{P}^1$ with
ramification indices $e_ip^{d_i}, \ i=1,\dots ,m$ and all $(p,e_i)=1$.
Then $\mu(\Gamma)<\frac{1}{12}$, the data from \ref{6.1} and the earlier
part of \ref{6.2} imply that $m\leq 3$ must hold. 
\end{proof}

\begin{lists}\label{lists}\label{6.3}{\rm
 Using \ref{6.1} and \ref{6.2} one finds that the amalgams with $\mu
<\frac{1}{12}$ belong to special cases of Proposition \ref{3.2}
and Proposition \ref{4.8}. For each of these cases we compute (some)
values of $p,q,n_*$ with $\mu <\frac{1}{12}$. We keep the numbering of
the groups in \ref{3.2} and \ref{4.8} and we denote  the groups by
$A(i),\ A(ii),\dots, B(i),\dots$.\\
{\bf (A)  two branch points and Proposition \ref{3.2}}:\\
$A(i)$.  {\small $B(2n\cdot n_1,q+1)\ast_{C_{q+1}}{\rm
  PGL}_2(\mathbb{F}_q)\ast_{B(n,q-1)}B(n\cdot n_2,q-1)$}  with $n_1\geq
1$, $n_2\geq 2$. For $n_1=1,\ n_2=2$ one  has
$\mu=\frac{q^2-2}{q(q^2-1)}$
and this is $<\frac{1}{12}$ for $q\geq 13$. \\

\noindent $A(ii)$. {\small $B(n\cdot n_1,q-1)\ast_{B(n,q-1)}{\rm PGL}_2(\mathbb{F}_q)
           \ast_{C_{q+1}}{\rm PGL}_2(\mathbb{F}_q)\ast_{B(n,q-1)}B(n\cdot n_2,q-1)$}
           with  $n_1,n_2\geq 2$. For $n_1=n_2=2$ one has
           $\mu=\frac{q^2-2}{q^2(q-1)}$ and this is $<\frac{1}{12}$
           for $q\geq 13$.\\

\noindent $A(iii)$. $p\neq 2$, {\small $B(2\cdot n\cdot
  n_1,(q+1)/2)\ast_{C_{(q+1)/2}}{\rm PSL}_2(\mathbb{F}_q)
            \ast_{B(n,(q-1)/2)}B(n\cdot n_2,(q-1)/2)$}
            with $n_1\geq 1$, $n_2\geq 2$ and $q\geq 11$. For $n_1=1,\
            n_2=2$ one has $\mu=\frac{2(q^2-2)}{q(q^2-1)}$ and this is
            $<\frac{1}{12}$ for $q\geq 25$.\\

\noindent $A(iv)$. {\small $p\neq 2$, $B(n\cdot
  n_1,(q-1)/2)\ast_{B(n,(q-1)/2)}{\rm PSL}_2(\mathbb{F}_q)
            \ast_{C_{(q+1)/2}}{\rm PSL}_2(\mathbb{F}_q)\ast_{B(n,(q-1)/2)}B(n\cdot n_2,(q-1)/2)$}
            with $n_1,n_2\geq 2$. For $n_1=n_2=2$ one has
            $\mu=\frac{2(q^2-2)}{q^2(q-1)}$ and this is
            $<\frac{1}{12}$ for $q\geq 25$.\\

\noindent $A(v)$. {\small $B(n_1,m)\ast _{C_m} B(n_2,m)$} with $1\leq n_1\leq n_2, m\geq
6$. Then $\mu =\frac{p^{n_2}-p^{n_2-n_1}-1}{p^{n_2}m}$. This is
$<\frac{1}{12}$ for $m\geq 12$ and some cases with smaller $m$.\\

\noindent $A(viii)$. $p=3$, {\small ${\rm PSL}_2(\mathbb{F}_9)\ast_{B(2,4)}B(2\cdot n_2,4)$}
            with $n_2\geq 2$  has $\mu =\frac{1}{40}-\frac{1}{4.3^{2n_2}}$.\\

\noindent $A(ix)$. $p=2$, {\small $2\not|\ell$, $B(n_1,\ell)\ast_{C_\ell}D_\ell$}
          with $\ell| 2^{n_1}-1$; $n_1\geq 2$ has
          $\mu=\frac{2^{n_1-1}-1}{2^{n_1}\ell}$ and this is
 $<\frac{1}{12}$ for $\ell\geq 7$.\\

\noindent $A(xi)$. $p=2$, $q>2$, {\small
  $D_{q+1}\ast_{C_{q+1}}{\rm PGL}_2(\mathbb{F}_q)
           \ast_{B(n,q-1)}B(n\cdot n_2,q-1)$}
            with $n_2\geq 2$ has $\mu
            =\frac{q^{n_2}-q-1}{q^{n_2}(q^2-1)}$ and
this is $<\frac{1}{12}$ for $q>8$.\\

\noindent {\bf (B) three branch points and Proposition \ref{4.8}}, $p>2$, $(p,\ell)=1$: \\
$B(i)$. {\small $D_\ell \ast_{C_\ell}B(n_3,\ell)$} with  $\ell >5$ has $\mu=
\frac{p^{n_3}-2}{2\ell p^{n_3}}<\frac{1}{12}$.\\

\noindent $B(ii)$. {\small $D_{q+1}\ast_{C_{q+1}}{\rm
  PGL}_2(\mathbb{F}_q)\ast_{B(\mathbb{F}_q)}B(n\cdot m,q-1)$} has
$\mu =\frac{q^m-2}{2(q-1)q^m}$. This is $<\frac{1}{12}$ for $q\geq 7$.\\

\noindent $B(iii)$. {\small  $D_{(q+1)/2}\ast_{C_{(q+1)/2}}{\rm
  PSL}_2(\mathbb{F}_q)\ast_{B(\mathbb{F}_q)}B(n\cdot m,(q-1)/2)$} has
 $\mu=\frac{q^m-2}{q^m(q-1)}$. This is 
$<\frac{1}{12}$ for $q\geq 13$.\\

}\end{lists}

\begin{remark}\label{remark-6.4}{\bf The branch groups for $\Gamma$
    with $\mu(\Gamma)<\frac{1}{12}$. }  {\rm \\
The methods used in the proof of theorem 5.3 in \cite{P-V} allow one to determine the branch points
and hence the branch groups of $\Gamma$. They correspond to the branch points of the $\Gamma_v$, $v\in T^c$
that do not correspond to an edge $e\ni v$.

If $v\in T^c$ is not extremal, then $p\mid \#\Gamma_v$, since $\mu(\Gamma)<\frac{1}{12}$.
Therefore $\Gamma_v$ has two branch points, since ${\rm maxp}=0$.
The corresponding branch groups stabilize the edges $e\ni v$ and
therefore the associated ramification points are not contained in $\Omega$.
In particular, the branch groups of $\Gamma_v$ do not contribute to the branch groups of $\Gamma$
if the vertex $v$ is not extremal.  
The branch groups of an amalgam $\Gamma$ with $\mu(\Gamma)<\frac{1}{12}$ are those branch groups
of the two groups $\Gamma_v$, with $v\in T^c$ an extremal vertex,
that do not stabilize an edge $e\ni v$.\\

Let $B=B(s,\ell)$ and $B^\prime=B(s^\prime,\ell^\prime)$ denote groups of Borel type 
with $s,s^\prime\geq 1$ and let integers $\ell$ denote cyclic groups of order $\ell$.
The only cyclic groups $C_\ell$ with $p\nmid \ell$,
that occur are those with $\ell=2, q+1$ and $\frac{q+1}{2}$.

By inspecting the lists \ref{lists} above, one verifies that
the ramification indices and ramification groups for amalgams $\Gamma$ with
$\mu(\Gamma)<\frac{1}{12}$ are as follows:
\begin{small}
\begin{itemize}
\item  $(2,2,|B|)$ for $D_\ell \ast_{C_\ell}B(n,\ell)$, $D_{q+1}\ast_{C_{q+1}}{\rm PGL}_2(\mathbb{F}_q)$,\\
 $ D_{q+1}\ast_{C_{q+1}}{\rm PGL}_2(\mathbb{F}_q)\ast_{B(\mathbb{F}_q)}B(n\cdot m,q-1)$,
 $D_{(q+1)/2}\ast_{C_{(q+1)/2}}{\rm PSL}_2(\mathbb{F}_q)$,\\
 $D_{(q+1)/2}\ast_{C_{(q+1)/2}}{\rm PSL}_2(\mathbb{F}_q)\ast_{B(\mathbb{F}_q)}B(n\cdot m,(q-1)/2)$
  with $p>2$.
\item  $(2,|B|)$ for $D_\ell \ast_{C_\ell}B(n,\ell)$,
$ D_{q+1}\ast_{C_{q+1}}{\rm PGL}_2(\mathbb{F}_q)$,\\
 $ D_{q+1}\ast_{C_{q+1}}{\rm PGL}_2(\mathbb{F}_q)\ast_{B(\mathbb{F}_q)}B(n\cdot m,q-1)$
  with $p=2$.
\item  $(q+1,|B|)$ for ${\rm PGL}_2(\mathbb{F}_q)\ast_{B(n,q-1)}B(n\cdot m,q-1)$.
\item $(\frac{q+1}{2},B)$ for ${\rm PSL}_2(\mathbb{F}_q)\ast_{B(n,\frac{q-1}{2})}
                           B(n\cdot m,\frac{q-1}{2})$ with $p>2$.
 \item $(|B|,|B^\prime|)$ for the remaining groups $\Gamma$ with
$\mu(\Gamma)<\frac{1}{12}$.
\end{itemize}
 \end{small}
}
\end{remark}

\subsection{\rm The determinant of an amalgam}\label{section 6.3}
The vertex groups $\Gamma_v$ of a realizable amalgam
$\Gamma$ with $\mu(\Gamma)<\frac{1}{12}$ belong to 
$\{ {\rm PGL_2}({\mathbb F}_q), {\rm PSL_2}({\mathbb F}_q), D_{\ell}, B(n\cdot n_1, \ell)|\;
\ell|q\pm 1 \}$.
The amalgam $\Gamma$ admits a determinant map 
$\det: \Gamma\longrightarrow \mathbb{F}_q^\ast
/(\mathbb{F}_q^\ast)^2$. The unit element of this group of two
elements is written as $(\mathbb{F}_q^\ast )^2$.

The restriction of the determinant map to a group 
$\Gamma_v\subset  {\rm PGL_2}({\mathbb F}_q)$ is the usual determinant map.
If $\ell |\frac{q\pm 1}{2}$, then the group $D_\ell$ has an embedding into ${\rm PGL}_2(\mathbb{F}_q)$
such that $\det(D_\ell)= (\mathbb{F}^\ast_q)^2$ and another
embedding where this does not hold.
We have defined the determinant on the amalgam in such a way that 
$\det(D_\ell)= (\mathbb{F}^\ast_q)^2$ holds for $\ell|\;\frac{q\pm 1}{2}$. 
Note that the group $B(n\cdot n_1, q\pm 1)$ with $n_1>1$ cannot be embedded into
the group ${\rm PGL}_2(\mathbb{F}_q)$,
even though the determinant map with values in $\mathbb{F}_q^\ast$ (and
in $\mathbb{F}_q^\ast/(\mathbb{F}_q^\ast)^2$) is well defined.

If the kernel  $\Gamma^\prime$ of the determinant map is different from $\Gamma$, 
then $\Gamma^\prime\subset\Gamma$ has index two
and is again a realizable amalgam.
The vertex groups $\Gamma_v^\prime$ for $v\in T^c$ are contained in the set 
$\{ {\rm PSL_2}({\mathbb F}_q), D_{\ell}, B(n\cdot n_1, \ell)|\; \ell|\frac{q\pm 1}{2}\}$.\\

Let us now assume that $\Gamma^\prime\not=\Gamma$.
Since $\Gamma^\prime\subset \Gamma$ has finite index,
realizations of both groups act discontinuously on the same space of ordinary points $\Omega$.
Let $\Delta\subset \Gamma^\prime$ be a Schottky group, normal and of
finite index. Suppose that $\Delta$ is also normal in $\Gamma$. Then
the automorphism group ${\rm Aut}(X)$ of $X:=\Omega /\Delta$ is at least
$\Gamma/\Delta$ and thus larger than $\Gamma^\prime/\Delta$. If
${\rm Aut}(X)$ is larger than $\Gamma/\Delta$, then $\Gamma$ must be a
proper subgroup of finite index in another amalgam with $\mu
<\frac{1}{12}$. The Lists \ref{lists} can be used to verify this.\\

The above explains the examples in Lists \ref{lists}, namely:
$A(i)$ and $A(iii)$, $A(ii)$ and $A(iv)$, $A(v)$ with itself (and
different parameters), $B(i)$ with itself (and different parameters),
$B(ii)$ and $B(iii)$.\\

Further the amalgams $A(i)$, $A(ii)$, $B(i)$ with suitable parameters
and $B(ii)$, produce the full group ${\rm Aut}(X)$. For the other cases
$A(iii)$, $A(iv)$ et cetera, there are $X$ such that these groups do
not produce the full group ${\rm Aut}(X)$. This happens when, for instance,
$\Gamma' \subset \Gamma$ has index 2. Let $g$ be an element in
$\Gamma \setminus \Gamma'$. If $\Delta$ is any normal Schottky subgroup
of $\Gamma'$ of finite index, then $\Delta ':=\Delta \cap g\Delta g^{-1}$ is a
normal Schottky subgroup of finite index for both $\Gamma'$ and
$\Gamma$. This produces the required example.
 
We note, in passing, that it is  likely that $\Gamma'$ has a normal Schottky
subgroup of finite index $\Delta$ which is not normal in $\Gamma$ (and
so $\Delta \neq \Delta'$).

\subsection{\rm Constructing normal Schottky subgroups of finite
  index}\label{section 6.4}
Let $\varphi :\Gamma \rightarrow H$ be a homomorphism of a realizable
amalgam to a finite group $H$. Suppose that the restriction of $\varphi$
to each vertex group $\Gamma_v\subset \Gamma$ is injective. Then
$\ker(\varphi)$ is a normal Schottky group of finite index. Indeed, let
$a\in \ker(\varphi)$ have finite order. Then $a$ is conjugated to an
element in some $\Gamma_v$. Then $a=1$ since the restriction of $\varphi$
to each $\Gamma_v$ is injective. 

  Conversely, if $\Delta\subset \Gamma$ is a normal Schottky group of
  finite index, then the restriction of the canonical homomorphism
 $\varphi :\Gamma \rightarrow  H:=\Gamma/\Delta$ to each vertex group $\Gamma_v$ is injective.

\begin{lemma}\label{6.4}
Let $\Gamma$ be a realizable amalgam.
Let $m$ be the smallest integer such that all groups $\Gamma_v$, $v\in T^c$
can be embedded into the group ${\rm PGL}_2(\mathbb{F}_{p^m})$.
Then there exists a group homomorphism $\varphi: \Gamma\longrightarrow {\rm PGL_2}({\mathbb F}_{p^m})$
such that the kernel $ker(\varphi)$ contains no elements of finite order.
\end{lemma}

\begin{proof} Consider the case $\Gamma=G_1\ast_{G_3}G_2$. One starts
  with an injective homomorphism $\varphi _1:G_1\rightarrow {\rm
    PGL}_2(\mathbb{F}_{p^m})$. By assumption there is also an
  embedding  $\varphi_2: G_2\rightarrow  {\rm
    PGL}_2(\mathbb{F}_{p^m})$. One has to show that $\varphi_2$ can be
  chosen such that the restriction of $\varphi_2$ to $G_3$ coincides with
  the restriction of $\varphi_1$. By studying the cases presented in
  Proposition 2.1 one concludes that $\varphi_2$ exists.  The general
  case of the lemma is done by induction on the number of vertex
  groups $\Gamma_v,\ v\in T^c$ of $\Gamma$.
\end{proof}

\begin{proposition}\label{6.5}
Let $\Gamma^\prime\subset\Gamma$ be realizable amalgams 
and let $T^c$ be the tree of groups belonging to $\Gamma$.
We assume that $\Gamma^\prime\subset\Gamma$ is normal and of finite index.
We moreover assume that for $v\in T^c$ the intersection $\Gamma_v\cap\Gamma^\prime$ is non-cyclic
whenever $\Gamma_v$ is non-cyclic.
Let $m$ be the smallest integer such that all vertex groups $\Gamma_v$, $v\in T^c$
can be embedded into ${\rm PGL}_2(\mathbb{F}_{p^m})$. 

Let $\varphi^\prime:\; \Gamma^\prime\longrightarrow {\rm PGL}_2(\mathbb{F}_{p^m})$ be a
group homomorphism
and such that the kernel $\Delta^\prime$ contains no elements of finite order.
Then there exists a group homomorphism
$\varphi:\; \Gamma\longrightarrow {\rm PGL}_2(\mathbb{F}_{p^m})$ 
such that $\varphi|_{\Gamma^\prime}=\varphi^\prime$ 
whose kernel $ker(\varphi)=\Delta$ 
contains no elements of finite order.
Then $\Delta^\prime\subset\Delta$ and
moreover, $\Delta=\Delta^\prime$ if and only if
$[im(\varphi):im(\varphi^\prime)]=[\Gamma:\Gamma^\prime]$.
\end{proposition}

\begin{proof}
Let $T^c$ denote the tree of groups for $\Gamma$ and $(T')^c$ denote the one for $\Gamma'$.
Since $\Gamma^\prime\subset\Gamma$ is normal,
the subgroup $\Gamma_v\cap\Gamma^\prime\subset \Gamma_v$ 
is a normal subgroup for each vertex $v\in T^c$.
Since $\Gamma_v$ normalizes $\Gamma_v\cap\Gamma^\prime$
the group $\Gamma_v$ can be embedded into the normalizer
of $\varphi^\prime(\Gamma_v\cap\Gamma^\prime)$ in ${\rm PGL}_2(\mathbb{F}_{p^m})$.
Moreover, it follows from the condition that
$\Gamma_v\cap\Gamma^\prime$ is non-cyclic if $\Gamma_v$ is non-cyclic,
that the tree of groups obtained from $T^c$ by replacing the groups $\Gamma_v$
by the intersections $\Gamma_v\cap\Gamma^\prime$ equals the tree $(T')^c$.
If $v,v^\prime\in T^c$ are vertices that form an edge $e$,
then we can embed both $\Gamma_v$ and $\Gamma_{v^\prime}$
in such a way that the embedding of $\Gamma_e$ 
contains $\varphi^\prime(\Gamma_e\cap \Gamma^\prime)$.
By induction we can extend this to all vertices $v\in T^c$.
This gives a well-defined group homomorphism $\varphi$
that extends the homomorphism $\varphi^\prime$.

By construction the kernel $\Delta$ of $\varphi$ contains no elements of finite order.
Since its restriction to $\Gamma^\prime$ equals $\varphi^\prime$
its kernel contains $\Delta^\prime$.
The equality $[\Delta :\Delta^\prime]\cdot [im(\varphi): im(\varphi^\prime)]=[\Gamma : \Gamma^\prime]$
implies that $\Delta=\Delta^\prime$ 
if and only if $[im(\varphi): im(\varphi^\prime)]=[\Gamma : \Gamma^\prime]$.
\end{proof}

\begin{proposition}\label{6.6}
Let $\Gamma$ be a realizable amalgam with $\mu(\Gamma)<\frac{1}{12}$.
Let $m$ be the smallest integer such that all vertex groups $\Gamma_v$, $v\in T^c$
can be embedded into ${\rm PGL}_2(\mathbb{F}_{p^m})$. Suppose that the
kernel of the group homomorphism \\ 
$\varphi:\; \Gamma\longrightarrow {\rm PGL}_2(\mathbb{F}_{p^m})$
contains no elements of finite order.\\
{\em Suppose $p>2$}.
 If $\det(\Gamma)= (\mathbb{F}_{p^m}^\ast)^2$ and $\Gamma_v\not= D_\ell$, $\ell|\frac{q\pm 1}{2}$
for all $v\in T^c$, then $im(\varphi)={\rm PSL}_2(\mathbb{F}_{p^m})$. 
 If  $\det(\Gamma)= (\mathbb{F}_{p^m}^\ast)^2$ and $\Gamma_v= D_\ell$, $\ell|\frac{q\pm 1}{2}$
for some vertex $v\in T^c$, then both possibilities  ${\rm PGL}_2(\mathbb{F}_{p^m})$
or ${\rm PSL}_2(\mathbb{F}_{p^m})$ for $im(\varphi)$ occur.\\  In the other cases  $im(\varphi)={\rm PGL}_2(\mathbb{F}_{p^m})$.\\
{\em Suppose $p=2$}, then $im(\varphi)={\rm PGL}_2(\mathbb{F}_{p^m})$.
\end{proposition}

\begin{proof}
Using the Lists \ref{lists} one has that $\mu(\Gamma)<\frac{1}{12}$ implies that
the order of at least one of the groups $\Gamma_v$, $v\in T^c$
is divisible by $p$. Therefore the order of $im(\varphi)$ is divisible by $p$.

Let us first show that the group $im(\varphi)$ is not contained in a group of Borel type.
If $\Gamma_v\cong D_\ell$ or $\Gamma_v\cong {\rm PGL}_2(\mathbb{F}_{q})$ for a vertex $v\in T^c$,
then $im(\varphi)$ contains a subgroup isomorphic to $\Gamma_v$.
In particular, the image $im(\varphi)$ is not contained in a group of Borel type.

By Lists \ref{lists}, the remaining case is $\Gamma=B(n_1,\ell)\ast_{C_\ell}B(n_2,\ell)$.
In this case a generating element $t\in C_\ell$ acts 
as $t$ on the $p$-part of one of the groups $\Gamma_v$
and as $t^{-1}$ on the $p$-part of the other.
Since $\varphi$ is a homomorphism of groups, 
it follows that the image of both groups is not embedded in a single group of Borel type if $t\not=-1$.
Hence the image $im(\varphi)$ is not contained in a group of Borel type if $t\not= -1$.
If $t=-1$, then $\ell=2$ and $\mu(\Gamma)\geq \frac{1}{6}$.
Hence the case $t=-1$ does not occur.

If $p\not= 2,3$, then the only subgroups of ${\rm PGL}_2(\mathbb{F}_{p^m})$
of order divisible by $p$ are the groups ${\rm PSL}_2(\mathbb{F}_{p^s})$,
${\rm PGL}_2(\mathbb{F}_{p^s})$ with $s|m$ and groups of Borel type.
Since $m$ is the smallest integer such that all groups $\Gamma_v$ can be embedded into
${\rm PGL}_2(\mathbb{F}_{p^m})$, the image $im(\varphi)$ cannot be a group 
${\rm PSL}_2(\mathbb{F}_{p^s})$ or ${\rm PGL}_2(\mathbb{F}_{p^s})$ with $s<m$.
Now the proposition for $p\not=2,3$ follows from the fact that the group homomorphism
$\varphi$ preserves the determinant except for the groups $D_\ell$ with $\ell|\frac{q\pm 1}{2}$.

Using Lists \ref{lists}, we  exclude for $p=3$ the possibility  $im(\varphi)=A_5$ and we 
exclude for $p=2$ the possibility  $im(\varphi)=D_\ell$ with odd
$\ell$. Hence $im(\varphi)$ is also ${\rm PSL}_2(\mathbb{F}_{p^m})$ or ${\rm PGL}_2(\mathbb{F}_{p^m})$
for $p=2,3$. 
\end{proof}

\begin{example}\label{6.7} {\rm
Let $\Gamma$ be a realizable amalgam with
$\mu(\Gamma)<\frac{1}{12}$. By \ref{6.4} and \ref{6.6} there exists a
surjective homomorphism $\varphi:\Gamma\rightarrow H$ with $H$ either
${\rm PGL_2}(\mathbb{F}_q)$ or ${\rm PSL_2}(\mathbb{F}_q)$ and $\ker(\varphi)$ a Schottky group.
{\it Here we give  a series of examples of other constructions of normal Schottky subgroups 
$\Delta \subset \Gamma$ of finite index}.\\

\noindent (1). Let $\Gamma:=\Gamma_{v_1}\ast_{\Gamma_e}\Gamma_{v_2}:=
{\rm PGL}_2(\mathbb{F}_{q})\ast_{B(n,q-1)}B(n\cdot d,q-1)$
and let $H:={\rm PGL}_2(\mathbb{F}_{q})^d$ with $d>1$.
There exists a group homomorphism $\varphi:\;\Gamma\longrightarrow H$,
such that the kernel $ker(\varphi)$ is a Schottky group. Indeed, define
$\varphi$ for $g\in \Gamma_{v_1}$ by $\varphi(g)=(g,g,\dots,g)\in {\rm PGL}_2(\mathbb{F}_q)^d$. 
The $p$-part $B(n\cdot d,1)$ of the group $B(n\cdot d,q-1)$ is written
as $\{{1\ v\choose 0\ 1 }|v=(a_1,\dots,a_d)\in \mathbb{F}_q^d\}$ such
that the $p$-part of $\Gamma_e$ is $\{{1\ v\choose 0\ 1 }|v=(a_1,0,\dots,0)\in \mathbb{F}_q^d\}$.
Finally define $\varphi$ by the formula $\varphi ({1\ v\choose 0\ 1 })=({1\ a_1\choose 0\ 1 },{1\
  a_1+a_2\choose 0\ 1},\dots ,{1\ a_1+a_d\choose 0\ 1})$.
 One easily verifies that $im(\varphi)$ is\\
 $\{ g=(g_1,\ldots,g_d)\in H={\rm PGL}_2(\mathbb{F}_{q})^d
|\; \det(g_i)=\det(g_j), 1\leq i<j\leq d \} $.\\

\noindent (2). Since $B(d\cdot n, 1)=\prod_{i=1}^s B(d_i\cdot n,1)$ with $\sum_{i=1}^s d_i=d$, 
this construction can be generalised by replacing the group $H$ by
the group \\ $H^\prime:=\prod_{i=1}^s  {\rm PGL}_2(\mathbb{F}_{q^{d_i}})$ with $\sum_{i=1}^s d_i=d$.\\

\noindent (3). A similar $\varphi$ exists for 
$\Gamma:=\Gamma_{v_1}\ast_{\Gamma_e}\Gamma_{v_2}:=
{\rm PGL}_2(\mathbb{F}_{q})\ast_{C_{q+1}}B(2n\cdot d,q+1)$.
Let $H$ be the group $H={\rm PGL}_2(\mathbb{F}_{q^2})^d$.
The $p$-part of the group $\Gamma_{v_2}$ equals
$B(2n\cdot d,1)\cong B(2n,1)^d$.
The group homomorphism $\varphi$ embeds the group 
$\Gamma_{v_1}={\rm PGL}_2(\mathbb{F}_{q})$ diagonally in the group $H$.
Moreover, $\varphi(\Gamma_{v_2})$ embeds the $p$-part as a group $B(2n,1)^d$
that is normalised by $\varphi(\Gamma_e)=C_{q+1}$.
The image of $\varphi$ equals $\varphi(\Gamma)={\rm PSL}_2(\mathbb{F}_{q^2})^d$.\\

\noindent (4). For $p=2$ and odd $\ell$ there exist  surjective group homomorphisms 
$\varphi:\; D_\ell\ast_{C_2} C_2^n\longrightarrow D_\ell\times
C_2^{n-1}$ such that the kernel is a Schottky group.\\

\noindent (5). 
 For $p=3$ there exists a surjective group homorphisms $\varphi:\; A_4\ast_{C_3} C_3^n\longrightarrow A_4\times C_3^{n-1}$
such that the kernel is a Schottky group.}
\end{example}

\section{\rm  Mumford curves with many automorphisms} \label{section 7}

In \S \ref{section 6} we have shown that there exist many amalgams $\Gamma$
with $\mu(\Gamma)<\frac{1}{12}$ that give rise to Mumford curves $X$
such that ${\rm Aut}(X)={\rm PGL}_2(\mathbb{F}_q)$.
In the theorem below we determine the minimal genus for Mumford curves
with automorphism group ${\rm PGL}_2(\mathbb{F}_q)$.

In \S \ref{section 7.1} and \S \ref{section 7.2} we describe the Schottky groups and the Mumford curves
obtained in some detail.

\begin{theorem}\label{6.8} \label{7.1}
Let $X=\Omega/\Delta$ be a Mumford curve with  ${\rm Aut}(X)=\Gamma/\Delta
={\rm PGL_2(\mathbb{F}_{q})}$, where $\Gamma$ is the normalizer of
$\Delta$ in ${\rm PGL_2(K)}$. Then the genus $g$ of $X$ satisfies $g\geq \frac{q(q-1)}{2}$. Equality
holds in precisely the following cases:\\ 

\indent $q=4$ and $\Gamma={\rm PGL}_2(\mathbb{F}_2)\ast_{C_2}B(2,1)$,\\
\indent $q=p^n>2$ and $\Gamma={\rm PGL}_2(\mathbb{F}_q)\ast _{C_{q+1}}D_{q+1}$,\\
\indent $q=p^n>2$ and  $\Gamma=D_{q-1}\ast_{C_{q-1}}B(n,q-1)$.\\

The branch groups are
$C_3$ and $B(2,1)$ for $\Gamma={\rm PGL}_2(\mathbb{F}_2)\ast_{C_2}B(2,1)$ 
and for the other amalgams the branch groups are
$C_2$, $C_2$ and $B(n,q-1)$ if $p>2$ and $C_2$ and $B(n,q-1)$ if $p=2$.
\end{theorem}

\begin{proof}
For a fixed $q$ we have to determine the realizable amalgams $\Gamma$ admitting a
surjective homomorphism $\varphi: \Gamma \rightarrow {\rm
  PGL}_2(\mathbb{F}_q)$ such that the kernel is a Schottky group (use \ref{6.6}) and
minimal $\mu (\Gamma)<\frac{1}{12}$ (which corresponds to a minimal genus
$g$ for $X$). 
The value $\mu(\Gamma)$ is entirely determined by the branch groups of the amalgam $\Gamma$.
The rather short list of the branch groups that occur (see remark \ref{remark-6.4})
allows us to find the minimal value $\mu(\Gamma)$.
We only need to determine the combinations of branch groups that give the minimal value for $\mu(\Gamma)$. 

\bigskip

Assume ${\rm Aut}(X)={\rm PGL}_2(\mathbb{F}_{p^m})$ and $\mu(\Gamma)<\frac{1}{12}$.
Then $\Gamma$ has two or three branch points and  at least one of the branch groups is of Borel type
(see \ref{remark-6.4}).
 
As before (proof of \ref{6.2}), we  use the Riemann--Hurwitz--Zeuthen formula 
for $X:=\Omega /\Delta\rightarrow \mathbb{P}^1=\Omega /\Gamma$ and
$g-1=[\Gamma :\Delta]\cdot \mu(\Gamma)$ to calculate the (minimum) value of the function $\mu$.
Now $\mu(\Gamma)=-1+(\sum_i c_{G_i})/2$.
Here $c_{G_i}:=\frac{(\ell_i+1)p^{n_i}-2}{\ell_ip^{n_i}}$
is the contribution of the branch group $G_i$ of order $\ell_ip^{n_i}$ with
$(\ell_i,p)=1$ to the value of $\mu(\Gamma)$.

The contribution $c_B$ of a single group of Borel type $B=B(s,\ell)$
to the value of $\mu(\Gamma)$ equals
$c_B=\frac{(1+\ell)p^s-2}{\ell p^s}
=1+\frac{p^s-2}{\ell p^s}$.
This contribution $c_B$ is minimal when 
either $p^s=2$ or $p^s>2$ and the values of $\ell$ and $p^s$ are both maximal.
If $p^s=2$, then $B=B(1,1)$ and $c_B=1$.
The only infinite group that has only branch groups $B=B(1,1)$ for $p=2$
is the amalgam $\Gamma={\rm PGL}_2(\mathbb{F}_2)\ast_{C_3}{\rm PGL}_2(\mathbb{F}_2)
\cong D_3\ast_{C_3}D_3$. Indeed, the amalgams with two branch groups $C_2$ are 
$D_\ell \ast_{C_\ell}D_\ell$ with $2\nmid \ell$. Only $\ell=3$ can occur in view of the lists 6.3.
This amalgam produces a (Tate) curve of genus $g=1$.
Therefore the case with only branch groups $B=B(1,1)$ with $p=2$ is excluded.\\

The next case to consider are the groups $\Gamma$ for $p=2$ 
with a single branch group $B(1,1)$ and a branch group different from $B(1,1)$.
Such $\Gamma$ satisfy $\mu(\Gamma)\geq \frac{1}{12}$ or contain a dihedral group.
Those containing a dihedral group will be treated later in this proof.\\
  
It follows that we have only to consider the case
of amalgams $\Gamma$ with a branch group $B(s,\ell)\subset{\rm PGL}_2(\mathbb{F}_{p^m})$, 
where $\ell$ and $p^s>2$ are both maximal for the particular types of ramification indices. \\

If $\Gamma$ has three branch points then $p>2$ and the branch groups are 
$C_2$, $C_2$ and $B(s,\ell)$ with $s<m$ and $\ell|p^m-1$.
Therefore $B=B(m,p^m-1)$ and $c_B=1+\frac{p^m-2}{p^m(p^m-1)}$ gives the minimum value of $\mu(\Gamma)$
for this type of amalgam.
Then the minimal value of $\mu(\Gamma)$ 
for such a group equals \\
$\mu(\Gamma)=-1+(2c_{C_2}+c_B)/2=-1+(1+1+\frac{p^m-2}{p^m(p^m-1)})/2 =\frac{p^m-2}{2p^m(p^m-1)}$.

The amalgams for $p=2$, $p^m>2$ with branch groups $C_2$ and $B(m,p^m-1)$
give exactly the same minimum value of $\mu(\Gamma)$. \\

If $\Gamma$ has two branch points then either both branch groups are of Borel type
or one branch group equals a cyclic group $C_{q+1}$ or $C_{\frac{q+1}{2}}$
and the other equals a group of Borel type.
Let us first consider the case where both branch groups are of Borel type.
If $p>2$, then we only need to consider the case 
where both groups of Borel type are maximal groups $B=B(m,p^m-1)$.
Then $\mu(\Gamma)=-1+c_B=\frac{p^m-2}{p^m(p^m-1)}$.
This is a factor two larger than the previous case.
Thus  amalgams of this type do not obtain the minimal value of $\mu(\Gamma)$.
If $p=2$, then the minimal value is obtained by taking $B=C_2$
and $B^\prime=B(m,p^m-1)$. This situation has already been considered above.\\

Consider the case with two branch points, where one branch group equals $C_{q+1}$ and the other is a group of Borel type.
These branch groups correspond to amalgams $\Gamma$ of the form
$\Gamma={\rm PGL}_2(\mathbb{F}_{q})\ast_{B(n,q-1)}B(n_1\cdot n,q-1)$.
The case with minimal contribution $c_B$ to the value of $\mu(\Gamma)$
occurs when $n_1=2$. Then $p^m=q^2=p^{2n}$ and $B=(2n,p^n-1)$.
Then $c_B=1+\frac{p^{2n}-2}{p^{2n}(p^{n}-1)}$
and $c_{C_{q+1}}=\frac{q}{q+1}=\frac{p^{n}}{p^{n}+1}$.
We obtain
$\mu(\Gamma)
=-1+(\frac{p^{n}}{p^{n}+1}+1+\frac{p^{2n}-2}{p^{2n}(p^{n}-1)})/2
=\frac{p^{2n}-p^{n}-1}{p^{2n}(p^{2n}-1)}$.
Since $p^{2n}-p^{n}-1\geq\frac{p^{2n}-2}{2}$,
the inequality 
$\mu(\Gamma)=\frac{p^{2n}-p^{n}-1}{p^{2n}(p^{2n}-1)}\geq\frac{p^{2n}-2}{2p^{2n}(p^{2n}-1)}$ holds.
Equality holds if and only if $p^n=2$.
Then the amalgam equals $\Gamma={\rm PGL}_2(\mathbb{F}_2)\ast_{C_2}B(2,1)$
and has branch groups $C_3$ and $B(2,1)$.

The case where one branch group is the cyclic group $C_{\frac{q+1}{2}}$ and the other branch group 
is of Borel type is similar. It does not result in an amalgam 
$\Gamma$ with $\mu(\Gamma)\leq \frac{p^m-2}{2p^m(p^m-1)}$.

We have now treated all the amalgams $\Gamma$ with $\mu(\Gamma)<\frac{1}{12}$.
The minimal value of $\mu(\Gamma)$ that occurs for a Mumford curve 
$X=\Omega/\Delta$ with automorphism group
$\Gamma/\Delta={\rm PGL}_2(\mathbb{F}_{p^m})$ is
$\mu(\Gamma)=\frac{p^m-2}{2p^m(p^m-1)}$.
Then $\Gamma$ has branch groups $C_2$, $C_2$, $B(m,p^m-1)$ if $p>2$,
branch groups $C_2$ and $B(m,p^m-1)$ if $p=2$, $p^m>2$
or branch groups $C_3$ and $B(2,1)$ with $p^m=4$.
The amalgams with branch groups $C_2$, $C_2$ and $B(m,p^m-1)$ with $p>2$
and branch groups $C_2$ and $B(m,p^m-1)$ with $p=2$ and $p^m>2$
are $\Gamma=D_{q-1}\ast_{C_{q-1}}B(n,q-1)$
and $\Gamma=D_{q+1}\ast_{C_{q+1}}{\rm PGL}_2(\mathbb{F}_q)$ with $q=p^m>2$.
For $q=p^m=4$ we also have the amalgam
$\Gamma={\rm PGL}_2(\mathbb{F}_2)\ast_{C_2}B(2,1)$
with branch groups $C_3$ and $B(2,1)$.
\end{proof}

\begin{remark}\label{7.2} {\rm $\ $ \\
The extreme amalgams ${\rm PGL}_2(\mathbb{F}_q)\ast _{C_{q+1}}D_{q+1}$
and $D_{q-1}\ast_{C_{q-1}}B(n,q-1)$ are families in two
ways. First of all $q=p^m$ varies and secondly for a fixed amalgam the
embeddings as a discontinuous group in ${\rm PGL}_2(K)$ are
parametrised by a punctured open disk $\{\lambda \in K|\
0<|\lambda|<1\}$ (see \S \ref{section 7.1} and \S \ref{section 7.2}).  
In \S \ref{section 8} we will show that the above two extreme families are
the only ones that have a maximal number of automorphisms. \\ 
}\end{remark}

\subsection{\rm The amalgam ${\rm PGL}_2(\mathbb{F}_q)\ast
  _{C_{q+1}}D_{q+1}$.}\label{section 7.1}

In this section the Schottky group and the equation for the extreme
family of Mumford curves  belonging to the amalgam  ${\rm PGL}_2(\mathbb{F}_q)\ast_{C_{q+1}}D_{q+1}$ 
are made explicit.

\begin{lemma}\label{lemma-6.10} \label{7.3} Suppose $q=p^n>2$. Then $\Gamma:={\rm
    PGL}_2(\mathbb{F}_q)\ast _{C_{q+1}}D_{q+1}$ has a unique normal
 subgroup $\Delta$, free of rank $\frac{q(q-1)}{2}$, such that 
$\Gamma/\Delta ={\rm PGL}_2(\mathbb{F}_q)$. The embeddings of $\Gamma$
as discontinuous subgroup of ${\rm PGL}_2(K)$ are parametrised by the
punctured disk $\{\lambda \in K|\ 0<|\lambda |<1\}$.  The Mumford curve $X_\lambda$
defined by the embedding and $\Delta$ has genus $\frac{q(q-1)}{2}$ and group of automorphisms  
${\rm PGL}_2(\mathbb{F}_q)$. 
\end{lemma}

\begin{proof}
We choose in ${\rm PGL}_2(\mathbb{F}_q)$ an element $a$ of order $q+1$
and an element $z$ with $z^2=1,\ za=a^{-1}z$. Let $<a,z>$
denote the subgroup generated by $a$ and $z$ (we note that $<a,z>\cong D_{q+1}$).  
 In  $D_{q+1}$ we choose
an element $A$ of order $q+1$ and an element $Z$ with $Z^2=1,\
ZA=A^{-1}Z$. Then the amalgam $\Gamma$ is (since $A$ and $a$ are
identified) generated by  ${\rm  PGL}_2(\mathbb{F}_q)$ and $Z$  and the relations are $Z^2=1, \
Za=a^{-1}Z$. Choose a set of representatives $W$ of  the cosets ${\rm PGL}_2(\mathbb{F}_q)/<a,z>$.
We may assume that $1\in W$. We note that $\# W=\frac{q(q-1)}{2}$.\\ 

Every element in $\Gamma$ can be written in a unique way as a
``reduced word'' $w_1z^{\epsilon_1}Zw_2z^{\epsilon_2}Z\cdots
Zw_sz^{\epsilon_s}Zm$ with $w_1,\dots ,w_s\in W$, $\epsilon_1,\dots
\epsilon_s\in \{0,1\}$, $m\in {\rm PGL}_2(\mathbb{F}_q)$ and
$w_iz^{\epsilon_i}\neq 1$ for $i=2,\dots , s$. \\

 One defines the homomorphism $\varphi:\Gamma \rightarrow  {\rm
  PGL}_2(\mathbb{F}_q)$ by  $\varphi$ is the identity on  ${\rm
  PGL}_2(\mathbb{F}_q)$ and $\varphi(Z)=z$. The kernel $\Delta$ of $\varphi$ is the
smallest normal subgroup containing $\tau:=zZ$.  We note that $(w\tau w^{-1})^{-1}=w\tau^{-1}w^{-1}$,
 $a\tau
a^{-1}=\tau$, $z\tau z^{-1}=\tau^{-1}$ and $Z\tau Z^{-1}=\tau^{-1}$.

 One considers the set
$S:=\{w\tau w^{-1}|\ w\in W\}\cup \{w\tau^{-1}w^{-1}|\ w\in W\}$. This
set contains all conjugates $m\tau m^{-1}$ with $m\in {\rm
  PGL}_2(\mathbb{F}_q)\cup \{Z\}$. It follows that the kernel $\Delta$
is generated as a subgroup by the $\{ w\tau w^{-1}| w\in W\}$. Using
the unique representation of the elements of $\Gamma$ by ``reduced
words'' one finds that there are no relations among the above
generators of $\Delta$. Thus $\Delta$ is a free non-commutative 
 group on $\frac{q(q-1)}{2}$ generators. We note that $\varphi$ is
 unique up to conjugation and thus $\Delta$ is unique.\\

Every embedding $em: \Gamma \rightarrow {\rm PGL}_2(K)$ as
discontinuous group is (up to conjugation) given by $em$ is the
`identity' on ${\rm PGL}_2(\mathbb{F}_q)$ and $em(Z)$ is an element of
order two which permutes the two fixed points of $a$ and such that
$em(\tau)$ is a hyperbolic element with fixed points the fixed points
of $a$. Thus the embeddings form a family parametrised by
$\{\lambda \in K|\ 0<|\lambda |<1\}$.\\

Let an embedding $\Gamma \subset {\rm PGL}_2(K)$ be chosen and let 
$\Omega\subset \mathbb{P}^1(K)$ be the subspace of the ordinary points
for $\Gamma$. Then $X_\lambda :=\Omega /\Delta$ is a Mumford curve of genus 
$\frac{q(q-1)}{2}$ and its group of automorphisms is ${\rm PGL}_2(\mathbb{F}_q)$.
\end{proof}

\subsubsection {\rm  An algebraic description of the family of curves $X_\lambda$ for
   $p>2$} \label{section 7.1.1.0}
 {\it The first step} is a computation of the stable reduction of
$X_\lambda$.\\
 Using the tree $\mathcal{T}$, corresponding to the
 amalgam, one makes an analytic reduction, denoted by $\overline{\Omega}$, of $\Omega$. On this 
tree of projective lines over the residue field, the amalgam acts. 
Then $\overline{X}_\lambda:=\overline{\Omega}/\Delta$ is an analytic reduction of the curve
$X_\lambda$ (independent of $\lambda$). Using the description of
$\Delta$ one finds that one component $L$ of   $\overline{X}_\lambda$ is a
$\mathbb{P}^1$ with stabiliser ${\rm PGL}_2(\mathbb{F}_q)$ and the
other components $\{L_i\},\ i=1,\dots, \frac{q(q-1)}{2}$ are projective lines with
stabilisers $\cong D_{q+1}$. Each $L_i\cap L$ consists of two points
of $L(\mathbb{F}_{q^2})$ conjugated under the Frobenius $Fr_q$. The
stable reduction $R$ of $X_\lambda$ is obtained by contracting 
all lines $L_i$. Therefore $R$ is the projective line ``with
$\frac{q(q-1)}{2}$ ordinary nodes'', i.e., it is
$\mathbb{P}^1_{\mathbb{F}_q}$ where the pairs of points $\{\{a,a^q\}| a\in
\mathbb{F}_{q^2}\setminus \mathbb{F}_q\}$ are identified.
The curve $R$ is known as the Ballico-Hefez curve (See \cite{F-H-K} and \cite{H-S}).
Below (\ref{7.4} for $p\neq 2$ and \S \ref{section 7.1.1} for $p=2$)
we give an equation for  this curve (see also \cite{H-S} prop. 1.4).

\begin{lemma} \label{7.4} Suppose $p\neq 2$. Consider the stable curve $R$ over $\mathbb{F}_q$, defined by
  identifying on the projective line over $\mathbb{F}_q$ the points
  $a$ and $a^q$ for all $a\in \mathbb{F}_{q^2}$.  The
  homogeneous polynomial $F_0\in \mathbb{F}_q[x_1,x_2,x_3]$ of degree
  $q+1$ is defined by:\\
 $F_0=(x_1^2+x_2^2-wx_3^2)^{(q+1)/2}-(x_1^{q+1}+x_2^{q+1}-wx_3^{q+1})$ with
 $w\in \mathbb{F}_q^*\setminus (\mathbb{F}_q^*)^2$.

Then $F_0=0$ is an embedding of $R$ into the projective plane over
$\mathbb{F}_q$. Moreover, the group $G\subset {\rm PGL}_3(\mathbb{F}_q)$ of the automorphisms of the
projective plane having $F_0=0$ as invariant curve maps isomorphically
to the group ${\rm PGL}_2(\mathbb{F}_q)$ of automorphisms of $R$.
\end{lemma}
\begin{proof} {\it Case} (1), $q\equiv 3 \mod 4$ and we take $w=-1$. The group $G\subset {\rm
    PGL}_3(\mathbb{F}_q)$ of the automorphism preserving the quadratic form
  $x_1^2+x_2^2+x_3^3$ also preserves the unitary form
  $x_1^{q+1}+x_2^{q+1}+x_3^{q+1}$ and preserves $F_0=0$. This group is
  isomorphic to ${\rm PGL}_2(\mathbb{F}_q)$ since the quadratic form
  $x_1^2+x_2^2+x_3^3$ is equivalent to $xy-z^2$, the quadratic form
  for the 2-tuple embedding of the projective line into $\mathbb{P}^2$.

  The singular point $(1:0:0)$ of $F_0=0$ has a local equation of the form
  $(\frac{x_2}{x_1})^2+(\frac{x_3}{x_1})^2+$ higher order terms in
  $\mathbb{F}_q[[\frac{x_2}{x_1},\frac{x_3}{x_1}]]$. Thus the
  singularity is a node and the two tangent lines are not rational
  over $\mathbb{F}_q$. The stabilizer in $G$ of this node is seen to
  be $D_{q+1}$. The $G$-orbit of this node consists of
  $\frac{q(q-1)}{2}$ points. One verifies that there are no
  more singular points. It follows from the Pl\"ucker formulas that
 $F_0=0$ is an irreducible curve of geometric genus 0. The nodes of
 $F_0=0$ are identical to those of $R$ and the two groups of
 automorphisms coincide.\\
{\it Case} (2), $q\equiv 1\mod 4$. After changing the quadratic form
$x_1^2+x_2^2+x_3^2$ and the corresponding unitary form
$x_1^{q+1}+x_2^{q+1}+x_3^{q+1}$ into $x_1^2+x_2^2+wx_3^2$ (note that
$-1$ is a square) and  $x_1^{q+1}+x_2^{q+1}+wx_3^{q+1}$, the proof of Case (1) carries over.
\end{proof}

\begin{remark}\label{7.5} {\rm
We observe that, for any $\alpha \in \mathbb{F}_q^*$, the homogeneous equation  
$(x_1^2+x_2^2-\alpha^2x_3^2)^{(q+1)/2}-(x_1^{q+1}+x_2^{q+1}-\alpha^2x_3^{q+1})=0$ 
defines a set of $q+1$ projective lines intersecting normally. 
One can of course always take $\alpha=1$ and for $q\equiv 1\mod 4$ one can choose an $\alpha$
such that $\alpha^2=-1$.  
The stabilizer in $G\cong {\rm PGL}_2(\mathbb{F}_{q})$ of one of these lines is a subgroup isomorphic to
$B(n,q-1)$.
The lines are the ${\rm PGL}_2(\mathbb{F}_{q})$-orbit of the line through
the two points $a=(0,\alpha,1)$ and $b=(1,0,0)$. 
This situation will be considered in \S \ref{section 7.2}.
}
\end{remark}
 
\begin{corollary}\label{7.6}
 The family of Mumford curves of \ref{lemma-6.10} can be identified with the
 family of curves in $\mathbb{P}^2_K$ defined by $F=F_0+\lambda (x_1^2+x_2^2+\epsilon
 x_3^2)^{(q+1)/2}$ where $F_0$ is the polynomial of Lemma {\rm \ref{7.4}},\ 
 $\epsilon =1$ for $q\equiv 3\mod 4$ and $\epsilon =w $ for $q\equiv 1\mod 4$.
 Further $\lambda$ lies in the punctured open disk $\{a\in K|\ 0<|a|<1\}$.   
\end{corollary}

\begin{proof} $F$ is invariant under the action of the group $G\cong
  {\rm PGL}_2(\mathbb{F}_q)$ of the automorphisms of the quadratic
  form $(x_1^2+x_2^2+\epsilon x_3^2)$ over
  $\mathbb{F}_q$. One verifies that $F=0$ has no singularities. The
  obvious reduction of $F=0$ is the curve $F_0=0$ over $\mathbb{F}_q$.
Hence $F=0$ defines a Mumford curve with automorphism group   ${\rm
  PGL}_2(\mathbb{F}_q)$ and genus $\frac{q(q-1)}{2}$. The form of the
reduction shows that $F=0$ is an equation for a curve of Lemma \ref{lemma-6.10}. \end{proof}

\subsubsection{\rm The family of Mumford curves for 
${\rm PGL}_2(\mathbb{F}_q)\ast _{C_{q+1}}D_{q+1}$,
$p=2$.} \label{section 7.1.1}

We briefly recall the equation for the Ballico-Hefez curve for $q=2^n$, $n>1$, from \cite{H-S} prop. 1.4.
This curve describes the reduction of the Mumford curves belonging to the amalgam
${\rm PGL}_2(\mathbb{F}_q)\ast _{C_{q+1}}D_{q+1}$ with $q=2^n$, $n>1$.

Let the group ${\rm PGL}_2(\mathbb{F}_q)$, $q=2^n\geq 4$, act on the projective plane 
$\mathbb{P}^2_{\mathbb{F}_q}$ and preserve the quadric $x_0x_2-x_1^2=0$.
Then the group ${\rm PGL}_2(\mathbb{F}_q)$ preserves the hermitian curve $h(x)=0$ given by
$x_0x_2^q+x_2x_0^q+x_1^{q+1}=0$ and also preserves the curve $t(x)=0$ defined by
$x_1^{q+1}{\rm Tr}(\frac{x_0x_2}{x_1^2})
=x_1^{q+1}\sum_{i=0}^{n-1} \frac{x_0^{2^i}x_2^{2^i}}{x_1^{2^{i+1}}}=0$.
The restriction of the polynomial ${\rm Tr}$ to $\mathbb{F}_q$ is the trace map 
${\rm Tr}_{\mathbb{F}_q/\mathbb{F}_2}:\; \mathbb{F}_q\longrightarrow \mathbb{F}_2$.
Note that ${\rm Tr}(\frac{x_0x_2-x_1^2}{x_1^2})
={\rm Tr}(\frac{x_0x_2}{x_1^2}-1)
={\rm Tr}(\frac{x_0x_2}{x_1^2})-n$.

The Ballico-Hefez curve can now be defined as $h(x)-t(x)=0$.
This also defines the curve $F_0$ that forms the reduction of the
Mumford curves. The family of curves $Y_\lambda$ defined by the equations
$F_0+\lambda\cdot h(x)=0$ with 
$\lambda\in K$, $0<|\lambda|<1$, consists of the Mumford curves that
have the curve $F_0$ as their reduction.

\subsection{\rm The amalgam
  $\Gamma:=D_{q-1}\ast_{C_{q-1}}B(n,q-1)$.} \label{section 7.2}
  
As in the proof of Lemma \ref{lemma-6.10}, one obtains an essentially unique surjective homomorphism 
$\varphi:\Gamma \rightarrow {\rm PGL}_2(\mathbb{F}_q)$ with $\Delta:=\ker
(\varphi )$ a free group of rank $\frac{q(q-1)}{2}$. Further the embeddings
$\Gamma\rightarrow {\rm PGL}_2(K)$ as discontinuous group are parametrised by a punctured disk.

 We will now produce equations for these families of Mumford
 curves. As in \S \ref{section 7.1.1.0}, one can compute the stable reduction $R$ of
 $X_\lambda$. It consist of $q+1$ projective lines such that each line
 intersects each other line in one point. The stabiliser of each line
 is a group $\cong B(n,q-1)$ and the stabiliser of an intersection of two
 lines is a group $\cong D_{q-1}$.  

 Again, one makes the guess that the curve can be embedded (as smooth
 curve) in $\mathbb{P}^2_K$ defined by a homogeneous polynomial  $F$
 of degree $q+1$. Using the notation of \S \ref{section 7.1.1.0}  we define 
$F_0:= z_0\cdot \prod _{a\in \mathbb{F}_q} (a^2z_0-az_1+z_2)$.

 One easily verifies that $F_0=0$ is the union of $q+1$ lines having only simple
intersections. Further the stabilizer of $z_0=0$ is $B(n,q-1)$ and  
by definition $F_0$ is the ${\rm PGL}_2(\mathbb{F}_q)$-orbit of the
line $z_0$ and therefore ${\rm PGL}_2(\mathbb{F}_q)$-invariant.

 Assume now that $p\neq 2$. In analogy with \ref{7.6} one considers
 the following. The expression $z_1^2-z_0z_2$ is invariant under ${\rm
  PGL}_2(\mathbb{F}_q)$. Then the family of curves $Y_\lambda$ defined by
$F=F_0+\lambda (z_1^2-z_0z_2)^{\frac{q+1}{2}}$ with $\lambda\in K,\ 0<|\lambda|<1$
has the required properties.

 Indeed, it can be shown that $Y_\lambda$ is smooth. Its genus is $\frac{q(q-1)}{2}$
and the automorphism group is ${\rm PGL}_2(\mathbb{F}_q)$. The obvious
reduction of $F$ is equal to the reduction of $F_0$ and defines a stable
curve such that all irreducible components are projective lines. Hence
$Y_\lambda$ is a Mumford curve and corresponds to one of the two amalgams of Theorem \ref{6.8}.
It must be  $\Gamma:=D_{q-1}\ast_{C_{q-1}}B(n,q-1)$ because of the
structure of its reduction.

\subsubsection{\rm The family of Mumford curves for
  $D_{q-1}\ast_{C_{q-1}}B(n,q-1)$ and  $p=2$.}\label{section 7.2.1}

Let the group  ${\rm PGL}_2(\mathbb{F}_{q})$, $q=2^n$ act on the projective plane
$\mathbb{P}^2_{\mathbb{F}_q}$ preserving the conic $x_0x_2-x_1^2=0$.
This conic has $q+1$ points that are defined over $\mathbb{F}_q$.
They are the points $(1,t,t^2)$, $t\in \mathbb{F}_q$ and the point $(0,0,1)$.
Each of these points is stabilized by a group $\cong B(n,q-1)\subset {\rm PGL}_2(\mathbb{F}_{q})$.
The same holds true for the tangents of the conic at these points.
Since $p=2$ all the tangents intersect in the single point $(0,1,0)$
and cannot be used to describe the reduction of the Mumford curve!

Let $\mathbb{P}^\vee$ be the dual projective plane, where
the duality is given by the equation $x_0y_0+x_1y_1+x_2y_2=0$. 
The $\mathbb{F}_q$-rational points of the conic define a set
$\mathcal{L}$ of projective lines in the dual plane $\mathbb{P}^\vee$.
The set $\mathcal{L}$ consists of the lines $y_0+ty_1+t^2y_2=0$, $t\in\mathbb{F}_q$ combined with
the line $y_2=0$.
Each of these lines are again stabilized by a subgroup $\cong B(n,q-1)\subset {\rm PGL}_2(\mathbb{F}_{q})$.
Their intersection points now are the duals of the lines in $\mathbb{P}^2_q$ that intersect
the conic in two $\mathbb{F}_q$-rational points. 
There are $\frac{q(q+1)}{2}$ such lines and each such line is stabilized by a dihedral subgroup
$\cong D_{q-1}\subset {\rm PGL}_2(\mathbb{F}_{q})$.
Hence the set $\mathcal{L}$ of lines in the dual plane $\mathbb{P}^\vee$ consists of $q+1$ lines,
each stabilized by a subgroup $B(n,q-1)\subset {\rm PGL}_2(\mathbb{F}_{q})$.
These lines intersect in $\frac{q(q+1)}{2}$ points that are each stabilized by a dihedral group 
$D_{q-1}\subset {\rm PGL}_2(\mathbb{F}_{q})$.

In particular, the set $\mathcal{L}$ in $\mathbb{P}^\vee$ describes the reduction of our Mumford curve. 
Let $F_0$ be the equation of degree $q+1$ whose zeroes are the set of lines $\mathcal{L}$.
The group ${\rm PGL}_2(\mathbb{F}_{q})$ also preserves 
a hermitian form on the dual projective plane $\mathbb{P}^\vee$.
Therefore  ${\rm PGL}_2(\mathbb{F}_{q})$ also preserves the corresponding hermitian curve $h(y)=0$
that has again degree $q+1$.
The family of curves $Y_\lambda=F_0+\lambda\cdot h(y)$ with $\lambda\in K$, $0<|\lambda|<1$,
consists of the Mumford curves that have the curve $F_0$ as their reduction. 

\begin{remark}\label{7.7} {\rm
For $p=2$ a family of smooth plane curves of genus $\frac{q^2-q}{2}$, $q=2^n\geq 4$
with automorphism group ${\rm PGL}_2(\mathbb{F}_q)$ is known
(\cite{Fu} Theorem 1).
These curves are ordinary (see \cite{Fu} remark 1).
The curves are defined by the equation $z\prod_{t\in\mathbb{F}_q} (x+ty+t^2z)+\lambda\cdot y^{q+1}=0$
with $\lambda\in K,\;\lambda\neq 0,1$ 
in the projective plane $\mathbb{P}_{\mathbb{F}_q}^2$ with coordinates $x,y,z$. 

The subset of the curves with $0<|\lambda|<1$ consists of Mumford curves.
The reduction of these Mumford curves consists of $q+1$ projective lines  $\mathbb{P}_{\mathbb{F}_q}^1$.
Therefore this subset consists of the Mumford curves belonging to the amalgam 
$D_{q-1}\ast_{C_{q-1}}B(n,q-1)$ with $p=2$. 

}
\end{remark}

\section{\rm  An upper bound for the automorphism group} \label{section 8}

\subsection{\rm Establishing the upper bound}

The problem is to determine for any genus $g\geq 2$, the maximum,  call it
${\rm Max}(g)$,  of the order of ${\rm Aut}(X)$, where $X$ is a Mumford curve of genus $g$. For
small $g$, one can deduce from our paper that \\

 {\it ${\rm Max}(g)=12(g-1)$ for $g=2,3,4,5$ and all $p$. Further\\
 \indent  ${\rm Max}(6)=60$ for $p\neq 3$ and ${\rm Max}(6)=72$ for $p=3$}. \\

For genus $g>6$ it seems hardly possible to compute ${\rm Max}(g)$, except
for some special values of $g$. We will show that the two families  of Theorem \ref{6.8} with
automorphism group ${\rm PGL}_2(\mathbb{F}_q)$ and genus
$g=\frac{q(q-1)}{2}$ attain the value ${\rm Max}(g)$.\\

For the convenience of the reader we enumerate in the proposition below
the amalgams $\Gamma$ with $\mu(\Gamma)=\frac{1}{12}$.
This corrects some minor errors in Prop. 1.2 and the Theorem of \cite{C-K 2}.

\begin{proposition}\label{6.14} \label{8.1}
There are three amalgams $\Gamma$ with $\mu(\Gamma)=\frac{1}{12}$ for $p=2,3,5$
and four such amalgams for $p>5$, namely
\begin{footnotesize}
\begin{enumerate} 
\item[]$D_3\ast_{C_2}D_2\cong {\rm PGL}_2(\mathbb{F}_{2})\ast_{C_2}B(2,1)\cong B(1,2)\ast_{C_2}D_2$
{\rm (}for $p\geq 5,\ p=2,\ p=3${\rm )},  
\item[]$S_4\ast_{C_4}D_4\cong {\rm
    PGL}_2(\mathbb{F}_{3})\ast_{C_4}D_4$ {\rm (for }$p\geq 5, p=3${\rm )}, 
\item[]$A_4\ast_{C_3}D_3\cong B(2,3)\ast_{C_3}{\rm
    PGL}_2(\mathbb{F}_{2})$ {\rm (for }$p\geq 5, p=2${\rm )},
\item[]$A_5\ast_{C_5}D_5\cong {\rm
    PGL}_2(\mathbb{F}_{4})\ast_{C_5}D_5$ {\rm (for }$p>5, p=3,\ p=2${\rm )}.
\end{enumerate}
\end{footnotesize}
The branch groups are $C_2$, $C_2$, $C_2$ and $C_3$ if $p>3$ and $C_2$, $C_2$ and $B(1,2)$ if $p=3$.
If $p=2$, then the branch groups are $C_3$ and $B(2,1)$ for $\Gamma=D_3\ast_{C_2}D_2$
and the branch groups are $C_2$ and $B(2,3)$ for the two remaining groups.
\end{proposition}


\begin{corollary} \label{8.2bis}There are infinitely many integers $g\geq 2$ for which there is no Mumford curve
$X$ with  genus $g$ and $|{\rm Aut}(X)|=12(g-1)$.
\end{corollary}
\begin{proof}
If the Mumford curve $X$ satisfies $|{\rm Aut}(X)|=12(g-1)$, then $X$ is uniformized by a normal Schottky subgroup
$\Delta$ of finite index in an amalgam $\Gamma$ with $\mu(\Gamma)=\frac{1}{12}$.
Consider $\Gamma \neq D_3\ast _{C_2}D_2$ with $\mu(\Gamma)=\frac{1}{12}$. Then $\Gamma$ does not have
a normal Schottky subgroup of index 12.  This follows directly from the fact that either the order of one of the finite groups involved in the amalgam $\Gamma$ has order $>12$ or in the case $\Gamma\cong A_4\ast_{C_3}D_3$ from the fact that $D_3\not\subset A_4$.

We continue with  $\Gamma\neq  D_3\ast_{C_2} D_2$ and a normal Schottky subgroup $\Delta \subset \Gamma$. 
 Assume that  $|H|=|\Gamma/\Delta|=12 p^\prime$ for some prime $p^\prime> 5$.
Then $H$ contains a $p^\prime$-Sylow subgroup $C_{p^\prime}$. Let $m_{p^\prime}$ the number of $p^\prime$-Sylow subgroups of $H$.
By Sylows theorems $m_{p^\prime}\equiv 1\mod p^\prime$ and  $m_{p^\prime}|\; |H|$. 
Since $p^\prime> 5$,  $m_{p^\prime}=1$ and the subgroup is normal in $H$.
The amalgam $\Gamma$ contains no elements of order $p^\prime$.
In particular, the preimage $\Delta_{p^\prime}\subset\Gamma$ of the group $C_{p^\prime}\subset H$ contains no elements of finite order.
Therefore $\Delta_{p^\prime}\subset\Gamma$ is a normal Schottky subgroup of index $12$. This cannot be.
In particular, the three amalgams $\Gamma\not= D_3\ast_{C_2}D_2$ do not give rise to Mumford curves of genus $g=p^\prime +1$ with an automorphism group of order $12(g-1)$ for any prime $p^\prime> 5$.\\

Consider $\Gamma = D_3\ast_{C_2}D_2$. This group has two normal Schottky subgroups $\Delta$ with 
index 12. 
The following claim will end the proof of the corollary:\\

\noindent  {\it  $\Gamma$ has no normal Schottky groups $\Delta$ of index $12p^\prime$
if  $p^\prime\equiv 11\mod 12$ is prime}.\\

$\Gamma=D_3\ast_{C_2}D_2$ is isomorphic to the extended modular group ${\rm PGL}_2(\mathbb{Z})$ and 
plays an important role in the uniformization of Klein surfaces (i.e., algebraic curves defined over the field $\mathbb{R}$).
The proof of the claim follows from this observation.

The maximal order of the automorphism group of a compact Klein surface with boundary is again $12(g-1)$.
In \cite{M-2} theorem 1 and \S4 it is shown that the automorphism group of order $12(g-1)$
is a quotient of the extended modular group by a normal Schottky subgroup.
Moreover, it is proved that any quotient of the extended modular group by a normal Schottky subgroup occurs as the automorphism group of a compact Klein surface with boundary. 
In \cite{M-1} theorem 2, it is shown that there do not exist compact Klein surfaces with boundary of genus $g$ with $g-1$ a prime such that $g-1\equiv 11\mod 12$ that have an automorphism group of order $12(g-1)$. This is precisely the claim for the
$\Gamma=D_3\ast_{C_2}D_2$. \end{proof}

Mumford curves are ordinary, i.e., the p-rank of their Jacobian equals the genus $g$.
By \cite{N} \S 3 Corollary (ii), the quotient map $X\longrightarrow X/{\rm Aut}(X)$ is tamely ramified
for ordinary curves $X$, if $2\leq g\leq p-2$. 

In particular, for a Mumford curve $X=\Omega/\Delta$ with $\Delta\subset\Gamma$ a normal Schottky group,
this implies that $\mu(\Gamma)\geq\frac{1}{12}$.
It follows that for any genus $g\geq 2$, one has ${\rm Max}(g)\leq 12(g-1)$ provided $p\geq g+2$.
By Corollary \ref{8.2bis} there exist in fact arbitrarily large $g>6$ such that ${\rm Max}(g)<12(g-1)$ provided $p\geq g+2$.
Of course, for a fixed genus $g>2$ only for  $p<g+2$ one can have ${\rm Max}(g)>12(g-1)$. \\

The `extreme  cases' of Theorem \ref{6.8} leads to the {\bf claim} Theorem \ref{6.18}: \\

{\it ${\rm Max}(g)\leq \max \{ 12 (g-1),F(g)\}$  with
  $F(g):=2g(\sqrt{(2g+\frac{1}{4})}+\frac{3}{2})
         =g(\sqrt{(8g+1)}+3)$ holds for all  $p$ and $g$ with the exception of  $p=3,\ g=6$}. \\ 

The function $F$  is the rational function in $q$  which has the
property that for $g=\frac{q(q-1)}{2}$ the value of $F$ is $|{\rm
  PGL}_2(\mathbb{F}_q)|=q^3-q$.

 We note that $F(g)<12(g-1)$ only holds for $g=4,5$.
We will show that for $g=4,5$ one has  ${\rm Max}(g)=12(g-1)$.\\

{\it The strategy for the proof of the claim ${\rm Max}(g)\leq \max (12
(g-1),F(g))$}, which is adapted from \cite{C-K-K} lemma 6.2,  is as follows. 
Since the bound is not linear we have for
each realizable amalgam $\Gamma$ with $\mu(\Gamma )<\frac{1}{12}$ to
compute normal Schottky subgroups $\Delta$ with minimal index $|\Gamma
/\Delta|$.  In general, this is too difficult.  Instead, one tries to find a 
positive integer $N_0(\Gamma)$ such that $|\Gamma /\Delta |\geq
N_0(\Gamma)$ for all normal Schottky subgroups $\Delta$ of finite
index. One defines $g_0:=1+\mu(\Gamma) N_0(\Gamma)$.

\begin{lemma}\label{6.15} \label{8.2} Suppose that $N_0(\Gamma)$ and $g_0$ satisfy
  $g_0\geq 5$ and $N_0(\Gamma)\leq F(g_0)$. Then for any Mumford curve
  $X$ of genus $g$ corresponding to a normal Schottky subgroup
  $\Delta\subset \Gamma$ of finite
  index one has $|{\rm Aut}(X)|\leq F(g)$.  
\end{lemma}

\begin{proof}
$g-1=\mu(\Gamma)\cdot |{\rm Aut}(X)|\geq \mu(\Gamma) N_0(\Gamma)=g_0-1$ and so
$g\geq g_0$.
Since the function $\frac{F(g)}{g-1}$ is increasing for $g\geq 5$ one has
 $\mu(\Gamma)^{-1}=\frac{N_0(\Gamma )}{g_0-1}\leq \frac{F(g_0)}{g_0-1}\leq \frac{F(g)}{g-1}$ and
thus  $|{\rm Aut}(X)|=\mu(\Gamma)^{-1}\cdot (g-1)\leq F(g)$.
\end{proof}
An equivalent  formulation of the Lemma \ref{6.15} is the following.\\
{\it If} (i) $g_0\geq 5$,  $\mu(\Gamma)^{-1}\leq \max (12,
\frac{F(g_0)}{g_0-1})$ and (ii) $\Delta \subset \Gamma$ is a normal
Schottky subgroup of finite index, free on $g\geq g_0$ generators,  
{\it then} $|\Gamma /\Delta |\leq F(g)$.\\

\noindent 
We will call $N_0(\Gamma)$ {\it suitable} if $|\Gamma /\Delta|\geq
N_0(\Gamma)$ for all normal Schottky groups $\Delta \subset \Gamma$ of
finite index and $N_0(\Gamma)\leq F(g_0)$ with
$g_0:=1+\mu(\Gamma)N_0(\Gamma)\geq 5$. \\

Lemma \ref{6.16} reduces the proof of the claim to finding a suitable
$N_0(\Gamma)$ for every $\Gamma$ appearing as a sub-amalgam in Lists
\ref{lists}. {\it We call an amalgam of the form $G_1\ast _{G_3}G_2$  simple. }

\begin{lemma}\label{6.16} \label{8.3} The claim holds if for every  simple sub-amalgam
  $\Gamma$, different from ${\rm
    PGL}_2(\mathbb{F}_3)\ast_{B(1,2)}B(2,2)$, in the Lists {\rm \ref{lists}} a suitable
  $N_0(\Gamma)$ is produced.  \end{lemma}   
\begin{proof} Any amalgam $\Gamma$ with $\mu(\Gamma)\leq\frac{1}{12}$
that is not simple contains at least
one simple sub-amalgam $\Gamma^\prime$, different from  ${\rm PGL}_2(\mathbb{F}_3)\ast_{B(1,2)}B(2,2)$.
We claim that $N_0(\Gamma):=N_0(\Gamma ^\prime)$ is {\it suitable}.

Indeed, $\mu(\Gamma)\geq \mu(\Gamma ^\prime)$. Write
$g_0^\prime=1+N_0(\Gamma ^\prime)\mu(\Gamma ^\prime)$. Let
$\Delta\subset \Gamma$ be a normal Schottky subgroup of finite
index. Then $|\Gamma/\Delta |\geq |\Gamma ^\prime/(\Delta \cap \Gamma
^\prime)|\geq N_0(\Gamma ^\prime)$. Therefore $g:=1+|\Gamma /\Delta|\mu(\Gamma)\geq g_0^\prime\geq 5$.
Further $|\Gamma /\Delta|=\mu(\Gamma )^{-1}(g-1)\leq \mu(\Gamma
^\prime)^{-1}(g-1)\leq F(g)$ since $g\geq g_0^\prime$.  \end{proof}

Tables \ref{table-1} and \ref{table-2} provide {\it suitable} $N_0(\Gamma)$ for all simple
sub-amalgams that occur in Lists \ref{lists}, except for ${\rm
    PGL}_2(\mathbb{F}_3)\ast_{B(1,2)}B(2,2)$. 

Let $\varphi :\Gamma =G_1\ast_{G_3}G_2
  \rightarrow H$ be a surjective homomorphism to a finite group $H$
  such that its kernel is a Schottky group. Then $\varphi (G_i)\cong G_i,
  i=1,2$ are subgroups of $H$ and it follows that $|H|$ is a multiple
  of ${\rm l.c.m.} (|G_1|,|G_2|)$. Thus  we may suppose that $N_0(\Gamma)=n
  \cdot {\rm  l.c.m. }(|G_1|,|G_2|)$ for some integer $n$.

\begin{Table}\label{table-1} \label{8.4} {\rm 
In the table below we give the $\Gamma =G_1\underset{G_3}{\ast} G_2$ with
$\mu(\Gamma)<\frac{1}{12}$ such
that $N_0(\Gamma):={\rm l.c.m.} (|G_1|,|G_2|)$ is suitable.
     \begin{center}
  {\fontsize{9pt}{0.5pt}    
 \begin{tabular}{| l| r|r |r |}
  \hline
  $\Gamma$ & $N_0(\Gamma)$ & $\mu(\Gamma)$ & $g_0$ \\ \hline
  ${\rm PGL}_2(\mathbb{F}_{q})\underset{C_{q+1}}{\ast}D_{q+1}$ & $q^3-q$ & $\frac{q-2}{2q(q-1)}$
  & $\frac{q^2-q}{2}$ \\
  ${\rm PGL}_2(\mathbb{F}_{q})\underset{C_{q+1}}{\ast}{\rm PGL}_2(\mathbb{F}_{q})$ & $q^3-q$ 
  & $\frac{q^2-q-2}{q^3-q}$ & $q^2-q-1$ \\
  ${\rm PGL}_2(\mathbb{F}_{q})\underset{B(n,q-1)}{\ast} B(n\cdot m,q-1)$, $m>3$ & $q^m(q^2-1)$ 
  & $\frac{q^m-q-1}{q^m(q^2-1)}$ & $q^m-q$ \\
  ${\rm PGL}_2(\mathbb{F}_{3})\underset{B(1,2)}{\ast}B(3,2)$ 
  & $216$   & $\frac{23}{216}$ & $24$ \\
  ${\rm PGL}_2(\mathbb{F}_{q})\underset{C_{q+1}}{\ast} B(2n\cdot m,q+1)$, $m\geq 1$ 
  & $q^{2m}(q^2-1)$
  & $A$  
  & $1+q^{2m}(q^2-1)A$ \\
  ${\rm PSL}_2(\mathbb{F}_{q})\underset{C_{\frac{q+1}{2}}}{\ast} D_{\frac{q+1}{2}}$ & $\frac{q^3-q}{2}$ 
  & $\frac{q-2}{q(q-1)}$
  & $\frac{q^2-q}{2}$ \\
  ${\rm PSL}_2(\mathbb{F}_{q})\underset{C_{\frac{q+1}{2}}}{\ast}{\rm PSL}_2(\mathbb{F}_{q})$ & $\frac{q^3-q}{2}$ 
  & $\frac{2q^2-2q-4}{q^3-q}$ & $q^2-q-1$ \\
  ${\rm PSL}_2(\mathbb{F}_{q})\underset{B(n,\frac{q-1}{2})}{\ast}B(n\cdot m,\frac{q-1}{2})$, $m>3$ 
  & $\frac{q^m(q^2-1)}{2}$ 
  & $\frac{2q^m-2q-2}{q^m(q^2-1)}$ & $q^m-q$ \\
${\rm PSL}_2(\mathbb{F}_{3})\underset{B(1,1)}{\ast}B(3,1)$ 
  & $108$   & $\frac{23}{108}$ & $24$ \\
 ${\rm PSL}_2(\mathbb{F}_{5})\underset{B(1,2)}{\ast}B(2,2)$  & $300$ 
  & $\frac{19}{300}$ & $20$ \\
${\rm PSL}_2(\mathbb{F}_{5})\underset{B(1,2)}{\ast}B(3,2)$ 
  & $1500$ 
  & $\frac{119}{1500}$ & $120$ \\
${\rm PSL}_2(\mathbb{F}_{q})\underset{C_{\frac{q+1}{2}}}{\ast}B(2n\cdot m,\frac{q+1}{2})$, $m\geq 1$ 
  & $\frac{q^{2m}(q^2-1)}{2}$
  & $2A$
  & $1+q^{2m}(q^2-1)A$  \\
  \hline
 \end{tabular}
 }
  \end{center}
}
\end{Table}

where $A=\frac{q^{2m+1}-q^{2m}-q^{2m-1}-q+1}{q^{2m}(q+1)(q-1)}$.

\begin{Table}\label{table-2} \label{8.5} {\rm
In the table below $\mu(\Gamma)<\frac{1}{12}$ and $N_0(\Gamma)$ is strictly larger than ${\rm l.c.m.}(|G_1|, |G_2|)$.
{\it These values of $N_0(\Gamma)$ with $N_0(\Gamma)\leq F(g_0)$
are established in \S \ref{section 8.2} and \S \ref{section 8.3}.    } 
 
\begin{center}
 {\fontsize{6pt}{0.5pt} 
 \begin{tabular}{| l| r|r |r |}
  \hline
  $\Gamma$ & $N_0(\Gamma)$ & $\mu(\Gamma)$ & $g_0$ \\ \hline
  ${\rm PGL}_2(\mathbb{F}_{q})\ast_{B(n,q-1)}B(3\cdot n,q-1)$, $q>3$ & $q^6$ 
  & $\frac{q^3-q-1}{q^3(q^2-1)}$ & $\frac{q^6-q^4-q^3+q^2-1}{q^2-1}$ \\
  ${\rm PSL}_2(\mathbb{F}_{q})\ast_{B(n,\frac{q-1}{2})}B(3\cdot n,\frac{q-1}{2})$, $p>2$, $q>5$ 
  & $q^6$
  & $\frac{2q^3-2q-2}{q^3(q^2-1)}$ & $\frac{2q^6-2q^4-2q^3+q^2-1}{q^2-1}$ \\
${\rm PGL}_2(\mathbb{F}_{q})\ast_{B(n,q-1)}B(2\cdot n,q-1)$, $p>2$, $q>3$ & $\frac{(q^3-q)^2}{2}$ 
  & $\frac{q^2-q-1}{q^2(q^2-1)}$ & $\frac{(q^2-1)(q^2-q-1)+2}{2}$ \\
${\rm PGL}_2(\mathbb{F}_{q})\ast_{B(n,q-1)}B(2\cdot n,q-1)$, $p=2$, $q\geq 4$ & $(q^3-q)^2$ 
  & $\frac{q^2-q-1}{q^2(q^2-1)}$ & $(q^2-1)(q^2-q-1)+1$ \\
${\rm PSL}_2(\mathbb{F}_{q})\ast_{B(n,\frac{q-1}{2})}B(2\cdot n,\frac{q-1}{2})$, $p>2$, $q>5$ 
  & $\frac{(q^3-q)^2}{4}$
  & $\frac{2q^2-2q-2}{q^2(q^2-1)}$ & $\frac{(q^2-1)(q^2-q-1)+2}{2}$ \\
$B(n_1,\ell)\ast_{C_\ell}B(n_2,\ell)$, $n_1\geq n_2$ & $p^{n_1+n_2}\ell$ 
  & $\frac{p^{n_1}-p^{n_1-n_2}-1}{p^{n_1}\ell}$
  & $p^{n_1+n_2}-p^{n_1}-p^{n_2}+1$ \\
$B(n,\ell)\ast_{C_\ell}D_\ell$, $\ell|(q-1)$ & $(q^2+q)\ell$
  & $\frac{q-2}{2q\ell}$ & $\frac{q^2-q}{2}$ \\
  
  \hline
 \end{tabular}
 }
\end{center}
}
\end{Table}

\begin{theorem}\label{6.18} \label{8.6}
Let $X=\Omega/\Delta$ be a Mumford curve with $g>1$ and
$\Gamma$ the normalizer of $\Delta$.
Then $|{\rm Aut}(X)|\leq\max\{12(g-1),\: F(g)\}$, except for the case
$p=3$, $\Gamma= {\rm  PGL}_2(\mathbb{F}_3)\ast_{B(1,2)}B(2,2)$, $g=6$
and $|{\rm Aut}(X)|=72$.\\
\end{theorem}

\begin{proof}
One first easily verifies that for amalgams $\Gamma$ such that there exist $\Delta\subset\Gamma$
with $X=\Omega/\Delta$ of genus $g<5$ the value of $\mu(\Gamma)$ is
$\mu(\Gamma)\geq \frac{1}{12}$.
Hence Mumford curves corresponding to such amalgams (see \ref{6.19}) satisfy the theorem.

Now we assume that $g\geq g_0\geq 5$ and exclude $\Gamma={\rm
  PGL}_2(\mathbb{F}_3)\ast_{B(1,2)}B(2,2)$. This amalgam will be studied in Proposition \ref{6.22}. 
 Lemmata  \ref{6.15}, \ref{6.16} and the two tables \ref{table-1}, \ref{table-2} complete the
 proof. \end{proof}
Table \ref{table-1} is trivial to verify. The rest of \S 8 provides
examples treating special cases and finally the far from evident,
delicate  verification of Table \ref{table-2}.

\begin{lemma}\label{6.19} \label{8.7} The amalgam $\Gamma:=D_3\ast_{C_2}D_2$ is realizable for
 every $p$ and $\mu(\Gamma)=\frac{1}{12}$. For $g=2,\dots ,6$ there
 is a normal Schottky subgroup $\Delta _g$ which is free on $g$
 generators. Hence ${\rm Max}(g)\geq 12(g-1)$ for
$g=2,\dots ,6$. 
By the theorem above equality holds for $g=2,\ldots ,5$ and if $p\not=3$ also for $g=6$.
\end{lemma}

\begin{proof}
The first statement follows from the observations:
for $p=2$ one has $\Gamma\cong {\rm PGL}_2(\mathbb{F}_{2})\ast_{B(1,1)}B(2,1)$
and for $p=3$ one has $\Gamma\cong B(1,2)\ast_{C_2}D_2$.\\

Now  $\Delta_g=\ker \varphi_g$ where $\varphi_g$ is an explicit
surjective homomorphism of $\Gamma$ to a group of $12(g-1)$ elements such that
$\varphi_g$ is injective on $D_3$ and on $D_2$.\\
Write $\Gamma =<a,b,c\ |\  a^3=b^2=c^2=1,\ bab=a^2,\ bc=cb>$ with
$D_3=<a,b\ |\  a^3=b^2=1, \ bab=a^2>$ and $D_2=<b,c\ |\  b^2=c^2=1,\ bc=cb>$.\\

\noindent $\varphi_2:\;\Gamma\longrightarrow D_3\times C_2$ is defined by
$a\mapsto (a,1),\ b\mapsto (b,1),\ c\mapsto (1,-1)$,\\  where we have
written $C_2=\{\pm 1\}$.\\
\noindent
$\varphi_3:\;\Gamma\longrightarrow S_4\cong {\rm
  PGL}_2(\mathbb{F}_{3})$ is defined by $a\mapsto {1\ 1\choose 0\ 1 },\ b\mapsto
{-1\ 0\choose 0\ 1 },\ c\mapsto {0\ 1\choose 1\ 0 }$.\\
\noindent $\varphi_4:\Gamma\longrightarrow D_3\times D_3$, is defined by 
$a\mapsto (a,a),\ b\mapsto (b,b),\ c\mapsto (b,1)$. \\
\noindent  
$\varphi_5:\Gamma \longrightarrow  S_4\times C_2$ is defined by $\varphi_5(g)=(\varphi_3(g),\psi
(g))$, where $\psi (g)=1$ for $g=a,b$ and $\psi (c)=-1$ and
$C_2=\{\pm 1\}$.\\
\noindent $\varphi_6:\Gamma \longrightarrow A_5\cong{\rm PGL}_2(\mathbb{F}_4)$ is
defined by identifying  $D_3$ with the subgroup ${\rm PGL}_2(\mathbb{F}_2)$
of ${\rm PGL}_2(\mathbb{F}_4)$ and $D_2$ with the subgroup
${1\ \mathbb{F}_4\choose 0\ 1 }$ of   ${\rm PGL}_2(\mathbb{F}_4)$.
\end{proof}

\begin{remarks} \label{8.8} {\rm
$\Delta_2$ is the free group with generators $\delta
_1=aca^2c,\ \delta _2=a^2cac$. The group $\Gamma$ acts by conjugation
on $\Delta_2$ and on $\Delta_{2,ab}\cong \mathbb{Z}^2$. The matrices
of the conjugations by $a,b,c$ are ${-1\ -1\choose 1\ 0 },\ {0\
  1\choose 1\ 0 }, \ {-1\ 0\choose 0\ -1 }$. This implies that
$\Delta_{2,ab}/3\Delta_{2,ab}$ has a 1-dimensional subspace invariant
under conjugation. Hence there exists $\Delta \subset \Delta_2$ of
index 3 which is normal in $\Gamma$. One can take $\Delta$
as the required $\Delta_4$. Since there is a surjective homomorphism
$D_3\times D_3\longrightarrow D_3\times C_2$ one concludes that
$\Delta \subset \Delta_2$ is the kernel of $\varphi_4$.\\

  We note that for $p\not=2$ there exist other Mumford curves of genus $g=3,5$
with automorphism group $S_4$ and $S_4\times C_2$, respectively.
These curves corresponds to the amalgam $S_4\ast_{C_4}D_4\cong {\rm PGL}_2(\mathbb{F}_{3})\ast_{C_4}D_4$.
Moreover, the amalgams $A_5\ast_{C_5}D_5$ ($p\not=5$) and $A_4\ast_{C_3}D_3$ ($p\not=3$)
give rise to Mumford curves of genus $g=6$ with automorphism group
$A_5\cong {\rm PGL}_2(\mathbb{F}_4)$.     } 
\end{remarks}

\subsection{\rm The exceptional curves for $p=3$ with  $g=6$} \label{section 8.1} 

\begin{examples}\label{6.21} \label{8.9}  $\Gamma=\Gamma_{v_1}\ast_{\Gamma_e}\Gamma_{v_2}
={\rm PGL}_2(\mathbb{F}_{3})\ast_{B(1,2)}B(m,2)$ with $m\geq 2$. 
{\rm Write $B(1,2)={\pm 1 \ \mathbb{F}_3\choose 0\ \ \ 1}$, $B(m,2)={ \pm 1\
  W\choose 0\ \ \ \ 1 }$ where $\{\pm 1\}=\mathbb{F}_3^*$ and $W$ is  a vector space over $\mathbb{F}_3$ of dimension $m$.
The given homomorphism $\Gamma_e\rightarrow \Gamma_{v_2}$ provides $W$ with a 1-dimensional subspace over $\mathbb{F}_3$,
namely the image of $\{{1\ \mathbb{F}_3\choose 0\ \ 1 }\}$ into $\{{1\ W \choose 0\ 1 }\}\subset \Gamma_{v_2}$. This subspace is
denoted by $\mathbb{F}_3\subset W$. 

{\it Now we choose a $\mathbb{F}_3$-vector space $V\subset W$ with $W=\mathbb{F}_3\oplus V$}. The number of possibilities for $V$ is $\frac{3^m-3^{m-1}}{2}$.

Put $H_m:= \{(g_1,g_2)\in {\rm PGL}_2(\mathbb{F}_{3})\times {\pm 1\
  V\choose 0\ 1 }\mid\:  \det(g_1)\cdot \det(g_2)=1\}$. This finite group does not depend on the choice of $V$.\\ 

 Define a homomorphism   
$\varphi_m:\Gamma \longrightarrow H_m$, which does depend on the choice of $V$,  by:
 if $g\in {\rm PGL}_2(\mathbb{F}_{3})$, then $\varphi_m(g)=(g, {a\ 0\choose
   0\ 1 })$ with $a=\det g$; \\
\indent  if ${a\ b\choose 0\ 1 }{1\ v\choose 0\ 1 }\in B(m,2)$ with $a\in
 \{\pm 1\}, b\in \mathbb{F}_3, v\in V$, then\\ 
$\varphi_m({a\ b\choose 0\ 1 }{1\ v\choose 0\ 1 })=({a\ b \choose 0\ 1 },
{a\ 0 \choose 0\ 1 }{1\ v\choose 0\ 1 })$.
It is easily seen that  $\varphi_m$ is surjective and clearly 
$\ker \varphi_m$ is a Schottky group, depending on the choice of $V$.\\ 

The formula $(g-1)=\mu(\Gamma)\cdot |H_m|$ implies that $g=3^m-3$  and
that $\Delta_m:=\ker \varphi _m$ is a free group on $3^m-3$
generators. This can be made explicit as follows.

 One considers $D_2$ as subgroup of ${\rm PSL}_2(\mathbb{F}_3)
 \subset {\rm PGL}_2(\mathbb{F}_3)$ and we write ${1\ V\choose 0\ 1 }=C_3^{m-1}$.  
The group $H_m$ contains a normal subgroup $D_2\times C_3^{m-1}$.
Then $\Gamma$ contains the normal subgroup  $\varphi_m^{-1}(D_2\times C_3^{m-1})=D_2\ast C_3^{m-1}$.
The commutator subgroup $[D_2,\; C_3^{m-1}]\subset D_2\ast C_3^{m-1}$
is seen to be a free group on $(4-1)(3^{m-1}-1)=3^m-3$ generators,
namely the elements $aba^{-1}b^{-1}$ with $a\in D_2,\ a\neq 1$ and
$b\in C_3^{m-1},\ b\neq 1$. This group is contained in the kernel $\Delta_m:=ker(\varphi_m)$.
Moreover, the quotient $D_2\ast C_3^{m-1}/[D_2,\; C_3^{m-1}]$ equals
the group $D_2\times C_3^{m-1}$. Therefore the kernel $\Delta_m=ker(\varphi_m)$ is the group
$[D_2,\; C_3^{m-1}]$. }   \end{examples}

\begin{proposition}\label{6.22} \label{8.10}
Consider  $p=3$ and a realization $\Gamma$ of the amalgam ${\rm
  PGL}_2(\mathbb{F}_{3})\ast_{B(1,2)}B(2,2)$. Let $\Delta\subset \Gamma$ be a normal
Schottky group with finite index defining a Mumford curve $X=\Omega
/\Delta$.  

Then the three normal Schottky groups $\Delta=\Delta_2$ are the only cases with $|{\rm Aut}(X)|>\max
(12(g-1),F(g))$. Moreover, for $\Delta=\Delta_2$ one has $g=6$ and $|{\rm Aut}(X)|=72$. 
\end{proposition}

\begin{proof} One has $\mu(\Gamma)=\frac{5}{72}$ and therefore
$|{\rm Aut}(X)|$ is divisible by $72$ and $5|g-1$. Further $F(g)>\frac{72}{5}\cdot(g-1)$ holds for $g\geq 16$,
$F(6)<72$ and $F(11)< 144$. Therefore the genus of $X$ must be $g=6$ or $g=11$. \\

\noindent 
{\it Consider the case $g=6$ and thus $|{\rm Aut}(X)|=72$}. There exist
$50$ distinct groups of order $72$ (see e.g. \cite{P-W}). Only four of these groups contain a subgroup $A_4$.
These are the groups $S_4\times C_3$,  $A_4\times S_3$, $A_4\times C_6$
and $H_2=C_2\ltimes (A_4\times C_3)$. The group $H_2$ is the only group of order $72$ that
contains both a subgroup $S_4\cong {\rm PGL}_2(\mathbb{F}_{3})$ and a subgroup $B(2,2)$.
Therefore only the group $H_2$ can occur as the automorphism group of a Mumford curve
$X=\Omega/\Delta$ of genus $g=6$ such that $\Gamma$ is the normalizer of $\Delta$ in 
${\rm PGL}_2(K)$.\\

 The group homomorphism $\varphi_2:\;\Gamma \longrightarrow H_2$
 defined in \ref{6.21} has kernel $\Delta_2$ with the required
 properties. The map $\varphi_2$ is uniquely determined by the choice of $V\subset W$ with $\mathbb{F}_3\oplus V=W$.
   In particular, the three groups $\Delta_2\subset\Gamma$ are the  only normal Schottky subgroup of index $72$.
Therefore the curves $X=\Omega/\Delta_2$ are uniquely determined by the
embedding of $\Gamma$ into  ${\rm PGL}_2(K)$ and the choice of $V$.\\

\noindent 
{\it We will show that the assumption that there is a Mumford curve $X=\Omega/\Delta$ of genus $g=11$
with $|{\rm Aut}(X)|=144$ leads to a contradiction}. 

Assume the existence of  $\varphi:\Gamma \rightarrow H:=\Gamma/\Delta$ with $|H|=144$.\\
 Write $\Gamma^\prime={\rm  PSL}_2(\mathbb{F}_{3})\ast_{C_3}B(2,1)\cong A_4\ast_{C_3}C_3^2$.  
Then $\Gamma=C_2\ltimes \Gamma^\prime$. The following three steps
provide the contradiction.\\

\noindent (a). {\it Suppose that  a normal subgroup $\Gamma^\circ\subset\Gamma$ contains non-trivial
elements of finite order and that $\Gamma ^\circ$ is not contained in
$\Gamma^\prime$.  Then $\Gamma^\circ=\Gamma$}.\\
\noindent {\it Proof}.
Indeed, since $\Gamma^\circ\subset\Gamma$ is normal and contains
non-trivial elements of finite order,
the intersection $\Gamma^\circ\cap\Gamma_{v_1}$ or the intersection $\Gamma^\circ\cap\Gamma_{v_2}$
has to contain non-trivial elements.
Furthermore,  $\Gamma^\circ\cap\Gamma_{v_1}$ and $\Gamma^\circ\cap\Gamma_{v_2}$ are normal subgroups
of $\Gamma_{v_1}$ and $\Gamma_{v_2}$, respectively. 
Moreover, at least one of these intersections is not contained in $\Gamma^\prime$.

A normal subgroup of $\Gamma_{v_1}$ or of $\Gamma_{v_2}$ that is not contained in $\Gamma^\prime$
contains a subgroup $C_2$ that is not contained in
$\Gamma^\prime$. The subgroup $C_2$ stabilizes an edge
in $\mathcal{T}^c$. In particular, the normal group $\Gamma^\circ$ contains all the subgroups 
$C_2\subset\Gamma$ that stabilize an edge.

One verifies that for any vertex $v\in\mathcal{T}^c$ 
the subgroups $C_2\subset\Gamma_v$ that stabilize an edge $e\ni v$
generate the group $\Gamma_v$. {\it Therefore $\Gamma^\circ=\Gamma$
  holds}.\hfill $\square$\\

\noindent (b). {\it Let $\Delta\subset \Gamma$ be a maximal normal Schottky subgroup
  of finite index. The assumption that $\Delta$ is not contained in
  $\Gamma^\prime$ leads to a contradiction.}\\
{\it Proof}.
The order of  $\Gamma/\Delta$ is 144.  Since there is no simple group
of this order, it has a proper normal subgroup.  Its preimage
$\Gamma^\circ\subset \Gamma$ contains non trivial elements of finite
order (because $\Delta$ is maximal) and is not contained in
$\Gamma^\prime$. This implies the {\it contradiction}
$\Gamma^\circ=\Gamma$.\hfill $\square$\\

\noindent (c). {\it Let $\Delta\subset \Gamma$ be a maximal normal Schottky subgroup
  of finite index. The assumption $\Delta\subset \Gamma^\prime$ leads
  to a contradiction.}\\
Then $H'=\Gamma^\prime/\Delta$ has order 72 and contains a
subgroup isomorphic to $A_4$. As above there are four groups of order 72
having that property. Each one of them contains a subgroup of index
two (isomorphic to $A_4\times C_3$). We finish the proof by showing that $H'$ does not
have a subgroup of index two. Indeed, $H'$ is generated by the subgroups
$\varphi(\Gamma^\prime_{v_1})\cong A_4$ and
$\varphi(\Gamma^\prime_{v_2})\cong C_3^2$. Since $<g^2 \mid g\in
A_4>=A_4$ and $<g^2\mid g\in C_3^2>=C_3^2$, the group $H'$ has no
subgroup of index two.\end{proof}

\subsection{\rm Counting the number of elements of
  $\Gamma/\Delta$} \label{section 8.2}

In this section suitable values $N_0(\Gamma)$ are obtained for most of the groups
$\Gamma$ in table \ref{table-2}.
The groups $\Gamma$  treated are those where it is possible to identify enough distinct elements
in a quotient $\Gamma/\Delta$ without having any precise knowledge of the structure
of the quotient group.

\begin{proposition}\label{6.31}  \label{8.11} Let $\ell >5$, then  $N_0(\Gamma)=\ell \cdot p^{n_1+n_2}$ is suitable
  for  the amalgam  $\Gamma=B(n_1,\ell)\ast_{C_\ell}B(n_2,\ell)$.
\end{proposition}

\begin{proof} Consider a normal Schottky subgroup $\Delta\subset
  \Gamma$ of finite index and let $\varphi:\Gamma \longrightarrow
  \Gamma/\Delta$ denote the canonical map. We consider the set of
  elements $\varphi(g_1g_2)$ with $g_1\in B(n_1,\ell)$ and $g_2$
  in the unipotent subgroup $B(n_2,1)$ of $B(n_2,\ell)$.

  Suppose $\varphi (g_1g_2)=\varphi(g_3g_4)$. Then $\varphi(g_3^{-1}g_1)=\varphi
  (g_4g_2^{-1})$. If $g_4g_2^{-1}\in B(n_2,1)$ is not $1$, then it has order $p$.
  Now $g_3^{-1}g_1\in B(n_1,\ell)$ and since the restriction of $\varphi$
  to $B(n_1,\ell)$ is injective, $g_3^{-1}g_1$ has also order $p$ and
  therefore belongs to the unipotent subgroup $B(n_1,1)$ of
  $B(n_1,\ell)$. By definition of the amalgam $\Gamma$ and the
  assumption $\ell >5$ , the action by conjugation of
  the group $C_\ell$ on $B(n_1,1)$ and $B(n_2,1)$ is different. This
  yields a contradiction. Hence $g_2=g_4$ and $g_1=g_3$.

 We conclude that $|\Gamma /\Delta |\geq \ell \cdot p^{n_1+n_2}$. A somewhat
 long computation shows that $N_0(\Gamma):=\ell \cdot  p^{n_1+n_2}$
 and $g_0=(p^{n_1}-1)(p^{n_2}-1)$ satisfies  $N_0(\Gamma)\leq
 F(g_0)$.
\end{proof}

\begin{remark}\label{remark-8.12}{\rm  
Let $\Gamma$ be as above.
Let $\{v_1,v_2\}$ be the edge in the contracted tree
 $\mathcal{T}^c$ with vertex groups $\Gamma _{v_1}=B(n_1,\ell)$ and
 $\Gamma _{v_2}=B(n_2,\ell)$. We have shown that:  if  $g_1g_2 \in
 \Delta$ for $g_1\in B(n_1,\ell)$ and $g_2\in B(n_2,1)$, then $g_1=1$
 and $g_2=1$. This implies that for $\delta \in \Delta, \delta \neq
 1$, the vertex $\delta (v_1)$ is ``far away'' from $v_1$, which means
 that $\{\delta (v_1),v_2\}$ and $\{\delta(v_2),v_1\}$  are not edges. One concludes from this
 that the number of vertices of the graph $\mathcal{T}^c/\Delta$ is
 at least $1+\max (p^{n_1},p^{n_2})$.
}
\end{remark} 

\begin{proposition}\label{6.32} \label{8.12}
$N_0(\Gamma)=\ell (p^{2n}+p^n)$ is suitable for the amalgam\\
$\Gamma=B(n,\ell)\ast_{C_\ell}D_\ell$ and $\ell>5$. \end{proposition}

\begin{proof} $\Gamma =<B(n,\ell),\tau>$ where $\tau$ is an element of
  order two in $D_\ell$. The relations are $\tau ^2=1$ and $\tau
  {\zeta \ 0\choose 0\ 1 }\tau ={\zeta ^{-1}\ 0\choose 0\ 1 }$. Then
$\Gamma ^\prime :=<B(n,\ell ),\tau B(n,\ell)\tau>\cong B(n,\ell)\ast
_{C_\ell}B(n,\ell)$ is a normal subgroup of index 2 of $\Gamma$.\\

The contracted tree $\mathcal{T}^c$ of $\Gamma$ has an edge
$\{v_1,v_2\}$ such that $\Gamma_{v_1}=B(n,\ell)$ and $\Gamma
_{v_2}=D_\ell$. The vertex $v_2$ has only two edges. One sees that
$\mathcal{T}^c$ has two types of vertices. The vertices with stabiliser
isomorphic to $B(n,\ell)$, like $v_1$ and having $p^n$ edges.
The other ones are vertices with  stabiliser isomorphic to $D_\ell$ and have two
edges. After `contracting' the second type of vertices one obtains
the contracted tree for the subgroup $\Gamma ^\prime$.

One concludes from remark \ref{remark-8.12} that any
quotient $\mathcal{T}^c/\Delta$, where $\Delta$ is a Schottky
subgroup  of $\Gamma$ of finite index, has at least $p^n+1$ vertices
with stabiliser isomorphic to $B(n,\ell)$. Using that each vertex of
`type' $B(n,\ell)$ has $p^n$ edges and that each vertex of `type'
$D_\ell$ has two edges one finds that the genus of the graph
$\mathcal{T}^c/\Delta$ is $\geq g_0=1+ \frac{(p^n+1)(p^n-2)}{2}$. Then
$N_0(\Gamma)=\mu(\Gamma)^{-1}(g_0-1)=\ell (p^{2n}+p^n)$ is suitable
since it is $\leq F(g_0)$. We note that, by Theorem \ref{6.8}, one obtains an
equality $N_0(\Gamma)=F(g_0)$ for $\ell =p^n-1$. \end{proof}


\begin{proposition}\label{6.23} \label{8.13} Consider $\Gamma=G_1\ast_{G_3}G_2 $. \\
{\rm (1)} For $\Gamma={\rm PSL}_2(\mathbb{F}_{q})\ast_{B(n,\frac{q-1}{2})}B(3\cdot n,\frac{q-1}{2})$
 suitable bounds are\\ $N_0(\Gamma)=l.c.m.(|G_1|,|G_2|)$ for $q=3,5$ and $N_0(\Gamma)=q^6$ for $q>5$.\\
{\rm (2)} For $\Gamma={\rm
  PGL}_2(\mathbb{F}_{q})\ast_{B(n,q-1)}B(3\cdot n,q-1)$ suitable
bounds are\\  $N_0(\Gamma)=l.c.m.(|G_1|,|G_2|)$ for $q=3$ and $N_0(\Gamma)=q^6$ for $q>3$.
\end{proposition}

\begin{proof} For $q=3,5$, see Table \ref{table-1}.  Suppose  $q>5$ is odd and consider the amalgam $\Gamma:={\rm
    PSL}_2(\mathbb{F}_q)\ast_{B(n,\frac{q-1}{2})} B(nm,\frac{q-1}{2})$
  and a surjective homomorphism $\varphi :\Gamma \longrightarrow H$ for
  some finite group $H$ such that the kernel is a Schottky group.
The element $t={a\ \ \ 0\choose 0\ a^{-1} }$, where $a$ a
  primitive root of unity of order $q-1$, generates
  $C_{\frac{q-1}{2}}$. Put $w={0\ -1\choose 1\ \ \ 0 }$. Then
  $wtw^{-1}=t^{-1}$. Consider ${1\ \mathbb{F}_q^m\choose 0\ 1 }\subset
  {*\ \mathbb{F}_q^m\choose 0\ *^{-1} }=B(nm,\frac{q-1}{2})$ and let 
$U_+=\varphi ({1\ \mathbb{F}_q^m\choose 0\ 1 }), \ U_-=\varphi
(w)U_+\varphi(w)^{-1} $. 

Then conjugation by $\varphi (t)$ acts on
$U_+$ as multiplication by $a^2$ on $\mathbb{F}_q^m$. Conjugation by
$\varphi (t)$ an $U_-$ acts on $\mathbb{F}_q^m$ as multiplication by
$a^{-2}$. It follows that $U_+\cap U_-=\{1\}$ and that $U_+U_-$
consists of $q^m\times q^m$ elements and so $|H|\geq q^{2m}$.
 
For $m=3$ the value $N_0(\Gamma)=q^6$ is suitable.
For the group $\Gamma={\rm
  PGL}_2(\mathbb{F}_{q})\ast_{B(n,q-1)}B(3\cdot n,q-1)$ and any $q>3$
the same value of $N_0(\Gamma)$ is suitable.
\end{proof}

\subsection{\rm The group $\Gamma:={\rm  PGL}_2(\mathbb{F}_q)\ast_{B(n,q-1)}B(2n,q-1)$ with $q>3$} \label{section 8.3}

{\rm For the case $\Gamma:={\rm  PSL}_2(\mathbb{F}_q)\ast_{B(n,\frac{q-1}{2})}
B(nm,\frac{q-1}{2})$, studied in the proof of \ref{6.23}, we found a
suitable $N_0(\Gamma)$ for $m\geq 3$. For $m=2$ this fails and we have to develop a
rather different method. {\it We now present the long proof of the
existence of a suitable bound $N_0(\Gamma)$ for the cases}:\\
\[ \Gamma={\rm PGL}_2(\mathbb{F}_{q})\ast_{B(n,q-1)}B(2\cdot
n,q-1)\mbox{ with } q>3 \mbox{ and }  \]
\[ \Gamma ={\rm PSL}_2(\mathbb{F}_{q})\ast_{B(n,\frac{q-1}{2})}B(2\cdot
n,\frac{q-1}{2}) \mbox{ with } p\not=2,\  q>5.\]

{\it The strategy is as follows}. Let $\Delta \subset \Gamma$ be a
{\it `maximal' normal Schottky } subgroup of finite index. This means that
 any normal $\Gamma'\subset \Gamma $ strictly containing $\Delta$
has non-trivial elements of finite order. As remarked before, it suffices to consider
these maximal $\Delta$.  The proper normal subgroups of $\Gamma /\Delta$
corresponds to the normal subgroups $\Gamma'$ of $\Gamma$ with
$\Delta\subsetneq \Gamma '\subsetneq \Gamma$.

One determines the finitely many normal subgroups
$\Gamma^\prime\subset\Gamma$ of finite index that contain non-trivial elements of finite order.
These subgroups are used to obtain a factorization of
the quotient group $\Gamma/\Delta$ into finite simple groups.
An estimate of the minimal order of the simple groups occurring in the factorization 
of the quotient group $\Gamma/\Delta$ is then used to obtain suitable values $N_0(\Gamma)$.
The final result is Theorem \ref{6.30}.} \hfill  $\square$\\

\noindent {\it  Notation and definition}.\\
 $\Gamma=\Gamma_{v_1}\ast_{\Gamma_e}\Gamma_{v_2}={\rm
  PGL}_2(\mathbb{F}_{q})\ast_{B(n,q-1)}B(2\cdot n,q-1)$ where $e$ is
the edge $(v_1,v_2)$ in $\mathcal{T}^c$. The group $B(2\cdot n,1)$ is written as 
$\left(\begin{array}{cc} 1& \mathbb{F}_q\oplus \mathbb{F}_qx\\
         0&1 \end{array}\right)$, where $x\in K-\mathbb{F}_q$.
         The subgroup $B(n,1)\subset B(2\cdot n,1)$ that $B(2\cdot n,1)$ has in
     common with ${\rm PGL}_2(\mathbb{F}_q)$ is then the group $\left(\begin{array}{cc} 1& \mathbb{F}_q\\
         0&1 \end{array}\right)$. The other $q$ subgroups of $B(2\cdot
     n,1)$, isomorphic to $\mathbb{F}_q$, are  $B_a:=\left(\begin{array}{cc} 1&  \mathbb{F}_q(x-a)\\
         0&1 \end{array}\right)$ with $a\in \mathbb{F}_q$. The edges
     of $v_1$ are the cosets ${\rm PGL}_2(\mathbb{F}_q)/B(n,q-1)$ and
     the edges      of $v_2$ are the cosets $B(2\cdot
     n,q-1)/B(n,q-1)$. The group $B(n,q-1)$ fixes only one edge of  $ v_2$ (namely $e$)
     and the group $B_a$ leaves no edge of $v_2$ invariant.  \\

 {\it Let $\Gamma(B_a)\subset \Gamma$ denote the normal subgroup generated
 by $B_a$}.  
                                                                                        
  \begin{lemma}\label{6.25} \label{8.14} Let $\Gamma:= {\rm
    PGL}_2(\mathbb{F}_{q})\ast_{B(n,q-1)}B(2\cdot n,q-1)$. \\
  Then $\Gamma\cong \Gamma_{v_1}\ltimes \Gamma(B_a) ={\rm
    PGL}_2(\mathbb{F}_{q})\ltimes \Gamma(B_a)$
and $\Gamma(B_a)\cap \Gamma_{v_2}=B_a$.\\
 If $p>2$, then ${\rm PSL}_2(\mathbb{F}_{q})\ast_{B(n,\frac{q-1}{2})}B(2\cdot n,\frac{q-1}{2})=
           {\rm PSL}_2(\mathbb{F}_{q})\ltimes \Gamma(B_a)$.
\end{lemma}

\begin{proof} $\Gamma$ is generated by ${\rm PGL}_2(\mathbb{F}_q)$ and
  $B_a$. There is a unique homomorphism $\varphi_a:\Gamma \rightarrow
  {\rm PGL}_2(\mathbb{F}_q)$ which is the identity on ${\rm PGL}_2(\mathbb{F}_q)$
  and maps $B_a$ to $1$. The kernel of $\varphi_a$ contains $\Gamma(B_a)$
  and cannot be larger because the image of $\varphi_a$ is
  ${\rm PGL}_2(\mathbb{F}_q)$. If $\ker \varphi_a\cap \Gamma_{v_2}$ is greater
  than $B_a$, then one finds the contradiction $\ker \varphi_a\cap
  \Gamma_{v_1}\neq \{1\}$. From this the first statements follow. The final statement has a similar proof.
\end{proof}

\begin{remark} {\it Description of $\Gamma(B_a)$ as an amalgam. } {\rm \\
The group $B_a$ is normal in $B(2\cdot n,q-1)$ and acts transitively on the edges $e\ni v_2$.
Therefore $\Gamma (B_a)$ is generated by the groups $hB_ah^{-1}$ where $h\in
  \Gamma_{v_1}$ runs in a set of representatives of
  $\Gamma_{v_1}/\Gamma_{e}$. 
  The intersection $B_a\cap hB_ah^{-1}\subset \Gamma_{v_1}$ is trivial if $h(v_2)\not= v_2$.
  In particular, $\Gamma(B_a)$ is the free product of the $q+1$ groups $hB_ah^{-1}$. 
  }
\end{remark}

\begin{lemma}\label{6.26} \label{8.15}
Let $\Gamma^\prime\subset\Gamma$ be a proper normal subgroup containing non-trivial elements of finite order.
Then the following statements hold:
\begin{enumerate}
\item[i)] If $\Gamma={\rm PSL}_2(\mathbb{F}_{q})\ast_{B(n,\frac{q-1}{2})}B(2\cdot n,\frac{q-1}{2})$,
          $p\not=2$, $q>3$, then $\Gamma '=\Gamma(B_a)$
 for some $a\in  \mathbb{F}_q$.
\item[ii)] If $\Gamma={\rm PGL}_2(\mathbb{F}_{q})\ast_{B(n,q-1)}B(2\cdot n,q-1)$, $p=2$, 
         $q\geq 4$, then $\Gamma^\prime=\Gamma(B_a)$ for some $a\in\mathbb{F}_q$.
\item[iii)] If $\Gamma={\rm PGL}_2(\mathbb{F}_{q})\ast_{B(n,q-1)}B(2\cdot n,q-1)$, $p\not=2$, 
           $q>3$, then $\Gamma^\prime={\rm PSL}_2(\mathbb{F}_{q})\ast_{B(n,\frac{q-1}{2})}B(2\cdot n,\frac{q-1}{2})$
           or $\Gamma '=\Gamma(B_a)$ for some $a\in \mathbb{F}_q$.
\end{enumerate}
\end{lemma}

\begin{proof}
We only prove statement (iii). The proofs for the other cases are
similar. The intersection $\Gamma '\cap \Gamma_{v_1}$ is normal in ${\rm
  PGL}_2(\mathbb{F}_q)$ and ${\rm PSL}_2(\mathbb{F}_q)$ is simple for
$q>3$. Therefore we only have to consider the following cases.\\ 

\noindent 
(1). {\it Suppose that }$\Gamma'\cap \Gamma_{v_1}=\{1\}$, then $\Gamma '\cap
\Gamma_{v_2}$ is a non trivial normal subgroup of $\Gamma_{v_2}$ and
has intersection $\{1\}$ with $\Gamma_e$. It follows that $\Gamma'\cap
\Gamma_{v_2}=B_a$ for some $a\in \mathbb{F}_q$ and therefore
$\Gamma'=\Gamma (B_a)$.\\
(2). {\it Suppose that} $\Gamma'\cap \Gamma_{v_1}={\rm PSL}_2(\mathbb{F}_q)$, then
$\Gamma'\cap \Gamma_e=B(n,\frac{q-1}{2}) \subset \Gamma'\cap
\Gamma_{v_2}$. Take $h={1\ b\choose 0\ 1 }\in \Gamma_{v_2}$ and
$t={\zeta \ 0\choose 0\ 1 }\in B(n,\frac{q-1}{2})$. Since $\Gamma' \cap
\Gamma_{v_2}$ is normal in $\Gamma_{v_2}$, it  contains the element $hth^{-1}t^{-1}={1\
  b(1-\zeta)\choose 0\ \ \ \ 1 }$. It follows that $B(2\cdot
n,\frac{q-1}{2})\subset \Gamma'\cap \Gamma_{v_2}$. This is an
equality, since otherwise $\Gamma '\cap \Gamma_{v_1}$ contains an
element outside ${\rm PSL}_2(\mathbb{F}_q)$. Thus
$\Gamma'= {\rm PSL}_2(\mathbb{F}_{q})\ast_{B(n,\frac{q-1}{2})}B(2\cdot n,\frac{q-1}{2})$.\\
(3). {\it Suppose that} $\Gamma'\cap \Gamma_{v_1}=\Gamma_{v_1}$.  Then
$\Gamma'\cap \Gamma_{v_2}$ is a normal subgroup of $\Gamma_{v_2}$ containing
$\Gamma_e$. As in (2) one concludes that $\Gamma '\cap
\Gamma_{v_2}=\Gamma_{v_2}$ and $\Gamma'=\Gamma$.
\end{proof}

\begin{lemma}\label{6.27} \label{8.16}
Let $q>3$.
Let $\Gamma$ be ${\rm PGL}_2(\mathbb{F}_{q})\ast_{B(n,q-1)}B(2\cdot
n,q-1)$ if $p=2$ and
${\rm PSL}_2(\mathbb{F}_{q})\ast_{B(n,\frac{q-1}{2})}B(2\cdot n,\frac{q-1}{2})$ if $p\not=2$.\\
Let $\Delta\subset\Gamma$ be a maximal normal Schottky subgroup of
finite index.\\
Then either $H=\Gamma/\Delta$ is simple or
for some $a\in \mathbb{F}_q$ one has  $\Delta \subset \Gamma(B_a)\subset\Gamma$
with  $H\cong {\rm PSL}_2(\mathbb{F}_{q})\ltimes (\Gamma(B_a)/\Delta)$
and for all proper normal subgroups $N\subset\Gamma(B_a)$ that contain $\Delta$, the equality
$\Delta=\cap_{g\in {\rm PSL}_2(\mathbb{F}_q)} gNg^{-1}$ holds.
\end{lemma}

\begin{proof} We may suppose that $H=\Gamma /\Delta$ is not simple and
  has a minimal proper normal subgroup $N^\prime$. The preimage $\Gamma '\subset \Gamma$ of $N^\prime$
  contains a non-trivial element of finite order since $\Delta$ is
  maximal. Then $\Gamma'= \Gamma(B_a)$ follows from \ref{6.26} and 
   $H\cong {\rm PSL}_2(\mathbb{F}_{q})\ltimes
   (\Gamma(B_a)/\Delta)$.
   
   If the group $\Gamma(B_a)$ contains a proper normal subgroup $N\supset\Delta$,
   then the group $N$ is either $N=\Delta$ or is not stabilized
   by the group ${\rm PSL}_2(\mathbb{F}_{q})$. 
   In both cases $\Delta\subset\cap_{g\in{\rm PSL}_2(\mathbb{F}_{q})} gN g^{-1}$.
   Since $\Delta$ is maximal, equality must hold. This proves the lemma.
\end{proof}

\begin{example}\label{6.28} \label{8.17} Automorphism groups of Mumford curves $X$
  satisfying \\
  ${\rm Aut}(X)\cong \{ (g_1,g_2)\in  {\rm PGL}_2(\mathbb{F}_{q})^2 \mid \det(g_1)=\det(g_2)\}$.\\
 {\rm \indent Let $\Gamma$ be ${\rm PSL}_2(\mathbb{F}_{q})\ast_{B(n,\frac{q-1}{2})}B(2\cdot n,\frac{q-1}{2})$.
 For $a\in \mathbb{F}_q$ we define as before $\varphi_a:\Gamma\rightarrow
 {\rm PSL}_2(\mathbb{F}_q)$ by $\varphi_a$ is the `identity' on ${\rm
   PSL}_2(\mathbb{F}_q)$ and ${1\ \lambda x\choose 0\ \ \  1 }$ (with
 $\lambda \in \mathbb{F}_q$) is mapped to ${1\ \lambda a\choose 0\ 
   \ \ 1}$. The kernel of $\varphi_a$ is $\Gamma(B_a)$ (the normal
 subgroup generated by ${1\ \mathbb{F}_q(x-a)\choose 0\ \ \ \ \ \ \
   1}$) and $\Gamma\cong {\rm PSL}_2(\mathbb{F}_q)\ltimes \Gamma(B_a)$.  \\ 

 For $a,a'\in \mathbb{F}_q$ with $a\neq a'$, one considers
 $\varphi_{a,a'}=(\varphi_a,\varphi_{a'}):\Gamma \rightarrow {\rm
   PSL}_2(\mathbb{F}_q)\times {\rm PSL}_2(\mathbb{F}_q)$. The kernel $\Delta_{a,a'}$
 of $\varphi_{a,a'}$ is $\Gamma (B_a)\cap \Gamma (B_{a'})$ and clearly
 has no elements of finite order $\neq 1$. The image of ${\rm
   PSL}_2(\mathbb{F}_q)$ under $\varphi_{a,a'}$ is the diagonal embedding
 of this group. The image of ${1\ \lambda x\choose 0\ \  1}$, for
 $\lambda \in \mathbb{F}_q$,  is $({1\ \lambda a\choose 0\ \ 
   1},{1\ \lambda a'\choose 0\ \  1})$. From this one easily
 deduces that $\varphi_{a,a'}$ is surjective.\\
 
We conclude that $\Delta_{a,a'}$ is a maximal normal Schottky
subgroup. It is contained in $\Gamma (B_a)$ and $\Gamma/\Delta_{a,a'}$
is isomorphic to ${\rm PSL}_2(\mathbb{F}_{q})\ltimes (\Gamma(B_a)/\Delta_{a,a'})$ and
$\Gamma(B_a)/\Delta_{a,a'}$ is simple because it is isomorphic to ${\rm PSL}_2(\mathbb{F}_q)$.\\

The group $\Delta_{a,a'}$ is also normal in 
$\Gamma={\rm PGL}_2(\mathbb{F}_{q})\ast_{B(n,q-1)}B(2\cdot n,q-1)$
(for $p\neq 2$ and $q>3$).
The quotient $\Gamma/\Delta_{a,a'}$ is isomorphic to the group $\{ (g_1,g_2)\in  {\rm PGL}_2(\mathbb{F}_{q})^2 \mid
   \det(g_1)=\det(g_2)\}$.      
We note that $|\Gamma/\Delta_{a,a'}|=\frac{(q^3-q)^2}{2}$ and this is
the value for $N_0(\Gamma)$, see Table \ref{table-2},  that we will prove in \ref{6.30}. 
} \end{example}

\begin{lemma}\label{example-8.2} \label{8.18} 
 Let $\Gamma$ be ${\rm PGL}_2(\mathbb{F}_{q})\ast_{B(n,q-1)}B(2\cdot n,q-1)$ with $p=2$ or
 ${\rm PSL}_2(\mathbb{F}_{q})\ast_{B(n,\frac{q-1}{2})}B(2\cdot n,\frac{q-1}{2})$ with $p\not=2$.
Let $\Delta_c=[\Gamma(B_a),\Gamma(B_a)]$ denote the commutator
subgroup of $\Gamma(B_a)$. Then $\Delta_c$ is a normal Schottky
subgroup of $\Gamma$ and  $\Gamma /\Delta _c \cong {\rm PSL}_2(\mathbb{F}_q)\ltimes B_a^{q+1}$.
 For  $q>3$ the group $\Delta_c$ is maximal. $\Delta_c$ is not maximal for $q=2,3$. 
 \end{lemma}

\begin{proof}
 The group $B_a$ is normal in $B(2\cdot n,\frac{q-1}{2})$ (resp. $B(2\cdot n,q-1)$ if $p=2$).
 Therefore the group $\Gamma (B_a)$ is generated by the subgroups
 $G_i:=h_iB_ah_i^{-1}$ where $\{h_1,\dots ,h_{q+1}\}$ is a set of representatives of
  $\Gamma_{v_1}/\Gamma_{e}$. The group $\Gamma(B_a)$ is in  fact the
  free product $G_1\ast G_2\ast \cdots \ast G_{q+1}$ of these $q+1$
  groups (compare \ref{8.14}). Since the groups $G_i$ are commutative, it
follows that $\Delta_c=[\Gamma(B_a),\Gamma(B_a)]$ is generated by the commutators
$aba^{-1}b^{-1}$ with $a\in G_i,b\in G_j$ for all pairs $(i,j)$ with
$i\neq j$. Further $\Gamma(B_a)/\Delta_c\cong G_1\times \cdots \times
G_{q+1}\cong B_a^{q+1}=C_p^{n(q+1)}$ and $\Delta_c$ is a Schottky group.\\

 {\it Observations}.
The group ${\rm PSL}_2(\mathbb{F}_q)$ acts, by conjugation, transitively both on the groups $G_i$ and
on the groups $C_p\subset G_i$, $i=1,\ldots,q+1$.
Let $N:=\varphi^{-1}(C_p)$ for some group $C_p$ contained in one of the groups $G_i$.
Then the intersection $\cap_{g\in {\rm PSL}_2(\mathbb{F}_q)}gNg^{-1}$ is the group $\Delta_c$.\\
\indent We note that $\Delta_c\subset\Gamma(B_a)$ is not a maximal normal
Schottky group in $\Gamma(B_a)$. Indeed, consider
$\varphi':\Gamma(B_a)=G_1\ast \cdots \ast G_{q+1}\rightarrow B_a$ such
that the restriction of $\varphi'$ to each $G_i$ is a bijection. Then
$\ker(\varphi ')$ is a maximal normal Schottky group of $\Gamma(B_a)$, containing  
$\Delta_c$. \\

{\it Continuation of the proof}. Let $\Delta \supset \Delta_c$ be a maximal normal Schottky group
in $\Gamma$ and write
$\varphi:\Gamma \rightarrow \Gamma/\Delta$. The group $\Gamma$ acts, by
conjugation, on $\Gamma(B_a), \Delta_c$ and $\Gamma(B_a)/\Delta_c=G_1\times
\cdots \times G_{q+1}$. The stabilizers of the groups $G_i$ are the distinct
Borel subgroups $B_i$ (i.e., $B(n,\frac{q-1}{2})$ for $p\neq 2$ or
$B(n,q-1)$ for $p=2$). The stabilizer for a pair $G_i,G_j,\ i\neq j$
is a torus $T=B_i\cap B_j$ ($\cong C_{\frac{q-1}{2}}$ for $p\neq 2$ and $\cong C_{q-1}$ for $p=2$).  
For $q>3, \ q\neq 5$, the cyclic group $T$ acts with different
character on $G_i$ and $G_j$. It follows that $\varphi(G_i)\cap \varphi
(G_j)=\{1\}$ for $i\neq j$. Therefore $\varphi (\Gamma(B_a))$ contains
$q+1$ different groups isomorphic to $B_a$. \\

{\it $\Delta=\Delta _c$ (and so $\Delta_c$ maximal) 
follows from $\varphi (\Gamma(B_a))=\prod_{i=1}^{q+1}\varphi(G_i)$}.\\
First we prove that $<\varphi(G_j)|\;
j\not=i>=\prod_{j\not=i}\varphi(G_j)$ for any $i$, $1\leq i\leq q+1$.
Let $U_i\subset B_i= B(n,q-1)$ be the normal subgroup $U_i=B(n,1)$.
The group $U_i$ is the stabilizer of every non-trivial element $g\in\varphi(G_i)$.
Moreover, $U_i$ permutes the groups $\varphi(G_j)$, $j\not=i$ simply transitively.
Therefore the only elements $g\in\langle\varphi(G_j)|\; j\not=i\rangle$ that are stabilized by a group
$U_{j_0}$, $j_0\not=i$ are the elements $g\in \varphi(G_{j_0})$.
From this one concludes that 
 $\varphi(G_{j_0})\ \cap\langle\varphi(G_j)|\; j\not=i, j_0\rangle=\{ 1\}$.
In particular, the equality $<\varphi(G_j)|\; j\not=i>=\prod_{j\not=i}\varphi(G_j)$ holds.\\

{\it Now we show that $\varphi(G_i)\ \cap\langle\varphi(G_j)|\; j\not=i\rangle=\{ 1\}$
and the equality  $\varphi(\Gamma(B_a))=\prod_{i=1}^{q+1}\varphi(G_i)$
follows from that.}\\
Let $\psi_i$ be the map defined by $\psi_i(g)=\prod_{u\in U_i}ugu^{-1}$ for $g\in \Gamma(B_a)/\Delta_c$.
Since the group $\Gamma(B_a)/\Delta_c$ is abelian and $\psi_i(1)=1$,
the map $\psi_i$ is actually a group homomorphism.
One verifies that $\psi_i(G_i)=\{ 1\}$
and that $\psi_i(\Gamma(B_a)/\Delta_c)\subset\langle G_j|\; j\not=i\rangle$.
Let $G_i^\vee$ denote the image $G_i^\vee:= im(\psi_i)$.
For any $j\not=i$ one has $G_i^\vee=\langle \psi_i(g)|\; g\in G_j\rangle \cong G_j$.\\

{\it We will derive a contradiction from  $\varphi(G_i)\ \cap\langle\varphi(G_j)|\; j\not=i\rangle \not=\{ 1\}$}.
The only elements of $\langle \varphi(G_j)|\; j\not=i\rangle $ that are stabilized by the group 
$U_i\subset {\rm PSL}_2(\mathbb{F}_q)$ are the elements $\varphi(\psi_i(g))$ 
with $g\in \Gamma(B_a)/\Delta_c$.
Therefore $\varphi(G_i)\ \cap\langle\varphi(G_j)|\; j\not=i\rangle \subset\varphi(G_i^\vee)$ holds.
The $B_i$-orbit of a non-trivial element of
both  $\varphi(G_i)$ and $\varphi(G_i^\vee)$ generates the entire group 
$\varphi(G_i)$ and $\varphi(G_i^\vee)$, respectively.
In particular, if the intersection $\varphi(G_i)\cap\varphi(G_i^\vee)$ is non-trivial,
then the equality $\varphi(G_i)=\varphi(G_i^\vee)$ must hold.\\

Therefore the group $\Delta/\Delta_c\cong B(n,1)$ consists of elements of the form
$g_i\psi_i(g_j)$ with $g_i\in G_i$ and $g_j\in G_j$, $j\not=i$.
Since the group $U_i$ acts trivially on both $G_i$ and $G_i^\vee$,
it acts trivially on the group $\Delta/\Delta_c$.
The groups $U_j\subset {\rm PSL}_2(\mathbb{F}_q)$, $j=1,\ldots,q+1$ are conjugated
and therefore also act trivially on $\Delta/\Delta_c$. 
Since the groups $U_j$, $1\leq j\leq q+1$ generate the group ${\rm PSL}_2(\mathbb{F}_q)$,
the entire group ${\rm PSL}_2(\mathbb{F}_q)$ must act trivially on $\Delta/\Delta_c$.\\

{\it This contradicts the fact that $T=B_i\cap B_j$, $j\not=i$ acts non-trivially on $\Delta/\Delta_c$}.
Indeed, for any $t\in T$, $t\not=1$, one has
$tg_i\psi_i(g_j)t^{-1}=tg_it^{-1}\cdot \psi_i(tg_jt^{-1})\not=g_i\psi_i(g_j)$.
Therefore $\varphi(G_i)\cap\langle\varphi(G_j)|\; j\not=i\rangle =\{ 1\}$ must hold.
{\it Hence $\Delta=\Delta_c\subset\Gamma(B_a)$ is maximal if $q\not=5$}.\\

{\it Consider the case  $q=5$}.
The group $T=C_2$ acts with the same character on both $G_i$ and $G_j$.
However, there does not exist a group ${\rm PSL}_2(\mathbb{F}_5)\ltimes C_5$
such that an element $g\in {\rm PSL}_2(\mathbb{F}_5)$ of order two
acts as $-1$ on the group $C_5$.
Indeed, all elements of order two in ${\rm PSL}_2(\mathbb{F}_5)$ are conjugated
and must therefore act as $-1$ on $C_5$.
On the other hand, the group ${\rm PSL}_2(\mathbb{F}_5)$ contains subgroups
$C_2\times C_2$.
It is not possible for all three non-trivial elements in a group $C_2\times C_2$
to act as $-1$ on a group $C_5$.
As a consequence also for $q=5$ one has that 
$\varphi(G_i)\cap\varphi(G_j)=\{ 1\}$.
Therefore $\Delta_c=[\Gamma(B_a),\Gamma(B_a)]\subset\Gamma$
is also a maximal normal subgroup without elements of finite order if $q=5$. \\

{\it For the cases $q=2,3$}  the Borel subgroups are isomorphic to $B(1,1)$. 
In particular, the cyclic group $T$ stabilizing two distinct groups $G_i$ and $G_j$
is the trivial group. In this case the group $\Delta_c\subset\Gamma$ is not maximal.
In both cases there exists a map 
$\varphi_q:\;\Gamma\longrightarrow  {\rm PSL}_2(\mathbb{F}_q)\times C_q$
such that the kernel $ker(\varphi_q)\subset\Gamma$ is a maximal normal
Schottky group. Moreover, $\Delta_c\subset ker(\varphi_q)$.
If $q=3$, the kernel $ker(\varphi_3)$ is the group
$\Delta_2\subset\Gamma$ studied in \S \ref{section 8.1}. 
 \end{proof}

\begin{remark}{\rm
We note that the group $\Delta_c=[\Gamma(B_a),\Gamma(B_a)]$ is also a normal Schottky
group of the amalgam ${\rm PGL}_2(\mathbb{F}_{q})\ast_{B(n,q-1)}B(2\cdot n,q-1)$ if $p>2$.
It is maximal for $q\not= 2,3$. The quotient is a group ${\rm PGL}_2(\mathbb{F}_q)\ltimes B_a^{q+1}$.
}
\end{remark}

\begin{lemma}\label{8.19}
Let $q>3$.
Let $\Gamma$ be ${\rm PGL}_2(\mathbb{F}_{q})\ast_{B(n,q-1)}B(2\cdot
n,q-1)$ if $p=2$ and
${\rm PSL}_2(\mathbb{F}_{q})\ast_{B(n,\frac{q-1}{2})}B(2\cdot n,\frac{q-1}{2})$ if $p\not=2$.\\
Let $\Delta\subset \Gamma$ be a maximal normal Schottky group such that
$\Delta \subset \Gamma(B_a)$ for some $a$. Assume that there exists a proper normal subgroup
$N\subset \Gamma(B_a)$, $N\not=\Delta$ that contains $\Delta$.
We may choose $N$ to be minimal.
Let $N_i$, $i\in I=\{ 1,\dots,s\}$ denote the ${\rm PSL}_2(\mathbb{F}_q)$-conjugates of $N=N_1$.
Then the following statements hold:
\begin{enumerate}
\item[i)] $\varphi(N_i)\cap \varphi(N_j)=\{ 1\}$ if $i\not=j$.
\item[ii)] $\langle N_i|\; i=1,\ldots. s\rangle =\Gamma(B_a)$.
\item[iii)] The groups $\varphi(N_i)$ and $\varphi(N_j)$ commute if $i\not=j$.
\item[iv)] If $\varphi(N)$ is non-abelian, 
           then $\varphi(\Gamma(B_a))=\Gamma(B_a)/\Delta=\prod_{i\in I} \varphi(N_i)$.
\item[v)] If $\varphi(N)$ is abelian, then $\Delta=[\Gamma(B_a),\Gamma(B_a)]$ and
        $\varphi(\Gamma(B_a))=B_a^{q+1}\cong C_p^{n(q+1)}$.
\end{enumerate}
\end{lemma}

\begin{proof}
Since $N_i$ and $N_j$ are minimal, the intersection $N_i\cap N_j=\Delta$ if $i\not=j$.
This proves statement (i).
The group $\langle N_i|\; i=1,\ldots. s\rangle \supset\Delta$ is stabilized by $\Gamma_{v_1}$.
Since $\Delta$ is maximal, the $\Gamma_{v_1}$-invariant group
$<N_i|\; i=1,\ldots. s>$ must equal $\Gamma(B_a)$.

Let $a_i, b_i\in \varphi(N_i)$ and $a_j, b_j\in\varphi(N_j)$
be elements such that $a_ib_j=a_jb_i$ holds.
Then $a_j^{-1}(a_ib_ja_i^{-1})=b_ia_i^{-1}\in \varphi(N_i)\cap \varphi(N_j)=\{ 1\}$.
Therefore $b_i=a_i$ holds.
Similarly $(b_j^{-1}a_ib_j)b_i^{-1}=b_j^{-1}a_j\in \varphi(N_i)\cap \varphi(N_j)$ and $b_j=a_j$ holds.
In particular, the groups $\varphi(N_i)$ and $\varphi(N_j)$ commute.
This proves (iii).

Let us now consider statement (iv).
We prove the statement by induction on the number of groups $N_j$.
Let $N_{\leq i}$ be the group $N_{\leq i}=<N_j|\; j\leq i>$.
The statement holds trivially for $i=1$  and follows from statement (i) for $i=2$.
So let us assume the statement holds for all $j\leq i$ and show that the statement holds for $i+1\leq |I|$.
Since the group $\varphi(N_{i+1})$ commutes with the groups $\varphi(N_j)$, $j\leq i$, 
the group $\varphi(N_{\leq i})$ commutes with the group $\varphi(N_{i+1})$.

Let us now show that the intersection $\varphi(N_{\leq i})\cap \varphi(N_{i+1})$ is trivial. 
Let $g$ be an element in the intersection $\varphi(N_{\leq i})\cap \varphi(N_{i+1})$.
Then $g=g_1\cdots g_i$ with $g_j\in N_j$ for $j=1,\ldots,i$.
Since the groups $\varphi(N_{i+1})$  and $\varphi(N_{\leq i})$ commute,
the elements $g_j$ are contained in the centralizer $Z_{\varphi(N_j)}(\varphi(N_j))$ for $j=1,\ldots,i$.
The group $\varphi(N_j)\subset\varphi(\Gamma(B_a))$ is normal.
In particular, the centralizer $Z_{\varphi(N_j)}(\varphi(N_j))\subset\varphi(\Gamma(B_a))$ 
is also a normal subgroup.
Since the group $\varphi(N_j)$ is non-abelian and a minimal normal subgroup,
the centralizer $Z_{\varphi(N_j)}(\varphi(N_j))$ is trivial.
Therefore $g_j=1$ for $j=1,\ldots,i$ 
and the intersection $\varphi(N_{\leq i})\cap \varphi(N_{i+1})$ is trivial.
In particular, $\varphi(N_{\leq i+1})=\varphi(N_{\leq i})\times \varphi(N_{i+1})$.
By induction statement (iv) holds. 
Statement (v) has been treated in lemma \ref{example-8.2}. 
\end{proof}

\begin{lemma}\label{6.29} \label{8.20}
Let $G$ be a non-abelian finite simple group 
that contains a subgroup $B(s,1)$ of order $p^s$.
Then $|G|\geq |{\rm PSL}_2(\mathbb{F}_{p^s})|=(p^{3s}-p^s)/2$ if $p>2$
and $|G|\geq |{\rm PSL}_2(\mathbb{F}_{p^s})|=p^{3s}-p^s$ if $p=2$.
\end{lemma}

\begin{proof}
Theorem A of \cite{V} states the following: Let $G$ be a simple
non-abelian group and $G\not={\rm PSL}_2(\mathbb{F}_{a})$ for any
prime power $a$. Let $A\subset G$ be an abelian subgroup. Then $|G| >
|A|^3$. The proof of this statement uses the classification of the finite simple groups.\\

{\it Suppose now $s\geq 2$}. We can exclude $G\not={\rm
  PSL}_2(\mathbb{F}_q))$ for $q=p^n, s\leq n$, because then $|G|>p^{3s}>|{\rm PSL}_2(\mathbb{F}_{p^s})|$.
For $G={\rm PSL}_2(\mathbb{F}_{q})$ with $q=p^n, s\leq n$ one has
$|G|=|{\rm PSL}_2(\mathbb{F}_{q})|\geq |{\rm PSL}_2(\mathbb{F}_{p^s})|$.\\

{\it Suppose now} $s=1$. Then we have to exclude $G={\rm
  PSL}_2(\mathbb{F}_a)$ for two cases, namely $a$ is a power $p^n$, $n>s$ of $p$
and the case $a=(p^\prime)^t$, $p^\prime\neq p$ and $p|\frac{a\pm 1}{2}$. The first
case is handled is above. In the second case the to be excluded groups $G$
have orders $|G|=(a^3-a)/2\geq (p^3-p)/2$ if $p,p^\prime\not=2$.
If $p^\prime=2$ and $p\not=2$, then $|G|=a^3-a\geq (p^3-p)/2$.
If $p^\prime\not=2$ and $p=2$, then $a\geq 3$ and $|G|=(a^3-a)/2\geq p^3-p=6$.
This proves the lemma.
\end{proof}

\begin{lemma}\label{lemma-8.x} \label{8.21}
Let $q>3$.
Let $\Gamma$ be a group ${\rm PGL}_2(\mathbb{F}_{q})\ast_{B(n,q-1)}B(2\cdot n,q-1)$ if $p=2$
and a group 
${\rm PSL}_2(\mathbb{F}_{q})\ast_{B(n,\frac{q-1}{2})}B(2\cdot n,\frac{q-1}{2})$ if $p\not=2$.
Let $\Delta\subset\Gamma$ be a maximal normal subgroup that contains no elements of finite order
and is contained in a subgroup $\Gamma(B_a)$.
Then $|\Gamma(B_a)/\Delta|\geq |{\rm PSL}_2(\mathbb{F}_q)|$.
\end{lemma}

\begin{proof}
Let us first consider the case where the quotient $\Gamma(B_a)/\Delta$ is simple. 
The group $\Gamma$ contains a subgroup $C_{\frac{q-1}{2}}$ if $p>2$
(and $C_{q-1}$ if $p=2$) that
normalizes the subgroup $B_a\subset\Gamma(B_a)$.
If $p\not=2$ and $q>5$, then $\frac{q-1}{2}>2$
and  if $p=2$ and $q>3$, then $q-1>2$.
In particular, the action is not restricted to $\pm 1$ and this
implies that the quotient $\Gamma(B_a)/\Delta$ contains at least two distinct subgroups 
isomorphic to $B_a$.
In particular, the simple quotient group $\Gamma(B_a)/\Delta$ is non-abelian if $q>3$, $q\not=5$.
If $q=5$, then it follows from the non-existence of a group 
${\rm PSL}_2(\mathbb{F}_5)\ltimes C_5$ such that the elements of order two in
${\rm PSL}_2(\mathbb{F}_5)$ acts as $-1$ on the group $C_5$ (see \ref{example-8.2})
that the quotient is non-abelian. 
Therefore $|\Gamma(B_a)/\Delta|\geq
|{\rm PSL}_2(\mathbb{F}_{q})|$ (see \ref{6.29}).
In particular, the lemma holds if the quotient is simple. \\

Let us now assume that there exists a proper normal subgroup $N\subset \Gamma(B_a)$,
$N\not=\Delta$ such that $\Delta=\cap_{g\in {\rm PSL}_2(\mathbb{F}_q)} gNg^{-1}$.
Let us first consider the case where $\varphi(\Gamma(B_a))$ is abelian.
Then $\Delta=[\Gamma(B_a),\Gamma(B_a)]$ and $\varphi(\Gamma(B_a))=B_a^{q+1}$.
Therefore the order of $\varphi(\Gamma(B_a)$ equals $q^{q+1}> q^3>|{\rm PSL}_2(\mathbb{F}_q)|$,
since $q>3$.

Let us now consider the case where $\varphi(\Gamma(B_a))$ is non-abelian.
We may assume that $N\subset \Gamma(B_a)$, $N\not=\Delta$ is a minimal normal subgroup
containing the group $\Delta$. 
Since $\Gamma(B_a)$ contains elements of order $p$ and $\Delta$ contains no elements of finite
order the group $\varphi(\Gamma(B_a))$ contains elements of order $p$.\\

By the lemma above $\varphi(\Gamma(B_a))\cong \varphi(N)^s$, 
where $s=|I|=|\{gNg^{-1}|\; g\in {\rm PSL}_2(\mathbb{F}_q)\}|$.
Therefore $p$ divides the order of $\varphi(N)$.
Since $\varphi(N)$ is non-abelian, we may assume  $|\varphi(N)|\geq p+1$.
Then  $|\varphi(\Gamma(B_a))|\geq (p+1)^s$.

To obtain a lower bound for the order of $\varphi(\Gamma(B_a))$ we need to determine
a lower bound for the value of $s$.
The integer $s=|\{gNg^{-1}|\; g\in {\rm PSL}_2(\mathbb{F}_q)\}|$ equals the index of the
stabilizer of N in the group ${\rm PSL}_2(\mathbb{F}_q)$.
Therefore we have to determine the minimal index of a proper subgroup.
We only have to consider maximal subgroups of ${\rm PSL}_2(\mathbb{F}_q)$.
The relevant subgroups are the groups
$A_4,\; A_5,\; S_4,\; D_{\frac{q\pm 1}{2}},\; B(n,\frac{q-1}{2})$
and groups ${\rm PGL}_2(\mathbb{F}_{p^s})$, ${\rm PSL}_2(\mathbb{F}_{p^s})$ with $s|n$ if $p>2$
and $D_{q\pm 1},\; B(n,q-1)$ and ${\rm PGL}_2(\mathbb{F}_{p^s})$ with $s|n$ if $p=2$.

Direct calculation shows that the group $B(n,\frac{q-1}{2})$ (resp., $B(n,q-1)$ if $p=2$) 
is of minimal index $q+1$ if $q\not=5,9$.
If $q=5$ then $A_4\subset {\rm PSL}_2(\mathbb{F}_5)\cong A_5$ has minimal index $5$
and if $q=9$, then $A_5\subset {\rm PSL}_2(\mathbb{F}_9)\cong A_6$ has minimal index $6$.
We leave it to the reader to verify that $(p+1)^{q+1}\geq q^3>|{\rm PSL}_2(\mathbb{F}_q)|$ for $q>3$
and that $|\varphi(\Gamma(B_a))|\geq |{\rm PSL}_2(\mathbb{F}_q)|$ also holds for $q=5,9$.
\end{proof}

\begin{theorem}\label{6.30} \label{8.22}
 Suitable  $N_0(\Gamma)$ for the following  amalgams. 
\begin{enumerate}
\item[{\rm (i)}] $\Gamma={\rm PSL}_2(\mathbb{F}_{q})\ast_{B(n,\frac{q-1}{2})}B(2\cdot n,\frac{q-1}{2})$,
 $p\not=2$.\\ For $q=5$, $N_0(\Gamma)=l.c.m.(60,50)=300$;
  $N_0(\Gamma)=\frac{(q^3-q)^2}{4}$ for $q>5$. 
\item[{\rm (ii)}]  $\Gamma={\rm PGL}_2(\mathbb{F}_{q})\ast_{B(n,q-1)}B(2\cdot n,q-1)$, $p=2$, 
         $q\geq 4$,  $N_0(\Gamma)=(q^3-q)^2$.
\item[{\rm (iii)}] $\Gamma={\rm PGL}_2(\mathbb{F}_{q})\ast_{B(n,q-1)}B(2\cdot n,q-1)$, $p\not=2$, 
           $q>3$, $ N_0(\Gamma)=\frac{(q^3-q)^2}{2}$.
\end{enumerate}
\end{theorem}

\begin{proof} 
For (i) and $q=5$, we refer to Table \ref{table-1} 
(even though the proof that follows is also valid for $q=5$). 
Assume now $q>5$. It
suffices to consider $\Delta \subset \Gamma$ which is a {\it maximal} normal Schottky subgroup of
finite index.\\

If $\Delta$ is not contained in a subgroup $\Gamma(B_a)\subset\Gamma$
for some $a\in \mathbb{F}_q$, then the quotient group $H=\Gamma /\Delta$ is simple.
$H$ contains an abelian group $B(2n,1)$ and $H$ is
non-abelian.  By \ref{6.29} one has $|H|\geq |{\rm PSL}_2(\mathbb{F}_{q^2})|=\frac{q^6-q^2}{2}$.\\

Suppose that  $\Delta$ is contained in $\Gamma(B_a)\subset\Gamma$ for
some $a\in \mathbb{F}_q$.
Then $|\Gamma(B_a)/\Delta|\geq |{\rm PSL}_2(\mathbb{F}_{q})|$.
Thus $|H|=|{\rm PSL}_2(\mathbb{F}_{q})|\cdot |\Gamma(B_a)/\Delta|\geq
|{\rm PSL}_2(\mathbb{F}_{q})|^2=\frac{(q^3-q)^2}{4}$.
This proves statement (i) of the proposition.\\

The proof of statement (ii) of the proposition is entirely similar.\\
Consider statement (iii). If $\Delta$ is not contained in the normal subgroup
$\Gamma^\circ:={\rm PSL}_2(\mathbb{F}_{q})\ast_{B(n,\frac{q-1}{2})}B(2\cdot n,\frac{q-1}{2})\subset\Gamma$,
then the quotient group $H=\Gamma/\Delta$ is a non-abelian simple group.
Since $B(2n,1)\subset H$, we conclude that 
$|H|\geq |{\rm PSL}_2(\mathbb{F}_{q^2})|=\frac{q^6-q^2}{2}$.

Let us now consider the case where the group $\Delta\subset\Gamma$
is contained in the subgroup $\Gamma^\circ\subset\Gamma$ of index two.
Then $|H|=2\cdot |\Gamma^\circ/\Delta|$.\\

If $q>5$, (iii) follows directly from (i).
Since $|\Gamma^\circ/\Delta|\geq \frac{(q^3-q)^2}{4}$ also holds for $q=5$ by lemma \ref{lemma-8.x},
statement (iii) is also valid for $q=5$.
\end{proof}

\begin{remarks}\label{CKK-2} \label{8.23} {\it Comparison with the
    results of {\rm \cite{C-K-K, C-K 2}}.} 
{\rm The cases of Mumford curves with large group of automorphisms
considered in \cite{C-K-K} are $B(n,q-1)\ast_{C_{q-1}}D_{q-1}$
(Proposition 3 and \S 9) and \\ ${\rm PGL}_2(\mathbb{F}_q)\ast_{B(n,q-1)}B(n\cdot
d,q-1)$ (Proposition 4 and \S 10). 

The first amalgam is extreme, but the chosen normal Schottky subgroup does not have minimal index. 
The second amalgam is not extreme. \\

The extreme cases of Theorem \ref{6.8} are counter examples for the Theorem of \cite{C-K-K}.
Indeed, they satisfy  $|{\rm Aut}(X)|=F(g)$, where $F$ is the function $F(g)=2\cdot g \cdot
(\sqrt{2g+1/4}+3/2)\approx 2\sqrt{2}\cdot g\sqrt{g}$. This exceeds the
upper limit $2\sqrt{g}(\sqrt{g}+1)^2$ proposed in \cite{C-K-K} by a factor $\sqrt{2}$. \\

We note that the extremal amalgams of Theorem \ref{6.8} with $q=3,4$, satisfy
$\mu(\Gamma)=\frac{1}{12}$, $|{\rm Aut}(X)|=12(g-1)=F(g)$,
$g=\frac{q^2-q}{2}$. This has the following consequences.
Since $S_4\cong {\rm PGL}_2(\mathbb{F}_3)$, the
amalgams in \cite{C-K 2}  theorem part (a) also exist for $q=3$.
Since $A_5\cong {\rm PGL}_2(\mathbb{F}_4)$ the three amalgams in
\cite{C-K 2}  theorem part (b) also exist for $p=2$.
 
Moreover, $A_5\cong {\rm PSL}_2(\mathbb{F}_5)$ 
and hence the prime $p=5$ is unjustly excluded in \cite{C-K 2} theorem part (b).
(See also Proposition \ref{8.1}, Lemma \ref{8.7} and remark \ref{8.8}.)\\

The existence of three Mumford curves of genus $g=6$ with automorphism group of order $72$ for $p=3$
implies that the curves in \cite{C-K 2} theorem part (b) are not truly maximal.
Some, but not all of these omissions have been corrected in the errata \cite{C}.\\

 We briefly discuss some errors in the proof of the theorem in \cite{C-K-K}.\\
In the proof of \cite{C-K-K} proposition 6.5 it is stated that if an element $h\in\Gamma$ of finite order acts 
without fixed points on $\Omega$, then the image $\varphi(h)$ acts without fixed points on the Mumford curve
$X=\Omega/\Delta$ of genus $g$. In particular, the order $m$ of the element $h$ would then divide $g-1$.

The amalgam  $\Gamma={\rm PGL}_2(\mathbb{F}_{q})\ast_{C_{q+1}}D_{q+1}$
provides a counter example to this statement.
The group $\Gamma_e=C_{q+1}$ acts without fixed points on $\Omega$.
However there exists a Mumford curve $X=\Omega/\Delta$ of genus $g=\frac{q^2-q}{2}$.
In particular, $g-1=\frac{(q-2)(q+1)}{2}$ is not divisible by $q+1$ if $q$ is odd.\\

The three exceptional curves $X$ for $q=3$ of genus $g=6$ and automorphism group ${\rm Aut}(X)=H_2$ of order
$72$ provides  counter examples to \cite{C-K-K} proposition 6.4.
In particular, the order of the quotient $\Gamma/\Delta=H_2$ is $72$
and is not divisible by $q^4=81$.\\

The existence of the Mumford curves $X=\Omega/\Delta$ of genus $\frac{q^2-q}{2}$
and automorphism group ${\rm PGL}_2(\mathbb{F}_{q})$ for the amalgam $\Gamma=
D_{q-1}\ast_{C_{q-1}}B(n,q-1)$ contradicts proposition 6.9 of \cite{C-K-K}.
}   
\end{remarks}

\section{ \rm The orbifold induced by a Mumford group}\label{section 9}
Orbifolds, differential equations and discontinuous groups
are developed in \cite{A} for a $p$-adic ground field. In this section we adopt some of
these ideas and adapt them to positive characteristic.  

An {\it orbifold } on the projective line $\mathbb{P}^1$ over $K$ is given by
a finite set of (singular) points $\{a_1,\dots ,a_s\}$ and for each point $a_j$ a
(finite, non trivial) Galois extension $L_j$ of the field of fractions of the completion
of the local ring at  $a_j$ (i.e., $K((z-a_j))$ if $a_j\neq \infty$). 
In contrast to the characteristic zero case, this Galois extension is
not determined by its degree. A {\it global orbifold covering} 
$X\rightarrow \mathbb{P}^1$ is a Galois covering ramified at
$\{a_1,\dots ,a_s\}$ and inducing the Galois extension $L_j$ for
every $j$.

  The question when a given orbifold on $\mathbb{P}^1$ admits a global
  orbifold covering is wide open. Even in the case where $s>2$ and all
  $L_j$ are tame extensions there seems to be no answer known.\\

 We observe that a Mumford group $\Gamma$ induces an orbifold with a global orbifold
covering by a Mumford curve. Indeed, $\Gamma$ has a normal subgroup
$\Delta$ of finite index such that $\Delta$ has no torsion. Then
$\Delta$ is a Schottky group and $X:=\Omega / \Delta$ (where $\Omega$
is the set of ordinary points for $\Gamma$) is a Mumford curve. The
canonical morphism $X\rightarrow \Omega /\Gamma =\mathbb{P}^1$ defines
an orbifold together with a global orbifold covering by $X$. \\
  
An interesting issue is the type of local Galois extensions $L_j$
induced by $\Gamma$.  Let $K((t))$ denote the field of fractions of 
the completion of the local ring of a point in $\Omega$ having a non
trivial fixed group $G\subset \Gamma$. If $p\nmid \#G$, then the local
parameter $t$ can be chosen such that $G$ consists of the maps 
$\{t\mapsto \zeta t|\ \zeta ^e=1\}$ and $p\nmid e$. Then, as usual 
$K((t))^G=K((z))$ with $z=t^e$ and $K((t))\supset K((z))$ is tamely
ramified. If $p|\#G$, then $G$ is a group of Borel type $B(n,m)$ with
$n>0$. There is a local parameter $t$ (in fact a global parameter of
the projective line) such that $G$ acts as 
$\{t^{-1}\mapsto \zeta t^{-1}+a|\ \zeta ^m=1,\ \ a\in A\}$ where
$A\subset K$ is a finite dimensional vector space over $\mathbb{F}_p$ 
such that $\zeta A=A$ for all $\zeta $ with $\zeta ^m=1$.

 Define $z=(\prod _{\zeta ^m=1,\ a\in A} (\zeta t^{-1}+a))^{-1}$. Then
 one sees that $K((t))^G=K((z))$. Consider $K((z))\subset
 K((t))^{B(n,1)}\subset K((t))$. Then the field in the middle is
 $K((z^{1/m}))$ and the extension $K((z^{1/m}))\subset K((t))$ is
 given by a set of Artin--Schreier extensions.\\

 The lists of realizable amalgams with two or three branch points
 yield many orbifolds on $\mathbb{P}^1$ with singular points
 $0,\infty$ or $0,1,\infty$.  We note that the type of an amalgam
  in the lists determines the tame part of the extension
  $K((z))\subset K((t))$ and the degree of the Artin--Schreier part of
  the extension. According to Remarks 1.1, the Artin--Schreier part
  (of fixed degree and invariant under the cyclic group
  $C_m$) can be prescribed arbitrarily by varying the inbedding of the given amalgam
  as discontinuous subgroup of ${\rm PGL}_2(K)$.

\subsection{\rm Tame orbifolds induced by a Mumford group}
Tameness of the orbifold means that the corresponding Mumford group $\Gamma$
has no elements of order $p$. From our lists it follows that
${\rm br}(\Gamma)>3$. It is easy to produce a list of Mumford groups
$\Gamma$ which induce a tame orbifold and have ${\rm br}(\Gamma)=4$ (by using
Theorem 3.1). Namely:\\
(i). The decomposable $\Gamma$'s are $C_\ell\ast C_m$ with $p\nmid \ell \cdot
m$ and have  ramification indices $(\ell, \ell, m, m)$.\\
(ii). The indecomposable $\Gamma$'s are $G_1\ast_{G_3}G_2$ with
$G_1,G_2\in \{D_\ell, A_4,S_4,A_5 \}$, $p$ does not divide the orders of $G_1$ and $G_2$
and $G_3$ is a branch group for both $G_1$ and $G_2$.  \\

In case (ii), one computes the ramification indices for these Mumford groups as follows.
The cyclic group $G_3$ has two branch points both stabilized by $G_3$ itself.
Moreover, $G_3$ is a branch group of both $G_1$ and $G_2$.
Let $x_1\in \mathbb{P}^1$ be the branch point of $G_1$ that is stabilized by $G_3$
and let $x_2\in \mathbb{P}^1$ be the branch point of $G_2$ that is stabilized by $G_3$.
Then the branch points for the group $G_1\ast_{G_3}G_2$ are those of $G_1$ minus the point $x_1$
combined with the branch points of $G_2$ minus the point $x_2$.
(See also the proof of theorem 5.3 in \cite{P-V}.)

Let $\ell$ be the order of the cyclic group $G_3$.
Let the ramification indices for $G_1$ and $G_2$ be $(n_1,n_2,\ell)$
and $(m_1,m_2,\ell)$, respectively.
Then the ramification indices for the amalgam $G_1\ast_{G_3}G_2$
is the tuple $(n_1,n_2,m_1,m_2)$. Moreover the four
branch points in $\mathbb{P}^1$ determine a reduction of
$\mathbb{P}^1$ consisting of two intersecting projective lines over
the residue field. The two branch points for $G_1$ map to one line and
those of $G_2$ to the other line.

\noindent

      \noindent
      \setlength{\unitlength}{1.3mm}
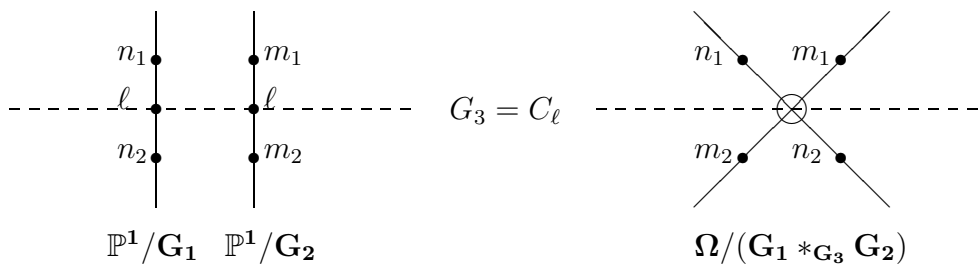
      \begin{picture}(00,30)

      \put(15,25){\line(0,-1){20}}
      \put(25,25){\line(0,-1){20}}
      \put(10,00){$\bf {\mathbb P}^1/G_1$}
      \put(22,00){$\bf {\mathbb P}^1/G_2$}

      \put(15,20){\circle*{1}}
      \put(15,15){\circle*{1}}
      \put(15,10){\circle*{1}}

      \put(25,20){\circle*{1}}
      \put(25,15){\circle*{1}}
      \put(25,10){\circle*{1}}

      \put(11,20){\bf $n_1$}
      \put(11,15){\bf $\ell$}
      \put(11,10){\bf $n_2$}

      \put(26,20){\bf $m_1$}
      \put(26,15){\bf $\ell$}
      \put(26,10){\bf $m_2$}

       \multiput(00,15)(2,0){21}{\line(1,0){1}}
       \put(45,14){\bf $G_3=C_\ell$}     
      
      \put(90,05){\line(-1,1){20}}%
      \put(70,05){\line(1,1){20}}
      \put(80,15){\circle {3}}
      \multiput(60,15)(2,0){20}{\line(1,0){1}}
      
      \put(75,20){\circle*{1}} 
      \put(85,10){\circle*{1}} 
      \put(85,20){\circle*{1}} 
      \put(75,10){\circle*{1}} 
      \put(70,20){\bf $n_1$}
      \put(80,10){\bf $n_2$}
      \put(80,20){\bf $m_1$}
      \put(70,10){\bf $m_2$}
      \put(70,00){$\bf \Omega/(G_1\ast_{G_3}G_2)$}
      
      \end{picture}

\begin{example}{\rm The group $\Gamma:=D_\ell*_{C_\ell}D_\ell$ with
$p\nmid 2\ell$ can be represented by the generators $\sigma
_1(z)=\zeta z$, where $\zeta$ is a primitive $\ell$th root of
unity, $\sigma _2(z)=\frac{1}{z}$ and $\sigma
_3(z)=\frac{\lambda }{z}$ with $0<|\lambda |<1$.
The first $D_\ell$ is generated by $\sigma_1,\sigma_2$ and the second
by $\sigma_1,\sigma _3$. The first $D_\ell$ has fixed points
$0,\infty$ and $\pm \sqrt{\zeta}^i$ for $i=0,\dots ,\ell -1$. For the
second $D_\ell$ the fixed points are $0,\infty$ and   $\pm \sqrt{\lambda \zeta}^i$ for $i=0,\dots ,\ell -1$.
The group $\Gamma$ is also generated by $\sigma_1,\sigma _2,\sigma _4$
where $\sigma_4 (z)=\lambda z$.  It follows that
$\{0,\infty\}$ is the set of limit points. Thus the branch points of
the two groups $D_\ell$ corresponding to $C_\ell$ disappear since they
are limit points. Thus $\Gamma$ has four branch points and they are in
the position described above.
}
\end{example}

It follows at once that all possible ramification tuples are $(2,2,a,b)$, $(2,3,c,d)$ and $(3,3,c,e)$
with $a,b\geq 2$, $c,e=3,4,5$ and $d\geq 3$ (with the restriction that
$p$ does not divide any ramification index).

Let now $G_1,G_2\in \{ A_4,S_4,A_5 \} $.
The ramification tuples $(2,2,a,b)$
correspond to the amalgams $G_1\ast_{C_\ell}G_2$, $\ell\not=2$,
$G_1\ast_{C_2}D_2$, $D_a\ast_{C_2}D_b$,  $D_m\ast_{C_m}D_m$. \\
The amalgams $G_1\ast_{C_2}D_d$ correspond to ramification tuples $(2,3,c,d)$.\\
The amalgams $G_1\ast_{C_2}G_2$ correspond to ramification tuples $(3,3,c,e)$.\\

The tuple $(2,2,2,2)$ corresponds to amalgams $\Gamma$ of the form $D_\ell\ast_{C_\ell}D_\ell$.
Then $\Omega\cong K^\ast$ and $\Omega/\Delta$ with $\Delta$ a normal torsion-free subgroup
of $\Gamma$ of finite index is a Tate curve and has genus $g=1$. 
Therefore the tuple $(2,2,2,3)$ is the smallest
one that gives rise to Mumford curves of genus $g>1$.
The Riemann-Hurwitz formula shows that
 Mumford curves $\Omega/\Delta$ corresponding to this index have an automorphism group of order $12(g-1)$
(See prop. \ref{6.14}).
For all other indices in our list this order is strictly less than $12(g-1)$.

\subsection{\rm Stratified bundles associated to a Mumford group}
We quickly  introduce the subject of stratified bundles in the
category of rigid spaces and sketch the
way Mumford groups may produce these bundles.

 Let $X$ be a smooth  rigid space over $K$ of countable type. On $X$
 there is a (rigid, quasi-coherent)  sheaf $\mathcal{D}_X$ of
 differential operators defined analogous to  the algebraic geometry
case (see \cite{EGA4}, \S 16, in particular 16.10 and also \cite{G} for the sheaf of differential
operators).

 For a  smooth affinoid space $Y={\rm Spm}(A)$ the
 sheaf $\mathcal{D}_Y$ is the (rigid) sheaf associated to the $A$-algebra
of differential operators on $A/K$ (as defined in
\cite{EGA4,G}). Further $\mathcal{D}_X$ is the rigid (quasi-coherent) sheaf obtained by
gluing the sheaves $\mathcal{D}_{Y_i}$ for an admissible affinoid covering $\{Y_i\}$ of $X$.  \\   

 We are interested in the case where $X$ has dimension 1 and
 especially in the case $\mathbb{P}^1_K$ and its open admissible
 subspaces. The basic example is the unit disk $U:={\rm Spm}(K\langle z\rangle)$. Its algebra of
 differential operators is  $K\langle z\rangle [ \{\partial _z^{(n)}\}_{n\geq 0}]$, where 
$\partial_z ^{(n)}$ is the operator on $K\langle z\rangle $ given by the formula 
$\partial_z^{(n)}(\sum _ja_jz^j)=\sum _ja_j{j\choose n }z^{j-n}$ (for
all $n\geq 0$ and we note that $\partial_z^{(0)}$ is the identity and is
identified with 1).  We note that the $\partial_z^{(n)}$ imitate the expressions $\frac{1}{n!}(\frac{d}{dz})^n$ which have only  a meaning
in characteristic zero. This $K\langle z\rangle $-module produces a quasi-coherent
sheaf on $U$ and yields for any affinoid subspace of $U$ an explicit
algebra of differential operators.\\

A stratified bundle $V$ on $X$ is a left $\mathcal{D}_X$-module on $X$
which is a vector bundle for its induced structure as $O_X$-module.
There is an extensive and interesting  theory of stratified bundles
in an algebraic context, see for instance \cite{G,E-M,Ki}. However this theory lacks explicit
examples. Here we construct examples of stratified bundles on
$\mathbb{P}^1_K$ having certain singularities, by using Mumford groups.\\

Let $\Gamma \subset {\rm PGL}_2(K)$ denote a Mumford group and let  $\Omega \subset \mathbb{P}^1(K)$ denote its subspace of ordinary
 points. The sheaf of differential operators $\mathcal{D}_\Omega$ has an obvious action of $\Gamma$.  Let $\pi :\Omega \rightarrow
 \mathbb{P}^1_K$ denote the canonical morphism. The sheaf $S:=(\pi
 _*\mathcal{D}_\Omega)^\Gamma$ is associated to the presheaf which maps every admissible open
 $U\subset \mathbb{P}^1_K$ to $\mathcal{D}_\Omega(\pi^{-1}(U))^\Gamma$. It can be seen that
 $S$ is a subsheaf of $\mathcal{D}_{\mathbb{P}^1_K}$ and that the two
 sheaves coincide outside the branch points. \\

At a branch point, the
 situation is somewhat complicated. Let $z=0$ be a branch point and
 let $t$ be the local parameter of a (ramification) point lying above
 $z=0$. The stalk of $\mathcal{D}_\Omega$ at that point is $K\{t\}[\{\partial
 _t^{(n)}\}_{n\geq 0}]$, where $K\{t\}$ denotes the local ring of the convergent
 power series. The action of $\Gamma$ reduces to the action      
 of the stabilizer $G$ and the stalk of $S$ at $z=0$ is then  
the algebra $(K\{t\}[\{\partial_t^{(n)}\}_{n\geq 0}])^G$.

 The stabilizer $G$ is, as before, $\{t^{-1}\mapsto \zeta t^{-1}+a|\ \zeta
^m=1,\ a\in A\}$. In principle, one can compute the algebra of
invariants under $G$. 

In the tame case, i.e., $A=0$, one has $z=t^m$
and the algebra of invariants is $K\{z\}[\{ t^{j(n)}\partial
_t^{(n)}\}_{n\geq 1}]$, where $j(n)\in \{0,1,\dots ,m-1\}$ satisfies
$n\equiv j(n) \mod m$. This is a subalgebra of $K\{z\}[\{\partial^{(n)}_z\}_{n\geq 0}]$.
In \S \ref{section 9.3} we do some explicit computations, show that
$K\{z\}[\{z^n\partial_z^{(n)}\}_{n\geq 0}]\subset (K\{t\}[\{\partial_t^{(n)}\}_{n\geq 0}])^G$
and give one example for the non tame case.\\

Let $\rho :\Gamma \rightarrow {\rm GL}(V)$ be a representation of
$\Gamma$ on a $d$-dimensional vector space $V$ over $K$. One
considers the (trivial) vector bundle  $\mathcal{O}_\Omega \otimes V$ on
$\Omega$ with left $\mathcal{D}_\Omega$-action through $\mathcal{O}_\Omega$ and
with $\Gamma$-action by $\gamma (f\otimes v)=\gamma(f)\otimes
\rho(\gamma)(v)$. Then $\mathcal{V}:=\pi_*(\mathcal{O}_\Omega \otimes V)^\Gamma$ is a vector bundle
on $\mathbb{P}^1_K$ (with rank equal to the dimension of $V$) which
has a left action by the subsheaf $S$ of
$\mathcal{D}_{\mathbb{P}^1_K}$.  By allowing singularities at the
branch points, this action extends to an action of $\mathcal{D}_{\mathbb{P}^1_K}$ with singularities. More explicitly: \\

Let $z=0$ be a branch point. In general, the stalk of $S$ at $z=0$ is a subalgebra of $K\{z\}[\{\partial_z^{(n)}\}_{n\geq 0}]$ and
has the form $K\{z\}[\{ z^{f(n)}\partial _z^{(n)}\}_{n\geq 1}]$ where $f(n)$ is the smallest
integer such that $z^{f(n)}\partial_z^{(n)}$ leaves $K\{t\}$ invariant.

The action of this stalk on the stalk of
$\mathcal{V}_0=K\{z\}^d$ introduces maps $\partial_z^{(n)}:\mathcal{V}_0\rightarrow
z^{-f(n)}\mathcal{V}_0$. Thus we obtain on $\mathcal{V}_0$ a structure of stratified bundle with singularities (poles).
 The singular point is called ``regular singular'' if $\mathcal{V}_0$ is invariant under all
$z^n\partial_z^{(n)}$. This is precisely the case when the branch point is tamely ramified (compare also \cite{Ki}). \\

Thus, we conclude that Mumford groups do not produce stratified
bundles with regular singularities at three points. However, according
to \S 7.1,
these groups produce many stratified bundles, regular singular at four points. \\

There is a {\it canonical stratified bundle of rank two} on
$\mathbb{P}^1_K$ associated to a Mumford group $\Gamma$ (see \cite{A},
Chapter II, \S 5, for the complex and the p-adic case). One considers
a representation $\rho :\Gamma \rightarrow {\rm SL}_2(K)$ which
induces the given embedding $\Gamma \subset {\rm PGL}_2(K)$. The
canonical stratification is associated to this representation $\rho$.\\

\subsection{\rm The higher derivations for branch
  points} \label{section 9.3}

{\it The tame case}.\\
Consider the tamely ramified extension $K(\{z\})\subset K(\{t\})$ (these are the fields of fractions of
$K\{z\}$ and $K\{t\}$) with $z=t^m$ and $p\nmid m$. 
We want to compute the extension of the standard higher derivation
$\{\partial_z^{(n)}\}_{n\geq 0}$ to $K(\{t\})$. An easy way is to
write this standard higher derivation as a $K$-linear homomorphism $\phi
:K(\{z\})\rightarrow K(\{z\})[[X]]$ given by $\phi(z)=z+X$. This $\phi$ extends to a 
homomorphism $\psi:K(\{t\})\rightarrow K(\{t\})[[X]]$. Now $\psi (t^m)=t^m+X=
t^m(1+t^{-m}X)$ implies 
$\psi (t)=t(1+t^{-m}X)^{1/m}=t(\sum _{n=0}^\infty {\frac{1}{m}\choose
n} t^{-mn}X^n)$.
Hence $\partial _z^{(n)}(t)=t\cdot {\frac{1}{m}\choose n }t^{-mn}$ for all $n\geq 0$ and
 $z^n\partial _z^{(n)}(t^i)={\frac{i}{m}\choose n }t^i$ for  $i=0,1,\dots ,m-1$. 
 This shows that $z^n\partial^{(n)}$ lies in $(K\{t\}[\{\partial_t^{(n)}\}])^G$.\\
 
 We note that $M:=K(\{t\})$ can be seen as a stratified bundle over $K(\{z\})$. It is {\it regular singular}.
 The elements $\{t^i| i=0,1,\dots ,m-1\}$ form a basis and the above formula shows
 that the  local exponents are $\{ \frac{i}{m}|\ i=0,1,\dots ,m-1\}$.\\

\noindent {\it The Artin-Schreier case}. \\
We consider the basic example $K(\{z\})\subset  K(\{t\})$ with
$t^{-p}-t^{-1}=z^{-1}$. Now $\phi
:K(\{z\})\rightarrow K(\{z\})[[X]]$ with $\phi(z)=z+X$ extends to a $\psi$
with $\psi(t^{-1})=t^{-1}+R$ with $R\in XK(\{t\})[[X]]$ and $\psi
(t^{-p})=t^{-p}+R^p$. Further $\psi (t^{-p})-\psi
(t^{-1})=\frac{1}{z+X}$. Hence
$R^p-R=\frac{1}{z+X}-\frac{1}{z}=\frac{-X}{z(z+X)}$. Thus 
\[R=\sum _{n\geq 0}(\frac{X}{z(z+X)})^{p^n}=\sum _{n\geq  0}\frac{X^{p^n}}{z^{p^n}(z^{p^n}+X^{p^n})} .\]
Write $X=z^2Y$, then $R=\sum _{n\geq 0,\ k\geq
  0}(-1)^{k}z^{kp^n}Y^{kp^n+p^n}$.
  
   Now $\psi(t)=\frac{1}{t^{-1}+R}=\frac{t}{1+tR}=
t(\sum _{\ell \geq 0}(-t)^\ell R^\ell)$ and $\partial
_z^{(n)}(t)$ has the form:\\ 
$\pm z^{-2n}t^{1+n}(1+r)$ with $r\in
tK\{t\}$.  Then $z^n\partial^{(n)}_z(t)=\pm z^{-n}t^{1+n}(1+r)$. The smallest integer $f(n)$, such that
$z^{f(n)}\partial_z^{(n)}$ leaves $K\{t\}$ invariant, is $>\frac{3n}{2}$
 and the singularity is {\it irregular}. Further  $M=K(\{t\})$ as stratified bundle over 
 $K(\{z\})$ is {\it  irregular singular}.
 
\begin{small}

\end{small}

\end{document}